\documentclass[12pt]{article}

\usepackage[english]{babel}
\usepackage{titlesec}
\usepackage{geometry}
\usepackage{enumitem}
\usepackage[utf8]{inputenc}
\usepackage{amsmath}
\usepackage{amssymb}
\usepackage{amsthm}
\usepackage[mathcal]{euscript}
\usepackage{color}
\usepackage{tikz-cd}
\usepackage{mathabx}
\usepackage{physics}
\usepackage{MnSymbol}
\usepackage{csquotes}
\usepackage{url}
\usepackage{caption}

\usepackage{xparse}
\let\realItem\item 
\makeatletter
\NewDocumentCommand\myItem{ o }{%
   \IfNoValueTF{#1}%
      {\realItem}
      {\realItem[#1]\def\@currentlabel{#1}}
}
\makeatother

\usepackage{enumitem}    
\setlist[enumerate]{
    before=\let\item\myItem,       
    label=\textnormal{(\arabic*)}, 
}

\numberwithin{equation}{section}

\title{Semi-deterministic processes with applications in random billiards}
\author{Peter Rudzis\thanks{University of North Carolina - Chapel Hill. Email: prudzis@unc.edu. Research supported in part by  the RTG award (DMS-2134107) from the NSF.}}
\date{December 30th, 2023}

\newtheorem{theorem}{Theorem}[section]
\newtheorem{lemma}[theorem]{Lemma}
\newtheorem{corollary}[theorem]{Corollary}
\newtheorem{proposition}[theorem]{Proposition}
\newtheorem{remark}[theorem]{Remark}

\DeclareMathOperator{\sgn}{sgn}

\DeclareMathOperator{\MKE}{MKE}
\DeclareMathOperator{\MKI}{MKI}
\DeclareMathOperator{\MK}{MK}
\DeclareMathOperator{\id}{id}

\begin{document}

\maketitle

\begin{abstract}
We study the ergodic properties of two classes of random dynamical systems: a type of Markov chain which we call the \textit{alternating random walk} and a certain stochastic billiard system which describes the motion of a free-moving rough disk bouncing between two parallel rough walls. Our main results characterize the types of Markov transition kernels which make each system ergodic -- in the first case, with respect to uniform measure on the state space, and in the second case, with respect to Lambertian measure (a classic measure from geometric optics). In addition, building on results from \cite{rudzis2022}, we give explicit examples of rough microstructures which produce ergodic dynamics in the second system. Both systems have the property that the transition kernel governing the dynamics is singular with respect to uniform measure on the state space. As a result, these systems occupy a kind of mean in the problem space between diffusive processes, where establishing ergodicity is relatively easy, and physically realistic deterministic systems, where questions of ergodicity are far less approachable.
\end{abstract}

\section{Introduction}

Establishing ergodicity of deterministic physical systems is a notoriously difficult problem \cite{simanyi-sinai}. Yet the same task is often easy if we overlay a sufficient amount of random noise on the dynamics. As a case in point, consider the behavior, on the one hand, of a billiard particle in some complicated spatial domain and, on the other hand, of a diffusion with reflecting boundary condition in the same domain. This observation motivates the study of systems with noisy dynamics such that, on each step of the evolution, the noise is concentrated in some submanifold of the phase space whose co-dimension is strictly greater than zero. We refer to such a dynamical system as a \textit{semi-deterministic process}. We propose that, between systems where ergodicity is trivial to prove and systems where the same problem is relatively intractable, semi-deterministic processes occupy a virtuous middle ground. 

This paper focuses on two closely related random dynamical systems: a Markov chain which we refer to as an \textit{alternating random walk}, and a stochastic billiard system, which describes the motion of a free-moving rough disk bouncing between two parallel rough walls. More precisely, each of these examples is a \textit{class} of dynamical systems in its own right, since in each case the choice of the Markov kernel which determines the evolution of the system is free up to the satisfaction of certain constraints. The main results of this paper characterize the choices of Markov kernels which make each system ergodic.

In both systems, the Markov kernel acts on a two-dimensional state space but is supported in a one-dimensional submanifold of the phase space which depends on the current state of the process. One should not think that this feature is artificially imposed on our model. Quite the contrary, in the billiard system which we consider, the feature arises naturally as a result of the presence of locally conserved quantities. Indeed, one may view our work as a followup to the monograph \cite{rudzis2022} which justifies the presence of these conserved quantities by showing that they arise from a ``physically realistic'' limiting procedure. We discuss this physical interpretation further below and in \S\ref{sec:billiards}.

We regard this paper as a ``proof of concept,'' supporting the notion that semi-deterministic processes inhabit a propitious sector of the dynamical systems problem space, deserving further exploration. The arguments involved in establishing our main results, stated in \S\ref{ssec:mainresults} and proved in \S\ref{sec:ergproofs}, are (we believe) fairly elegant, involving nontrivial mathematics, but avoiding heavy amounts of technical baggage. The proofs start out geometric, establishing a number of symmetry conditions which the dynamical systems under consideration must satisfy; appeal to established results from ergodic theory; and conclude somewhat unexpectedly with functional analysis, bringing crucially to bear the injectivity of the classic Abel integral transform on $L^1$. The results should be amenable to generalization -- see \S\ref{ssec:future} for further discussion.

\subsection{Physical motivation} \label{ssec:phys}

Consider a particle of gas moving down a long pipe (modeled, say, as a circular cylinder of radius $R > 1$). We model the particle as a rigid ball of radius 1, which moves linearly in the interior of the pipe and collides with the boundary according to some collision law which conserves kinetic energy. If the collision dynamics are such that direction of impact is normal to the surface of the wall, the evolution of the system is not hard to understand. We can produce more interesting (and presumably more physically realistic) dynamics by introducing small rough features on the surfaces of the disk and wall. This results in collision dynamics which are noisy and which mix linear and angular components of energy.

Restricting to the two-dimensional case, the pipe is modeled as the strip $\{(x_1,x_2) : |x_2| \leq R\}$ where $R > 1$, and the particle is modeled as a disk of radius 1, which moves freely in the interior of the strip and collides with the boundary according to some prescribed (and in general random) dynamical rule. The dynamical evolution of the system is then completely determined by the initial configuration of the disk and the sequence of post-collision angular and linear velocities. 

We assume that the disk has a rotationally symmetric mass density, and we denote the total mass of the disk by $m$ and the moment of inertia of the disk about its center by $J$. We define an inner product on $\mathbb{R}^3$ by 
\begin{equation} \label{eq:kinnerproduct}
\langle v, v' \rangle = \frac{1}{2}Jv_0v_0' + \frac{1}{2}m(v_1v_1' + v_2v_2'), \text{ where } v = (v_0, v_1, v_2), v' = (v_0', v_1', v_2'),
\end{equation}
and we denote the corresponding norm by $\| \cdot \|$.
If velocity takes the form $v = (v_0,v_1,v_2) \in \mathbb{R}^3$, where $v_0$ is the counter-clockwise angular velocity of the disk about its center, and $v_1$ and $v_2$ are the linear velocity components of the center of the disk, then the kinetic energy of the disk is given by $\| v\|^2$. We assume that kinetic energy of the disk is equal to 1 for all time (thus the collisions do not dissipate energy). 

Physically, the collision dynamics are supposed to depend on the asperities (microscopic rough features) on the surfaces of each body. For simplicity we suppose that the roughness on each body is spatially homogeneous, so that the post-collision velocity depends only on the pre-collision velocity (i.e. the previous post-collision velocity). We exclude the set of tangential velocities (velocities $v$ such that $v_2 = 0$) from the set of allowable velocities, since no collision can take place in this situation. We denote the kinetic energy unit 2-sphere by
\begin{equation} \label{eq:kinspheredef}
\mathbb{S}^2 := \{v \in \mathbb{R}^3 : \|v\|^2 = 1\} = \{(v_0,v_1,v_2) : 2^{-1}Jv_0^2 + 2^{-1}m(v_1^2 + v_2^2) = 1\}. 
\end{equation}
We also let
$$
\mathbb{S}^2_\pm := \{(v_0,v_1,v_2) \in \mathbb{S}^2 : \pm v_2 > 0\}.
$$
The sequence of post-collision velocities may be viewed as a Markov process on 
$$
\hat{\mathbb{S}}^2 := \{(v_0,v_1,v_2) \in \mathbb{S}^2 : v_2 \neq 0\}.
$$
We denote the Markov kernel on $\hat{\mathbb{S}}^2$ which determines the velocity process by $K(v, dv')$. Since the component $v_2$ of the post-collision velocity will alternate sign following each collision, we may assume that $K$ takes the form
\begin{equation} \label{eq:twosides}
K(v,dv') = \begin{cases}
K_+(-v,dv') \quad \text{ if } v_2 < 0 \\
R_* K_+(-v, dv') \text{ if } v_2 > 0,
\end{cases}
\end{equation}
where $R : \mathbb{S}^2_+ \to \mathbb{S}^2_-$ is the restriction of the 180 degree rotation about the $v_0$-axis in $\mathbb{R}^3$ to the unit hemispheres $\mathbb{S}^2_\pm$, and $K_+$ is some Markov kernel on $\mathbb{S}^2_+$, which determines the local collision behavior when the disk hits the wall. We refer to $K_+$ as the \textit{collision law} associated with the system. Note that $K_+$ takes the \textit{negative} of the pre-collision velocity to the post-collision velocity, ensuring that $K_+$ acts on the single space $\mathbb{S}^2_+$.

Since the collision law $K_+$ is supposed to capture the dynamical effect of the asperities, it should belong to a certain class $\mathcal{A}$ of ``physically realizable'' collision laws. This notion is given a precise formulation in the monograph \cite{rudzis2022}, which we also review in \S\ref{ssec:rough}. To sketch the main idea, imagine that we equip each body with rigid, $\epsilon$-scale geometric microstructures. The collision behavior depends on the microstructures we choose. Each element of $\mathcal{A}$ is, by definition, expressible as a certain limit as $\epsilon \to 0$ of a sequence of deterministic collision laws, each of which describes the collision dynamics when the $\epsilon$-scale rough features on the bodies interact through the classical laws of rigid body mechanics. The question addressed in \cite{rudzis2022} is the following:

\vspace{0.1in}

\textit{Question:} What is a mathematically convenient characterization of the class $\mathcal{A}$?

\vspace{0.1in}

\noindent The main result of \cite{rudzis2022} answers that $K_+$ lies in $\mathcal{A}$ if and only if it has the coordinate representation
\begin{equation} \label{eq:roughdecomp}
K_+(\theta,\psi; d\theta' d\psi') = P(\theta, d\theta') \delta_{\pi -\psi}(d\psi'), \quad (\theta,\psi) \in (0,\pi) \times (0,\pi),
\end{equation}
where $(\theta,\psi)$ is a particular choice of spherical coordinates restricted to the hemisphere $\mathbb{S}^2_+$ (see \S\ref{ssec:ergodic}), and $P$ is a Markov kernel on the interval $(0,\pi)$ which satisfies the following property:

\vspace{0.1in}

(P) $P$ is reversible with respect to the probability measure $\Lambda^1(d\theta) := \frac{1}{2}\sin\theta d\theta$, $ \theta \in (0,\pi)$.

\vspace{0.1in}

\noindent Moreover, \cite{rudzis2022} shows that the Markov kernel $P$ may be interpreted as a ``reflection law,'' describing reflections of point particles from a foreshortened version of the rough microstructure which gives rise to the Markov kernel $K_+$. (More precisely, see Theorem \ref{thm:colchar}). This fact is important, as it allows us to compute explicitly the collision law $K_+$ for many choices of microstructures. Applying the main results of this paper in \S\ref{ssec:examples}, we identify many examples of microstructures on the walls of the strip which give rise to ergodic dynamics for the velocity process described above. 

The physical significance of the product decomposition \eqref{eq:roughdecomp} is as follows. The ``north pole'' of the spherical coordinates $(\theta,\psi)$ coincides with the unit vector parallel to $\chi_0 := (1,-1,0)$. The trivial second factor in \eqref{eq:roughdecomp} means that collisions with the lower wall of the strip conserve the quantity 
\begin{equation} \label{eq:conserved1}
\langle v,\chi_0\rangle = Jv_0 - mv_1,
\end{equation} 
and collisions with the upper wall conserve the quantity 
\begin{equation} \label{eq:conserved2}
\langle v, R(\chi_0) \rangle = Jv_0 + mv_1.
\end{equation}
The presence of these conserved quantities has, in turn, a simple heuristic explanation. If the disk moves with velocity approximately parallel to $\chi_0$ when it collides with the lower wall of the strip, then the disk will ``roll'' along the lower boundary. Consequently, the impact of the disk with the wall will be negligible and the disk will continue to roll indefinitely. This suggests that translation in the direction $\chi_0$ is a symmetry of the configuration space, and therefore, by a loose application of Noether's theorem, the quantity \eqref{eq:conserved1} is conserved. Similar reasoning suggests that collisions with the upper boundary conserve \eqref{eq:conserved2}. Making all this rigorous, of course, takes much more work.

\subsection{Related work} \label{ssec:relatedwork}

The billiards results in this paper relate most closely to the work on Knudsen gas dynamics initiated by R. Feres and G. Yablonsky in \cite{feresyab2004} and continued in \cite{feres2007RW}, analyzing the dynamics of a point particle moving in a long cylindrical chamber with rough walls. These and other authors investigate ergodicity, convergence to stationarity, and the spectral analysis of such systems in \cite{fereszhang2010, fereszhang2012, chumleyferesgerman2021, CFGY2023}. While the scope of these papers is mainly restricted to point particle dynamics, our paper considers particles which are extended rough bodies. This distinction leads to the main technical novelty in our work, namely the presence of singular transition kernels for the dynamics. The question of ergodicity in our case is, consequently, much more delicate. 

As mentioned above, the main results of \cite{rudzis2022}, classifying the types of collision kernels which can be obtained through a certain ``physically realistic'' limiting procedure (see \S\ref{sec:billiards}), motivate the collision dynamics considered in this paper. A. Plakhov previously obtained related results for the point particle case, with the goal of bringing to bear techniques of optimal mass transport to problems of air resistance minimization and invisibility \cite{plakhov2004newtonprob, plakhov2009billiards, plakhov2012}. O. Angel, K. Burdzy, and S. Sheffield obtain a similar classification result in \cite{ABS2013detapprox}.

The alternating random walk which we introduce in \S\ref{sec:mainresults} resembles at least on a broad, conceptual level the coordinate hit-and-run algorithm for sampling from uniform measure on convex bodies \cite{andersondiaconis2007hitnrun}. While the transition kernel for a typical implementation of the coordinate hit-and-run algorithm samples uniformly from chords of the convex body parallel to the coordinate axes, our transition kernels are of general type, only being required to be reversible with respect to uniform measure on the chords. Thus the kinds of questions we ask are much different.

\subsection{Organization of the paper}

In \S\ref{sec:mainresults}, we describe the models and state the main results of our paper, first introducing the alternating random walk in \S\ref{ssec:altwalk}-\S\ref{ssec:mainresults}, and then describing the disk-in-strip billiard system in \S\ref{ssec:ergodic}. 

For the purpose of proving our main results, the coordinate representation \eqref{eq:roughdecomp} for the collision law $K_+$ can be taken as axiomatic. Nonetheless, in \S\ref{ssec:prelim}-\S\ref{ssec:rough}, we describe the precise sense in which collision laws having this representation can be obtained through a ``physically realistic'' limiting procedure, reviewing the main results of \cite{rudzis2022}. This allows us to describe, in \S\ref{ssec:examples}, a number of examples of explicit rough microstructures on the surfaces of the disk and the walls which give rise to ergodic dynamics in the disk-in-strip system. 

In \S\ref{ssec:isoproof}, we show that the alternating random walk and the rough disk velocity process are isomorphic as dynamical systems. In \S\ref{ssec:refnub}, we prove that a certain modification procedure applied to the microstructure, which we refer to as \textit{adding $\delta$-nubs}, results in a microstructure that produces ergodic dynamics, which on a short time-scale are ``close'' to the dynamics arising from the original microstructure. 

Lastly, we prove our main results on the ergodicity of the alternating random walk in \S\ref{sec:ergproofs}.

\subsection{Acknowledgements}

I would like to thank Krzysztof Burdzy for many insightful discussions. A number of the main ideas of this paper first took form under his supervision, during my PhD studies at the University of Washington.

\subsection{Notation}

If $(S,\mathcal{S})$ is a measurable space, then $\MK(S,\mathcal{S})$ denotes the set of all Markov kernels on $(S,\mathcal{S})$. That is, $\MK(S,\mathcal{S})$ is the set of all mappings $p : S \times \mathcal{S} \to \mathbb{R}$ such that (i) for all $A \in \mathcal{S}$, $p(\cdot,A)$ is a measurable map from $(S,\mathcal{S})$ into $\mathbb{R}$, equipped with the Borel $\sigma$-algebra, and (ii) for all $s \in S$, $p(s,\cdot)$ is a measure on $(S,\mathcal{S})$. If $S$ is a topological space, we may suppress the $\mathcal{S}$ from our notation, implicitly understanding that $\mathcal{S}$ is the Borel $\sigma$-algebra.

If $\mu$ is a measure on a measure space $X$, and $T : X \to Y$ is a measurable surjection, then $T_*\mu$ is the measure on $Y$ defined by $T_*\mu(B) = \mu(T^{-1}(B))$ for all measurable $B \subset Y$. If $k(x,dx')$ is Markov kernel on $X$, and moreover $T$ is bijective, then $T_*k(y, dy')$ is the Markov kernel on $Y$ defined by $T_*k(y, B) = k(T^{-1}(y), T^{-1}(B))$ for all $y \in Y$ and measurable $B \subset Y$.

If $\mu_1$ and $\mu_2$ are two measures on a measurable space $(S,\mathcal{S})$, the notation $\mu_1 \ll \mu_2$ means that $\mu_1$ is absolutely continuous with respect to $\mu_2$, and the notation $\mu_1 \gg \mu_2$ means that $\mu_2 \ll \mu_1$. In addition, we write $\mu_1 \perp \mu_2$ to indicate that $\mu_1$ and $\mu_2$ are mutually singular.

If $F$ is a Lebesgue measurable subset of $\mathbb{R}^d$, then $|F|$ denotes the $d$-dimensional Lebesgue measure of $F$. We may also write this as $|F|_{\mathbb{R}^d}$ if the dimension is not clear from context. In addition, $|\cdot|_F$ denotes the measure on $\mathbb{R}^d$ defined by $|G|_F = |G \cap F|$ for measurable subsets $G \subset \mathbb{R}^d$.

If $U \subset \mathbb{R}^d$, then the topological boundary, closure and interior of $U$, relative to the usual topology on $\mathbb{R}^d$, are denoted by $\partial U$, $\overline{U}$, and $\text{Int} U$, respectively.

\section{Main Results} \label{sec:mainresults}

\subsection{An alternating random walk in an elliptical region} \label{ssec:altwalk}

Let $q(u,du')$ be a Markov kernel on the interval $(-1,1)$. Consider the open elliptical region of the plane
\begin{equation} \label{eq:Edef}
E = \{(u,v) \in \mathbb{R}^2 : u^2 + (2\cos\gamma)uv + v^2 < 1\},
\end{equation}
where $\gamma \in (0,\pi)$. One axis of the ellipse $E$ lies along the diagonal $u = v$ in $\mathbb{R}^2$. The lengths of the major and minor axes are given, respectively, by 
$$
a, b = \csc\gamma \sqrt{2 \pm 2\cos \gamma}.
$$
Given $-\csc \gamma < u < \csc \gamma$, we define the interval
\begin{align}
E_u = \left\{v : -u \cos \gamma - \sqrt{1 - u^2 \sin^2 \gamma} < v < -u \cos \gamma + \sqrt{1 - u^2 \sin^2 \gamma}\right\}. 
\end{align}
For any point $P_0 = (u_0,v_0) \in E$, $E_{v_0} \times \{v_0\}$ is the horizontal chord of $E$ passing through $P_0$ and $\{u_0\} \times E_{u_0}$ is the vertical chord passing through $P_0$. 

Let $E^* := E \times \{\pm 1\}$, equipped with the product topology (the Borel topology on $E$ times the discrete topology on $\{\pm 1\}$). We define a Markov process on the state space $E^*$ as follows. Let $\ell_u : (-1,1) \to E_u$ be the linear increasing function which maps $(-1,1)$ onto $E_u$; explicitly, 
\begin{equation} \label{eq:Sform}
\ell_u(x) = -u\cos\gamma + x\sqrt{1 - u^2 \sin^2 \gamma}, \text{ for } x \in (-1,1).
\end{equation}
We let 
\begin{equation} \label{eq:hatqdef}
\hat{q}(x,dx') := q(-x,dx'), \text{ for } x \in (-1,1).
\end{equation}
We define $\{X_i, i \geq 0\}$ to be the Markov process on $E^*$ determined by the Markov kernel
\begin{equation} \label{eq:Qdef}
Q(u,v,s; du' dv' ds') = \begin{cases}

(\ell_v)_*\hat{q}(u, du') \otimes \delta_v(dv') \otimes \delta_{-s}(ds') & \text{ if } s = -1, \\
\delta_{u}(du') \otimes (\ell_u)_*\hat{q}(v, dv') \otimes \delta_{-s}(ds') & \text{ if } s = 1.
\end{cases}
\end{equation}

Put in words, the process we wish to describe by \eqref{eq:Qdef} is a random walk which alternates on each step between stepping along horizontal and vertical chords of $E$. If we identify the chord with $(-1,1)$ through an increasing linear rescaling, then law governing the next step of the process along that chord is given by $\hat{q}$. To make this into a Markov process, we introduce the extra coordinate $s \in \{\pm 1\}$ to keep track of whether the next step will be in the horizontal ($s = -1$) or the vertical ($s = 1$) direction. See also Figure \ref{fig:alternating}.

\begin{figure}
    \centering
    \includegraphics[width = 0.45\linewidth]{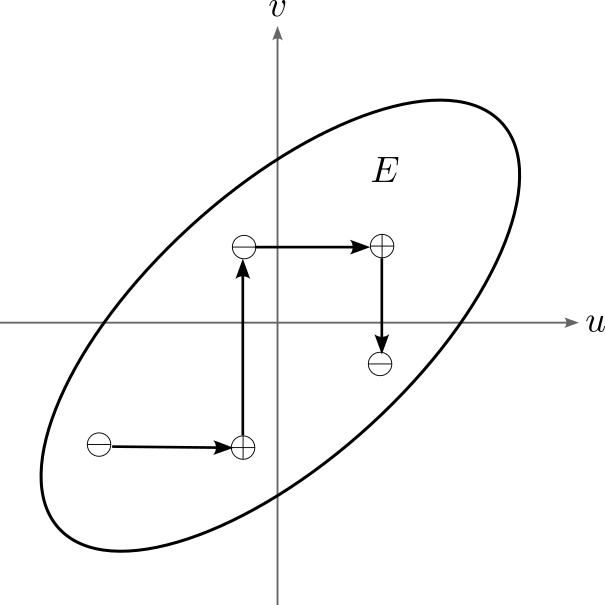}
    \caption{The alternating random walk alternates between taking steps along vertical and horizontal chords of the ellipse $E$. We make the process into a Markov process by introducing an extra ``sign'' variable $s \in \{\pm1\}$ to keep track of whether the next step will be in the horizontal or vertical direction.}
    \label{fig:alternating}
\end{figure}

\subsection{Invariant, reversible, and $\dag$-reversible measures}

We denote uniform measure on the interval $(-1,1)$ by $m^1(du)$ and we denote uniform measure on the state space $E^*$ by $m^2(du dv ds)$. That is, if $|\cdot|$ denotes Lebesgue measure, then $m^1(dx) = \frac{1}{2}\mathbf{1}_{(-1,1)}(x)|dx|$, and 
$$
m^2(du dv ds) = \frac{\mathbf{1}_E(u,v)}{2|E|}\left[|du dv| \otimes\delta_{1}(ds) + |du dv| \otimes\delta_{-1}(ds)\right].
$$

Suppose $m^1$ is a reversible measure for $q$, i.e. for any bounded and continuous functions $j,k$ on $(-1,1)$,
$$
\int_{(-1,1) \times (-1,1)} k(x)j(x')p(x,dx')m^1(dx) = \int_{(-1,1) \times (-1,1)} j(x)k(x')p^\dag(x,dx')m^1(dx).
$$
Then $\hat{q}$ and $Q$ satisfy a kind of time-reversal symmetry. More precisely, let $\mathcal{N} : (-1,1) \to (-1,1)$ be given by $\mathcal{N}(x) = -x$. If $p$ is any Markov kernel on $(-1,1)$, we let
$$
p^\dag(x, dx') = \mathcal{N}_*p(x, dx').
$$
We also let 
\begin{equation} \label{eq:Qtdef}
Q^\dag(u,v,s; du' dv' ds') = \begin{cases}

(\ell_v)_*\hat{q}^\dag(u, du') \otimes \delta_v(dv') \otimes \delta_{-s}(ds') & \text{ if } s = 1, \\
\delta_{u}(du') \otimes (\ell_u)_*\hat{q}^\dag(v, dv') \otimes \delta_{-s}(ds') & \text{ if } s = -1.
\end{cases}
\end{equation}
We say that $m^1$ is a \textit{$\dag$-reversible} measure for $p$ if for any bounded and continuous functions $j, k$ on $(-1,1)$, 
$$
\int_{(-1,1) \times (-1,1)} k(x)j(x')p(x,dx')m^1(dx) = \int_{(-1,1) \times (-1,1)} j(x)k(x')p^\dag(x,dx')m^1(dx).
$$
Similarly, we say that $m^2$ is a \textit{$\dag$-reversible} measure for $Q$ if for any bounded and continuous functions $f,g$ on $E^*$, 
\[
\begin{split}
& \int_{E^* \times E^*} g(u,v,s)f(u',v',s')Q(u,v,s;du' dv' ds')m^2(du dv ds) \\
& = \int_{E^* \times E^*} f(u,v,s)g(u',v',s')Q^\dag(u,v,s;du' dv' ds')m^2(du dv ds).
\end{split}
\]
Note that $\dag$-reversibility implies in both cases that $m^1$ and $m^2$, respectively, are invariant measures. In fact, we have the following.

\begin{proposition} \label{prop:ir}
The following statements are equivalent.

(i) $m^1$ is an invariant (resp. reversible) measure for $q$.

(ii) $m^1$ is an invariant (resp. $\dag$-reversible) measure for $\hat{q}$.

(iii) $m^2$ is an invariant (resp. $\dag$-reversible) measure for $Q$.
\end{proposition}

\noindent We prove this proposition in \S\ref{ssec:propsproofs}.

\subsection{Ergodicity of the alternating random walk} \label{ssec:mainresults}

The first of our main results characterizes the choices of $q$ which make $m^2$ is an ergodic measure for $Q$.

\begin{theorem} \label{thm:mainresult}
Assume that
\begin{enumerate}
\item[(i)] $m^1$ is a reversible measure for $q$, \label{cond:main1} 
\item[(ii)] $\gamma/\pi$ is irrational, and \label{cond:main2}
\item[(iii)] there exists a measure $\mu$ on $E^*$ which is ergodic for $Q$ and which is nonsingular with respect to $m^2$. \label{cond:main3}
\end{enumerate}
If $m^1$ is an ergodic measure for $q$, then $m^2$ is an ergodic measure for $Q$.

Moreover, if $\hat{q}^\dag = \hat{q}$, then the converse holds. More precisely if (i) holds and $m^2$ is an ergodic measure for $Q$, then $m^1$ is an ergodic measure for $q$.
\end{theorem}

Condition (iii) is not easy to verify directly. An equivalent statement is that there exists some $(m^2,Q)$-ergodic subset $F \subset E^*$ such that $m^2(F) > 0$, but even this may be hard to check. (See \S\ref{sssec:dual} for the definition of an ergodic subset.) The next theorem gives us a more convenient condition implying (iii).

\begin{theorem} \label{thm:qNS}
(i) If $m^1$ is a reversible measure for $q$, and for almost every $x \in (-1,1)$, $q(x, dx')$ is nonsingular with respect to $m^1(dx')$, then there exists a measure $\mu$ on $E^*$ which is ergodic for $Q$ and which is absolutely continuous with respect to $m^2$.

(ii) Moreover, if $q(x,dx')$ is nonsingular with respect to $m^1(dx')$ for all $x \in (-1,1)$, then every invariant measure for $Q$ is nonsingular with respect to $m^2$.
\end{theorem}

Note that the hypothesis of Theorem \ref{thm:mainresult} is stated in terms of $q$, but in \eqref{eq:Qdef} the Markov kernel $Q$ is defined in terms of $\hat{q}(x,dx') = q(-x,dx')$.
Although the reason for our introducing two different Markov kernels may right now seem opaque, when we state the corresponding result for the disk-in-strip system, the analogues of $q$ and $\hat{q}$ will have straightforward physical interpretations. See Remark \ref{rem:signs}.

In Theorem \ref{thm:mainresult}, we do not know if assumption (ii) can be weakened. However, the following simple counterexample shows that we cannot allow $\gamma = \pi/2$, corresponding to the case where $E$ is the open unit disk. Let $q$ be the Markov kernel on $(-1,1)$ defined by  
$$
q(x,dx') = \begin{cases}
|dx'|_{(0,1)} & \text{ if } x \in (-1,0), \\
|dx'|_{(-1,0)} & \text{ if } x \in (0,1), \\
m^1(dx') & \text{ if } x = 0.
\end{cases}
$$
One may check that $m^1$ is a reversible, ergodic measure for $q$. We observe that 
$$
\hat{q}(x,dx') = \begin{cases}
|dx'|_{(-1,0)} & \text{ if } x \in (-1,0), \\
|dx'|_{(0,1)} & \text{ if } x \in (0,1), \\
m^1(dx') & \text{ if } x = 0.
\end{cases}
$$
Let $R_1, R_2, R_3, R_4$ denote the restrictions of the open unit disk to the first, second, third, and fourth quadrants, respectively, of $\mathbb{R}^2$. Then the regions $R_i \times \{-1,1\} \subset E^*$, for $i \in \{1,2,3,4\}$, are disjoint invariant sets for the Markov chain determined by $Q$, and thus $m^2$ is not ergodic for $Q$. On the other hand, one may also show that the restrictions of $m^2$ to the regions $R_i \times \{-1,1\}$ are ergodic measures for $Q$.

The question of whether condition (iii) can be weakened is likewise open. Conditions (ii) and (iii) interact in an important way in the proof of Theorem \ref{thm:mainresult}, coming to together in an application of the following fact: Any non-null subset of the circle which is invariant under an irrational rotation must be a full-measure subset of the circle. This suggests that any weakening of both conditions (ii) and (iii) in Theorem \ref{thm:mainresult} would require a considerably different approach than ours.

Note that by Theorem \ref{thm:qNS}(ii), if the hypothesis of Theorem \ref{thm:mainresult} holds and moreover $q(x,dx')$ is nonsingular for \textit{every} $x \in (-1,1)$, then $m^2$ is uniquely ergodic. In general, however, the question of unique ergodicity is more delicate. See Remarks \ref{rem:uniqueerg0} and \ref{rem:uniqueerg} for more discussion of this issue. 

We prove Theorem \ref{thm:mainresult} in \S\ref{ssec:main1proof}, and we prove Theorem \ref{thm:qNS} in \S\ref{ssec:main2proof}.

\subsection{Application: ergodicity of the disk-in-strip billiard system} \label{ssec:ergodic}

To describe analogous ergodicity results for the disk-in-strip billiard system, we must first define the collision dynamics more precisely. We introduce a special system of spherical coordinates for the kinetic energy unit sphere $\mathbb{S}^2$, defined by \eqref{eq:kinspheredef}. Define vectors in $\mathbb{R}^3$,
\begin{equation} \label{eq:orthonormal}
\chi = \left(\frac{2}{J + m}\right)^{1/2}(1,-1,0), \quad n_1 = \left( \frac{2}{J^{-1} + m^{-1}} \right)^{1/2}(J^{-1}, m^{-1}, 0), \quad n_2 = \left(\frac{2}{m} \right)^{1/2}(0,0,1).
\end{equation}
Each of these vectors has unit kinetic energy. The triple $(\chi,n_1,n_2)$ forms a positively oriented orthonormal frame with respect to the inner product $\langle \cdot, \cdot \rangle$, defined by \eqref{eq:kinnerproduct}. We define a spherical coordinate mapping $G : (0,\pi) \times (0,\pi) \to \mathbb{S}^2_+$ by 
\begin{equation} \label{eq:spherecoord}
G(\theta,\psi) = (\cos\psi)\chi + (\cos\theta\sin\psi)n_1 + (\sin\theta\sin\psi)n_2.
\end{equation}

Recall the collision law $K_+$, which determines the Markov kernel $K$ for the post-collision velocity process on $\hat{\mathbb{S}}^2$ through equation \eqref{eq:twosides}. We take as axiomatic the assumption that $K_+$ has the following $(\theta,\psi)$-coordinate representation:
\begin{equation} \label{eq:roughcolform0}
(G^{-1})_*K_+(\theta,\psi; d\theta' d\psi') = P(\theta, d\theta') \delta_{\pi -\psi}(d\psi'), 
\end{equation}
where $P$ is allowed to be any Markov kernel on $(0,\pi)$ which is reversible with respect to the following probability measure on $(0,\pi)$:
\begin{equation} \label{eq:defLambda1}
\Lambda^1(d\theta) = \frac{1}{2}\sin\theta d\theta.
\end{equation}
In the future, when we work with spherical coordinate representations, we will often leave out the pre-composition with $G$ in our expressions.

A consequence of the fact that $P$ preserves the measure $\Lambda^1$ is that the Markov kernel $K$ for the velocity process preserves the following measure on $\mathbb{S}^2$:
\begin{equation} \label{eq:defLambda}
\Lambda^2(dv) = \frac{1}{2\pi}|\langle v, n_2 \rangle|\sigma^2(dv),
\end{equation}
where $\sigma^2$ is surface measure on $\mathbb{S}^2$, induced by the metric $\langle \cdot, \cdot \rangle$. The normalization factor guarantees that $\Lambda^2$ is a probability measure. We may extend the spherical coordinates $(\theta,\psi)$ to the whole sphere minus the poles by extending the domain to $(0,2\pi) \times (0,\pi)$. We let $\widetilde{G} : (0,2\pi) \times (0,\pi) \to \mathbb{S}^2 \smallsetminus \{(v_0,v_1,v_2) : v_0 = v_1 = 0\}$ denote the corresponding spherical coordinate map, defined by the same formula \eqref{eq:spherecoord}. The resulting coordinate representation of $\Lambda^2$ is given by
\begin{equation} \label{eq:lambda2coordrep}
\Lambda^2(d\theta d\psi) = \frac{1}{2\pi}|\sin\theta| \sin^2\psi d\theta d\psi, \quad (\theta,\psi) \in (0,2\pi) \times (0,\pi). 
\end{equation}
If we restrict the measure to the upper hemisphere $\mathbb{S}^2_+$, then this expression can also be written as 
\begin{equation}
\Lambda^2(d\theta d\psi) = \frac{1}{\pi} \Lambda^1(d\theta) \times \sin^2\psi d\psi, \quad (\theta,\psi) \in (0,\pi) \times (0,\pi),
\end{equation}
where $\Lambda^1$ is defined by \eqref{eq:defLambda1}.

We will see in Theorem \ref{thm:colchar} that the class of Markov kernels $K_+$ which can be obtained through a certain ``physically realistic'' limiting procedure is characterized by the property that $K_+$ has the coordinate representation \eqref{eq:roughcolform0}. This limiting procedure is meant to describe collisions between a freely moving disk and a fixed wall whose surfaces have been equipped with microscopic rough features, as the scale of these features tends to zero. The form \eqref{eq:roughcolform0} of a rough collision law is quite directly related to the choice of microstructures with which we equip the bodies, as we will explain in \S\ref{ssec:rough}. The Markov kernel $P(\theta,d\theta')$ describes the way in which a point particle reflects from a rough surface in $\mathbb{R}^2$, by giving the joint distribution of the angle of approach $\theta$ and the angle of exit $\theta'$. For an appropriate choice of rough microstructure on the disk, the rough microstructure which gives rise to $P$ may be obtained by foreshortening the original rough microstructure on the lower boundary of the strip by a factor of $(1 + mJ^{-1})^{-1/2}$.

We now describe the relationship between the alternating random walk and the velocity process for the disk-in-strip system. Recall the 180 degree rotation $R$, given by 
\begin{equation} \label{eq:180degree}
R(v_0,v_1,v_2) = (v_0, -v_1, -v_2).
\end{equation}
Let 
\begin{equation} \label{eq:gammadef}
\gamma = \arccos(\langle n_1, R(n_1) \rangle) = \arccos \left( \frac{J^{-1} - m^{-1}}{J^{-1} + m^{-1}} \right).
\end{equation}
Consider the homeomorphism $\Phi^+ : E \to \mathbb{S}^2_+$ given by 
$$
\Phi^+(u,v) = u R(n_1) + v n_1 + (1 - u^2 + v^2)^{1/2}n_2.
$$
To see that the image of $\Phi$ is indeed $\mathbb{S}^2_+$, note that 
\begin{equation} \label{eq:todisk}
\|u R(n_1) + v n_1\|^2 = u^2 + 2\langle R(n_1),n_1 \rangle uv + v^2 = u^2 + 2uv\cos\gamma + v^2.
\end{equation}
Since $n_1$ and $R(n_1)$ are linearly independent, $(u,v) \mapsto u R(n_1) + v n_1$ is a bijective linear map from $\mathbb{R}^2$ to the plane $\mathbf{P} := \{(v_0,v_1,v_2) : v_2 = 0\}$. The image of $E$ under this mapping is therefore the open unit disk in $\mathbf{P}$, in view of \eqref{eq:todisk} and the definition \eqref{eq:Edef} of $E$. Hence, the image of $\Phi^+$ is $\mathbb{S}^2_+$. 

Similarly, the mapping $\Phi^- : E \to \mathbb{S}^2_-$, given by 
$$
\Phi^-(u,v) = u R(n_1) + v n_1 - (1 - u^2 + v^2)^{1/2}n_2,
$$
is a homeomorphism. We define a homeomorphism $\Phi : E^* \to \hat{\mathbb{S}}^2$ by 
$$
\Phi(u,v,s) = \begin{cases}
\Phi^-(u,v) & \text{ if } s = -1, \\
\Phi^+(u,v) & \text{ if } s = 1.
\end{cases}
$$
We let $H : \hat{\mathbb{S}}^2 \to E^*$ denote the inverse of $\Phi$. Recall the Markov kernel $K$ on $\hat{\mathbb{S}}^2$, as well as the Markov kernel $P$ on $(0,\pi)$, appearing in the coordinate representation \eqref{eq:roughcolform0} of $K_+$. Noting that $h := \cos : (0,\pi) \to (-1,1)$ is a homeomorphism, we define a Markov kernel $q(x, dx')$ on $(-1,1)$ by 
\begin{equation} \label{eq:qtoP}
q(x,dx') = h_*P(x, dx').
\end{equation}
Let $\hat{q}$ and $Q$ be defined as in \eqref{eq:hatqdef} and \eqref{eq:Qdef}, respectively, with $q$ defined as above. We have 
\begin{lemma} \label{lem:iso}
Let $q$ be defined as in \eqref{eq:qtoP}, and let $\hat{q}$ and $Q$ be defined as in \eqref{eq:hatqdef} and \eqref{eq:Qdef}, respectively. Then
\begin{enumerate} 
\item[(i)] uniform measure $m^1(dx)$ on $(-1,1)$ is reversible for $q$; and
\item[(ii)] $Q = H_*K$.
\end{enumerate}
\end{lemma}

It follows from Lemma \ref{lem:iso} and Proposition \ref{prop:ir} that $Q$ preserves the measure $m^2(du dv)$ on $E^*$. Since $H$ and $h$ are homeomorphisms, $h$ induces an isomorphism of the measure preserving systems $((0,\pi), P, \Lambda^1)$ and $((-1,1), q, m^1)$, and by the lemma, $H$ induces an isomorphism of the measure preserving systems $(\hat{\mathbb{S}}^2, K, \Lambda^2)$ and $(E^*, Q, m^2)$. 

Let $P^\dag$ be defined 
$$
P^\dag := (h^{-1})^*q^\dag.
$$
Equivalently, if $\widetilde{\mathcal{N}}(\theta) = \pi - \theta$ for $\theta \in (0,\pi)$, then $P^\dag = \widetilde{\mathcal{N}}_*P$. Lemma \ref{lem:iso} and Theorem \ref{thm:mainresult} immediately imply the following characterization of the collision laws $K$ for which $\Lambda^2$ is an ergodic measure.

\begin{theorem} \label{thm:ergcol}
Let $\gamma$ be given by \eqref{eq:gammadef}, and assume that $\gamma/\pi$ is irrational. Also assume there exists some ergodic measure for $Q$ which is nonsingular with respect to $\Lambda^2$. If $\Lambda^1$ is an ergodic measure for $P$, then $\Lambda^2$ is an ergodic measure for $K$.

Moreover, if $P = P^\dag$, the converse statement holds. That is, if $\Lambda^2$ is an ergodic measure for $K$, then $\Lambda^1$ is an ergodic measure for $P$.
\end{theorem}

Note that there is no need to state the condition that $P$ is reversible with respect to $\Lambda^1$, since this condition is already contained in the definition of $K$.

One might wonder if Theorem \ref{thm:ergcol} really makes the matter of verifying whether $\Lambda^2$ is ergodic much easier, since we apparently must still verify whether the measure $\Lambda^1$ is ergodic for $P$. In fact, we will see in \S\ref{ssec:examples} that this theorem, combined with Remark \ref{rem:nonsing} and the main results of \cite{rudzis2022}, allows us to construct explicit examples of Markov kernels $K$ for which $\Lambda^2$ is ergodic. (Here ``explicit'' means that we can describe the microstructure which gives rise to the collision law explicitly, and we can write down some explicit formula for $K$.)

\begin{remark} \label{rem:nonsing} \normalfont
It follows from Theorem 1.4 that if, for almost every $\theta \in (0,\pi)$, $P(\theta, \cdot)$ is nonsingular with respect to $\Lambda^1$, then there exists some ergodic measure for $K$ which is absolutely continuous with respect to $\Lambda^2$. 

Moreover, if for every $\theta \in (0,\pi)$, $P(\theta, \cdot)$ is nonsingular with respect to $\Lambda^1$, then every invariant measure for $K$ is nonsingular with respect to $\Lambda^2$. Thus, if $\Lambda^2$ is ergodic, then it is uniquely ergodic.
\end{remark}

\begin{remark} \label{rem:signs} \normalfont
The Markov kernel $P(\theta,d\theta')$ on $(0,\pi)$ is supposed to define reflections of a point particle from a certain foreshortened rough microstructure. The angle of the \textit{negative} of the incoming velocity with the horizontal direction is given by $\theta$, and the angle of the outgoing velocity with the horizontal is given by $\theta'$. See \S\ref{sssec:charrough} for further discussion. On the other hand, we do not change the sign of the pre-collision velocity in the definition of the Markov kernel $K(v,dv')$ for the velocity process. In view of the isomorphism of measure preserving systems given by Lemma \ref{lem:iso}, this explains from a physical standpoint the ``mismatch'' of $q$ and $\hat{q}$ in the hypothesis and conclusion, respectively, of Theorem \ref{thm:mainresult}. 
\end{remark}

\begin{remark} \label{rem:uniqueerg0} \normalfont
One may also ask under what conditions $\Lambda^2$ is the unique ergodic measure for $K$. If the hypothesis of Theorem \ref{thm:ergcol} holds, and for every $\theta \in (0,\pi)$, $P(\theta,\cdot)$ is nonsingular with respect to $\Lambda^1$, then unique ergodicity is immediate from the theorem and Remark \ref{rem:nonsing}. (Recall that any two ergodic measures will be mutually singular.) More generally, however, the problem is more difficult and may not even be well-defined. Indeed, if $K$ is obtained from a rough microstructure (that is, through the limiting procedure described in \S\ref{sec:billiards}), then $K$ is only defined up to a $\Lambda^2$-null set. For more discussion of this issue, see Remark \ref{rem:uniqueerg}. 
\end{remark}

\subsection{Future work} \label{ssec:future}

The results of this paper can be generalized in a few different directions. One can consider analogues of the alternating random walk in general domains $D \subset \mathbb{R}^d$, where each step is taken in a coordinate direction along some chord of $D$. As we saw in condition (ii) of Theorem \ref{thm:mainresult}, to establish ergodicity, one apparently cannot avoid taking into account the arithmetic properties of the domain, at least if the class of transition kernels $q$ is sufficiently general. However, domain-dependent properties may become less important with a more restricted class of transition kernels $q$. Since finding examples of transition kernels $q$ which give rise to non-ergodic dynamics depends heavily on the symmetries of the domain, this suggests that, generically, an ergodic transition kernel $q$ should make the alternating random walk ergodic, provided that the domain $D$ is ``not very symmetric.'' 

Also of interest, both for the elliptical domain investigated here and for more general domains, is the problem of finding conditions on the transition kernel $q$ which give rates of convergence to stationarity for the alternating random walk. Depending on how general one makes the class of domains and how restrictive one makes the class of transition kernels, these types of questions interpolate to analyses of the coordinate hit-and-run algorithm, mentioned in \S\ref{ssec:relatedwork}.

One can generalize the billiard system considered in this paper to a system in which a rough disk moves in some bounded polygonal domain and undergoes collisions according to the collision law \eqref{eq:roughdecomp}. If the walls of the domain are not parallel, the analysis is more complex, since one must keep track of spatial coordinates, as well, in order to represent the system as a Markov process. Also of interest are billiard systems where there are multiple rough disks; as well as higher dimensional systems, where instead of disks, the particles are, say, hard spheres. Since \cite{rudzis2022} only deals with collisions in the two-dimensions, the appropriate generalization of the class of collision laws \eqref{eq:roughdecomp} to higher dimensions is not so clear. In view of the heuristic argument at the end of \S\ref{ssec:phys}, one reasonable approach would be to restrict the analysis to Markov kernels which conserve quantities analogous to \eqref{eq:conserved1} and \eqref{eq:conserved2}, coming from ``rolling collisions'' in higher dimensions. Such systems would again fall under the umbrella of semi-deterministic processes discussed in the introduction.

\section{Disk-in-strip billiard system} \label{sec:billiards}

This section provides a billiard theoretic justification for the class of collision dynamics considered in this paper. In addition, by connecting the nontrivial factor $P(\theta,d\theta')$ in the coordinate representation \eqref{eq:roughcolform0} to the choice of microstructure on the surfaces of the disk and the walls of the strip, we are able to construct explicit examples of microstructures which produce ergodic dynamics. The relationship between $P$ and the microstructure was originally fleshed out in the monograph \cite{rudzis2022}. We begin with some preliminary notions.

\subsection{Preliminaries} \label{ssec:prelim} 

\subsubsection{Ergodic theory from the ``dual'' point of view} \label{sssec:dual}

In the ergodic theory of Markov processes, results are typically framed in such a way as to elucidate the structure of the families of measures which are stationary or ergodic, relative to some Markov process which is fixed. One can just as well consider the ``dual'' line of inquiry, instead fixing the measure and investigating the structure of the families of Markov processes which make the measure invariant or ergodic. The second framing is more natural in some situations, including the setting of this paper and other situations remarked upon below.

Recall that if $(X,\mathcal{X})$ is a measure space, then $\MK(X,\mathcal{X})$ denotes the set of all Markov kernels on $(X,\mathcal{X})$. If $\mu$ is a measure on $(X,\mathcal{X})$, then we let $\MKI(X,\mathcal{X},\mu)$ denote the set of all Markov kernels on $(X,\mathcal{X})$ which preserve the measure $\mu$. I.e. $p \in \MKI(X,\mathcal{X},\mu)$ if and only if $p \in \MK(X,\mathcal{X})$, and $\int p(x,\cdot) \mu(dx) = \mu$.

We say that a set $A \in \mathcal{X}$ is \textit{$(\mu,p)$-invariant} if for $\mu$-a.e. $x \in A$, $p(x,A) = 1$. It is easy to verify that the collection of $(p,\mu)$-invariant sets in $\mathcal{X}$ forms a $\sigma$-algebra. 

We say that a set $A \in \mathcal{X}$ is \textit{$(p,\mu)$-ergodic} if $A$ is $(\mu,p)$-invariant, and for every $B \in \mathcal{X}$ such that $B \subset A$, if $B$ is $(\mu,p)$-invariant, then $\mu(B) \in \{0,\mu(A)\}$. Furthermore, we say that a measure $\mu$ is \textit{ergodic} with respect to $p \in \MK(X,\mathcal{X})$ if $X$ is a $(p,\mu)$-ergodic set. That is, for every $(\mu,p)$-invariant set $B \in \mathcal{X}$, $\mu(B) \in \{0,1\}$.

This notion of ergodicity of measures is the same as the usual one for Markov processes, which we briefly review now (for more details, we refer the reader to \cite{hairer-notes}). Let $S = X^\mathbb{Z}$ be the space of $X$-valued, bi-infinite sequences, and let $\mathcal{X} = \sigma(\bigotimes_{i \in \mathbb{Z}} A_i : A_i \in \mathcal{X})$ be the product $\sigma$-algebra on $S$. If $\mu$ is a probability measure on $(X,\mathcal{X})$ and $p \in \MKI(X,\mathcal{X},\mu)$, one may construct a probability measure $\mathbb{P}_\mu$ on $(S, \mathcal{S})$ whose finite-dimensional distributions are given as follows: For integers $M \leq N$ and sets $A_{M}, A_{M+1}, A_{M+2}, \dots, A_N \in \mathcal{X}$,
\begin{equation}
\begin{split}
    & \mathbb{P}_\mu(\{\{x_i\}_{i \in \mathbb{Z}} \in \Omega : x_i \in A_i \text{ for } M \leq i \leq N\}) \\
    & = \int_{A_{M}} \cdots \int_{A_{N-1}} \int_{A_{N}} p(x_{N-1}, d x_{N}) p(x_{N-2}, d x_{N-1}) \\
    & \hspace{2.5in}\cdots \ p(x_{M}, d x_{M+1}) \mu(d x_{M}). 
\end{split}
\end{equation}
Define the \textit{shift map} $s : S \to S$ by 
\begin{equation}
    s(\{x_i\}_{i \in \mathbb{Z}}) = \{x_{i+1}\}_{i \in \mathbb{Z}}, \quad \text{ for } (x_i)_{i \in \mathbb{Z}} \in \Omega.
\end{equation}
Using invariance of $\mu$, one may show that the shift map preserves the measure $\mathbb{P}_\mu$. According to the usual definition, the probability measure $\mu$ is ergodic if the corresponding probability measure $\mathbb{P}_\mu$ is ergodic with respect to the action of the shift map, i.e. for all $B \in \mathcal{S}$, if $s^{-1}(B) = B$, then $\mathbb{P}_{\mu}(B) \in \{0,1\}$. We have

\begin{proposition} \label{prop:ergsets}
A probability measure $\mu$ on $(X,\mathcal{X})$ is ergodic with respect to a Markov transition kernel $p$, in the sense given above, if and only if, for every $(\mu,p)$-invariant set $B \in \mathcal{X}$, $\mu(B) \in \{0,1\}$.
\end{proposition}

\noindent For a proof of this proposition, we refer the reader to \cite[Corollary 5.11]{hairer-notes}.

\begin{remark} \label{rem:dualerg} \normalfont
Let us mention a few research areas where the ``dual'' perspective, putting spaces of Markov chains front-and-center, is more natural. One case arises in the context of rough billiards, which we discuss in more detail below, in \S\ref{ssec:rough}. See specifically Theorem \ref{thm:colchar} and the work of Plakhov \cite{plakhov2004newtonprob, plakhov2009billiards, plakhov2012}, and Angel, Burdzy, and Sheffield \cite{ABS2013detapprox}. A rather different case may be found in Markov Chain Monte Carlo (MCMC) algorithms \cite{speagle2020}. Here the point of interest is the measure $\mu$ from which one wishes to sample; one accomplishes this by finding \textit{some} Markov kernel $P \in \MKE(\mu)$ which mixes fast. A third example, which does not strictly fall under the category of ergodic theory, but still arguably participates in this framing is optimal mass transport. Here the goal is to minimize a cost functional over couplings of a pair of probability distributions $\mu$ and $\nu$ which have been fixed. By disintegration of measure, any such coupling can be identified with a Markov kernel $P$ such that $P\mu = \nu$. When $\mu = \nu$, the problem of optimal mass transport reduces to that of minimizing a cost functional over the space $\MKI(\mu)$. A good introduction to this subject area may be found in \cite{santambrogio2015}.
\end{remark}

\subsubsection{Pseudometric topology} \label{sssec:pseudo}

Let $X$ be a locally compact and $\sigma$-compact topological space, and let $\mathcal{X}$ be the Borel $\sigma$-algebra on $X$. Let $C_c(X \times X)$ denote the space of compactly supported continuous functions on $X \times X$, equipped with the topology of uniform convergence, and let $d$ be a metric on the dual space $C_c(X \times X)^*$ which induces the vague (or weak$^*$) topology on $C_c(X \times X)^*$. For example, we could take 
$$
d(\nu_1,\nu_2) := \sum_{i = 1}^\infty 2^{-i} \frac{|\int f_i d(\nu_1 - \nu_2)|}{1 + |\int f_i d(\nu_1 - \nu_2)|}, \quad \nu_1,\nu_2 \in C_c(X \times X)^*,
$$
where $\{f_i\}_{i \in \mathbb{N}}$ is some enumeration of a countable dense subset of $C_c(X \times X)$. (For further discussion, see \cite[Chapter 7]{folland1999}. Note that it is important that we consider the dual space $C_c(X \times X)^*$ instead of $C_0(X \times X)^*$; the latter space may not be metrizable.) For any $p \in \MK(X,\mathcal{X})$ and $\mu \in C_c(X)^*$, we define the measure $p \otimes \mu \in C_c(X \times X)^*$ by
$$
p \otimes \mu(dx,dx') = p(x,dx')\mu(dx).
$$
Given a measure $\mu$ on $(X,\mathcal{X})$, we equip $\MK(X,\mathcal{X})$ with the pseudometric
$$
d^\mu(p,p') := d(p \otimes \mu, p' \otimes \mu), \quad p, p' \in \MK(X,\mathcal{X}).
$$
The set $\MK(X,\mathcal{X})$ partitions into equivalence classes defined by the equivalence relation: $p \sim_\mu p'$ if and only if $d^\mu(p,p') = 0$. Denote the set of equivalence classes by $\MK(X,\mathcal{X},\mu)$ (or just $\MK(\mu)$ for short). The pseudometric $d^\mu$ on $\MK(X,\mathcal{X})$ induces a metric on $\MK(X,\mathcal{X},\mu)$, which by abuse of notation we denote also by $d^\mu$. Observe that the equivalence relation $\sim_\mu$ preserves the property that $P \in \MKI(X,\mathcal{X},\mu)$ and the property that $P \in \MKE(X,\mathcal{X},\mu)$. In what follows, we will regard $\MKI(X,\mathcal{X},\mu)$ and $\MKE(X,\mathcal{X},\mu)$ as metric subspaces of $\MK(X,\mathcal{X},\mu)$.

\begin{remark} \label{rem:uniqueerg} \normalfont
One might wonder how unique ergodicity fits into this picture. Note that if $p \in \MKE(\mu)$, then as long as there exists some non-empty $\mu$-null set $N \in \mathcal{X}$, there will always be a Markov kernel $p' \sim_\mu p$ such that $\mu$ is ergodic but not uniquely ergodic with respect to $p'$. Indeed, we may define $p'(x,\cdot) = p(x,\cdot)$ if $x \in X \smallsetminus N$ and $p'(x,\cdot) = \delta_x(\cdot)$ if $x \in N$. For any choice of $x_0 \in N$, the measures $\mu$ and $\delta_{x_0}$ are distinct measures which are ergodic for $p'$. In other words, if $p$ is only defined up to its equivalence class in $\MKE(\mu)$ (e.g. if $p$ defined as a some limit of Markov kernels in the topology induced by $d^\mu$), then it makes no sense to ask the question whether $\mu$ is uniquely ergodic for $p$, as $p$ is only defined up to $\mu$-null sets. This type of issue actually arises in the definition of rough collision laws, presented in the next subsection. On the other hand, not all is lost; for if $\lambda$ is some other measure (not necessarily invariant for $p$), and $p$ is defined up to its equivalence class in $\MK(\lambda)$, then it does make sense to ask whether $\mu$ is the unique ergodic measure for $p$ which is not singular with respect to the measure $\lambda$.
\end{remark}

\subsection{Billiard theoretic derivation of the class of rough collision laws} \label{ssec:rough}

\subsubsection{Rough reflection laws on submanifolds of $\mathbb{R}^d$.} We begin by reviewing a billiard theoretic construction of rough reflection laws, originally presented in \cite[\S 6]{rudzis2022} with variants of the construction appearing previously in \cite{plakhov2004newtonprob, plakhov2009billiards, plakhov2012, ABS2013detapprox}. Any general treatment of mathematical billiards must impose regularity conditions on the boundary of a billiard domain, so that for almost all initial conditions, the billiard trajectory is defined for all time and avoids certain pathological behavior (e.g. tangential collisions or accumulation of collision times). We require our domains to be subsets of $\mathbb{R}^d$ with boundaries which are $C^2$ except possibly on a Hausdorff $(d-1)$-measure zero set of ``singularities.'' To avoid getting distracted by technicalities, we refer the reader to \cite[\S 6.1.1]{rudzis2022} for the precise statements of the regularity conditions, and we will take as given that for almost all initial conditions, our billiard trajectories are defined for all time and avoid the pathological behavior mentioned above. Since this subsection serves mainly to justify the definition of the class $\mathcal{A}$ of collision laws, defined in \S\ref{ssec:ergodic}, and the statements and proofs of our main results do not use definitions and results presented here, a greater level of detail does not seem warranted.

To begin, consider $\mathbb{R}^d$, where $d \geq 1$, equipped with some inner product $\langle \cdot, \cdot \rangle$. Let $\mathcal{M}_1$ be some smooth, embedded submanifold of $\mathbb{R}^d$ with boundary. Suppose $\mathcal{M}_0$ is any closed subset of $\mathbb{R}^d$ with $\mathcal{M}_0 \supset \mathcal{M}_1$, such that 
\begin{enumerate}
\item[(i)] $\sup_{x \in \mathcal{M}_1}\text{dist}(x, \mathbb{R}^d \smallsetminus \mathcal{M}_0) < \infty$, and 
\item[(ii)] $\mathcal{M}_0$ satisfies the regularity conditions of \cite[\S 6.1.1]{rudzis2022}.
\end{enumerate} 
For each $x \in \partial \mathcal{M}_0$, we say that a vector $v \in \mathbb{R}^d$ \textit{points into $\mathcal{M}_0$ at $x$} if there exists $\epsilon > 0$ such that for all $s \in (0,\epsilon)$, $x + sv  \in \text{Int}\mathcal{M}_0$. We refer to a point $x_0$ in $\partial \mathcal{M}_0$ as a \textit{regular point} of the boundary if there is a neighborhood $U$ of $x_0$, relative to $\mathbb{R}^d$, such that $U \cap \partial \mathcal{M}_0$ is a $C^2$-embedded submanifold of $\mathbb{R}^d$. In a neighborhood of any such point in $\partial\mathcal{M}_0$, there is a differentiable field of unit normal vectors $n(x)$, where normality is with respect to the inner product $\langle \cdot, \cdot \rangle$. Note that at such a point $x_0$, a vector $v$ is inward-pointing if $\langle v, n(x_0) \rangle > 0$. Similar such definitions apply to $\mathcal{M}_1$, except note every point in $\partial \mathcal{M}_1$ is a regular point of its boundary.

Consider a billiard particle in $\mathcal{M}_0$ which moves linearly in the interior of $\mathcal{M}_0$ and undergoes \textit{specular reflection} from the boundary. In other words, if $x(t)$ is the trajectory of the billiard particle in $\mathbb{R}^d$, which at some time $t_0$ collides with the boundary of $\mathcal{M}_0$ at a regular point of $\partial \mathcal{M}_0$, then 
\begin{equation}
x'(t_0+) = x'(t_0-) + 2\langle x'(t_0-), n(x(t_0)) \rangle n(x(t_0)), \label{eq:spec}
\end{equation}
where $x'(t_0\pm)$ denotes, respectively, the left (-) and right (+) -hand derivatives of $x$ at $t_0$, and $n(x(t_0))$ is the unit normal vector pointing into $\mathcal{M}_0$ at $x(t_0)$. 

Let $\|\cdot\|$ denote the norm on $\mathbb{R}^d$ determined by the inner product $\langle \cdot, \cdot \rangle$. Since the specular reflection law \eqref{eq:spec} preserves $\| \cdot \|$, the speed of the billiard particle is constant for all time. Let $S^{d-1} \subset \mathbb{R}^d$ denote the unit sphere with respect to the inner product $\langle \cdot, \cdot \rangle$. The state of the billiard particle at any time may be described by a pair $(x,v) \in \mathcal{M}_0 \times S^{d-1}$, where $x$ is the position of the billiard particle, and $v$ is its velocity.

Define
$$
S_+\partial \mathcal{M}_1 = \{(x,v) \in \partial\mathcal{M}_1 \times S^{d-1} : v \text{ points into $\mathcal{M}_1$ at } x\}.
$$
Assume that the billiard has the following property:
\begin{enumerate}
\item[(P)] There exists a full-measure subset $\mathcal{D} \subset S_+\partial\mathcal{M}_1$, such that for every $(x,v) \in \mathcal{D}$, the trajectory starting from $(x,-v)$ is defined for all time and eventually returns to $\mathcal{D}$ in some state $(x',v') \in S_+\partial \mathcal{M}_1$. \label{property-welldef}
\end{enumerate}
We define the mapping $P^{\mathcal{M}_0, \mathcal{M}_1} : \mathcal{D} \to \mathcal{D}$ by 
\begin{equation}
P^{\mathcal{M}_0, \mathcal{M}_1}(x,v) = (x',v'). \label{eq:reflaw}
\end{equation}
We call $P^{\mathcal{M}_0, \mathcal{M}_1}$ the reflection law associated with $\mathcal{M}_0$ and $\mathcal{M}_1$. 

Two key properties of the reflection law are summarized in the following proposition, proved in \cite{rudzis2022}.

\begin{proposition} \label{prop:rrefP}
\cite[Proposition 6.5]{rudzis2022} (i) $P^{\mathcal{M}_0, \mathcal{M}_1}$ is an involution on $\mathcal{D}$, i.e. $P^{\mathcal{M}_0, \mathcal{M}_1} \circ P^{\mathcal{M}_0, \mathcal{M}_1} = \text{Id}_{\mathcal{D}}$.

(ii) Let $m_{\partial\mathcal{M}_1}(dx)$ denote the surface measure on $\partial \mathcal{M}_1$ induced by the inner product $\langle \cdot, \cdot \rangle$ on $\mathbb{R}^d$, and let $\sigma(dv)$ denote surface measure on $S^{d-1}$. Let $\widehat{\Lambda}(dx dv)$ be the measure on $S_+\partial\mathcal{M}_1$ defined by restricting the measure $|\langle v, n(x) \rangle| \sigma(dv) m(dx)$ on $\partial \mathcal{M}_1 \times S^{d-1}$ to $S_+\partial\mathcal{M}_1$. Then $P^{\mathcal{M}_0, \mathcal{M}_1}$ preserves $\widehat{\Lambda}$ in the sense that 
$$
\widehat{\Lambda}((P^{\mathcal{M}_0, \mathcal{M}_1})^{-1}(B)) = \widehat{\Lambda}(B) \text{ for all Borel sets } B \subset S_+\partial \mathcal{M}.
$$
\end{proposition}
The measure $\widehat{\Lambda}$ appearing in the proposition above is referred to as \textit{Lambertian measure}.

We say that a Markov kernel $P(x,v; dx' dv')$ on $S_+\partial\mathcal{M}_1$ is a \textit{rough reflection law} if it takes the form
$$
P(x,v; dx' dv') = \delta_x(dx')\widetilde{P}(x,v; dv'),
$$
where, for each $(x,v) \in S_+\mathcal{M}_1$, $\widetilde{P}(x,v; dv')$ is a measure on the fiber $S_+\partial\mathcal{M}_1|_{x} := \{(x',v') \in S_+\partial\mathcal{M}_1 : x' = x\}$; and there exists a sequence $\{\mathcal{M}_\epsilon\}_{\epsilon > 0}$ of closed subsets of $\mathcal{M}_1$ satisfying conditions (i) and (ii), such that as $\epsilon \to 0$, 
$$
\mathcal{M}_\epsilon \to \mathcal{M}_1 \text{ in Hausdorff distance},
$$
and
\begin{equation}
\delta_{P^{\mathcal{M}_\epsilon, \mathcal{M}_1}(x,v)}(dx' dv') \to P(x,v; dx' dv') \text{ in } \MK(S_+\partial\mathcal{M}_1,\widehat{\Lambda}).\label{eq:markovlimit}
\end{equation}
Note that since $\mathcal{M}_\epsilon \subset \mathcal{M}_1$ for all $\epsilon$, the convergence of $\mathcal{M}_\epsilon$ to $\mathcal{M}_1$ in Hausdorff distance is equivalent to 
$$
\lim_{\epsilon \to 0}\sup_{x \in \mathcal{M}_1} \inf_{y \in \mathbb{R}^d \smallsetminus \mathcal{M}_\epsilon} \|x-y\|  \to 0. 
$$
The second convergence is with respect to the metric $d^{\widehat{\Lambda}}$ defined in \S\ref{sssec:pseudo}.

\subsubsection{Rough reflection laws on half-spaces.} The above construction simplifies in some respects if $\mathcal{M}_1$ is a closed half-space of form $\{x \in \mathbb{R}^d : \langle x - x_0, n_0 \rangle \geq 0\}$. In this case, 
$$
S_+\partial \mathcal{M}_1 = L(x_0,n_0) \times S_+(n_0),
$$
where $L(x_0,n_0) = \{x \in \mathbb{R}^d : \langle x - x_0, n_0 \rangle = 0\}$ and $S_+(n_0) = \{v \in S^{d-1} : \langle v, n_0 \rangle > 0\}$. In this setting, we say that $P(x,v; dx' dv')$ is \textit{homogeneous} if it takes the form 
$$
P(x,v; dx' dv') = \delta_x(dx') \widetilde{P}(v, dv'),
$$
for some Markov kernel $P(v, dv')$ on $S_+(n_0)$. For this paper, we will only be concerned with homogeneous rough reflection laws on half-spaces.

\subsubsection{Disk-and-wall system.} We are now ready to describe the precise sense in which the collision dynamics for the rough disk in the strip, described above, are the limiting case of classical rigid body collision dynamics. In the prelimit, we consider a freely moving disk and a fixed (infinitely massive) wall, whose surfaces are equipped with periodic, $\epsilon$-scale rigid rough features. To define the wall precisely, we introduce the following ingredients:
\begin{itemize}
\item A \textit{cell} $\Sigma$, defined to be a closed subset of the cylinder $\mathbb{S}^1 \times \mathbb{R}$ such that 
\begin{enumerate}
\item[1.] $\mathbb{S}^1 \times (-\infty,-1] \subset \Sigma \subset \mathbb{S}^1 \times (-\infty,0]$; \label{cell1}
\item[2.] the complement of $\Sigma$ in $\mathbb{S}^1 \times \mathbb{R}$ is connected; \label{cell2}
\item[3.] there exists a nontrivial loop $\gamma \subset \mathbb{S}^1 \times \mathbb{R}$ starting and ending at the point $(1,0)$ such that $\gamma$ lies entirely in $\Sigma$ (nontrivial means that $\gamma$ cannot be contracted in $\mathbb{S}^1 \times \mathbb{R}$ to a point); and \label{cell3}
\item[4.] the topological boundary of $\Sigma$ is piecewise $C^2$. \label{cell4}
\end{enumerate}
\item A \textit{roughness scale} $\epsilon > 0$.
\end{itemize}
We refer to a pair $(\Sigma,\epsilon)$ satisfying the above conditions as a \textit{shape-scale pair}. We define the wall $W = W(\Sigma, \epsilon)$ be the unique subset of $\mathbb{R}^2$ which is invariant under the translation
$$
(x_1,x_2) \mapsto (x_1 + \epsilon, x_2) : \mathbb{R}^2 \to \mathbb{R}^2,
$$
and whose image under the covering map 
$$
(x_1,x_2) \mapsto (e^{2\pi i x_1/\epsilon}, x_2/\epsilon) : \mathbb{R}^2 \to \mathbb{S}^1 \times \mathbb{R}
$$
is $\Sigma$.

In the definition of the cell $\Sigma$, the conditions 1-3 eliminate certain pathological examples which we do not want to consider. The first condition guarantees that the boundary of the wall $W(\Sigma,\epsilon)$ is bounded in the strip $-\epsilon \leq x_2 \leq 0$; the second condition says that there are no ``bubbles'' in the wall $W$; and the third eliminates examples where there are fissures in the wall which are arbitrarily long in the horizontal direction. Note that we allow the wall to have multiple connected components, and it is only the complement which is required to be connected. For the precise meaning of ``piecewise $C^2$'' in the fourth condition, see \cite[\S1.3.1]{rudzis2022}.

The pre-limiting collision dynamics do not admit a particularly tractable description if we allow the rough features on the disk to be as general as those of the wall (although the heuristic justification for the representation \eqref{eq:roughdecomp}) discussed at the end of \S\ref{ssec:phys} remains in force). Instead, we equip the surface of the disk with a special choice microstructure to interact with the general microstructure on the wall. Let $\rho(\epsilon)$ be some positive, non-increasing function of $\epsilon$ such that 
$$
\rho(\epsilon) \to 0, \text{ and } \frac{\epsilon^{1/2}}{\rho(\epsilon)} \to 0,  \text{ as } \epsilon \to 0.
$$
We define a \textit{reference body} as a certain subset of the plane,
$$
D = D(\epsilon) = \left( \bigcup_{k = 0}^{N(\epsilon) - 1} \{S_k(\epsilon)\} \right) \cup D_0 \subset \mathbb{R}^2,
$$
where 
$$
N(\epsilon) := \left\lceil \frac{2\pi}{\rho(\epsilon)} \right\rceil, \text{ and } S_k(\epsilon) := (\sin k\rho(\epsilon), -\cos k\rho(\epsilon)) \text{ for } 0 \leq k \leq N(\epsilon) - 1,
$$
and 
$$
D_0 \subset \{(x_1,x_2) : x_1^2 + x_2^2 \leq (1 - \rho(\epsilon)^2)^2\}.
$$
The reference body may be described as a ``disk with satellites.'' The quantity $N(\epsilon)$ is the number of satellites, and $S_k(\epsilon)$ is the position of the $k$-satellite, indexed in increasing order counterclockwise around the perimeter of the disk, starting from $S_0(\epsilon)$ at the bottom. 

For $(x_0,x_1,x_2) \in \mathbb{R}^3$, we denote by $D(x_0,x_1,x_2)$ the subset of $\mathbb{R}^2$ obtained by rotating the reference body $D$ counterclockwise about its center by an angle of $x_0$, and translating it by the vector $(x_1,x_2)$. We will sometimes abuse notation by letting $D$ denote both the reference body and the physical body which it represents. 

During a collision, only the satellites $S_k$ can come into contact with the wall $W$. Thus $D_0$ is of no mathematical significance in the analysis but only serves to convince the reader of the physical generality of the model.

We assume that $D$ has some rotationally symmetric mass density $\lambda$, not depending on $\epsilon$, and we let 
$$
m = \int_{\mathbb{R}^2} \lambda(dx), \quad J = \int_{\mathbb{R}^2} |x|^2\lambda(dx),
$$
denote the mass and moment of inertia, respectively, of the body. The configuration space for the disk-and-wall system is given by 
$$
\mathcal{M}(\Sigma, \epsilon) := \overline{\{(x_0,x_1,x_2) \in \mathbb{R}^3 : D(x_0,x_1,x_2) \cap W(\Sigma,\epsilon) = \emptyset\}},
$$
where the topological closure is taken in $\mathbb{R}^3$. We equip the ``velocity space''  $\mathbb{R}^3$ with the inner product $\langle \cdot, \cdot \rangle$ and norm $\| \cdot \|$ defined by \eqref{eq:kinnerproduct}.

The evolution of the system is described by a continuous trajectory $x(t) \in \mathcal{M}(\Sigma,\epsilon)$. The trajectory will follow billiard dynamics under appropriate physical assumptions, which may be stated somewhat informally as follows:
\begin{enumerate}
\item[P1.] The disk $D$ and wall $W$ may not interpenetrate.
\item[P2.] The system is subject to Euler's laws of rigid body motion (see \cite[\S 3]{rudzis2022}).
\item[P3.] When not in contact, the net force applied to each body is zero; and upon contact, a single impulsive force is applied to the disk at the point of contact and directed parallel to the unit normal vector on the wall.

\item[P4.] The kinetic energy of the disk is preserved for all time. 
\end{enumerate}
In view of P3, the trajectory will be linear as long as it does not make contact with the boundary of $\mathcal{M}(\Sigma,\epsilon)$. Moreover, the collision dynamics are given as follows.

\begin{proposition} \label{prop:isbilliard}
The unique collision dynamics which satisfy P1-P4 and guarantee that the trajectory of the system is defined for all time $t \in \mathbb{R}$, on a Lebesgue full-measure subset initial conditions, is specular reflection with respect to the inner product $\langle \cdot, \cdot \rangle$.
\end{proposition}

In particular, for almost all initial conditions, the trajectory will only make contact with the boundary of $\mathcal{M}(\Sigma,\epsilon)$ at points where there is a well-defined normal vector with respect to the inner product $\langle \cdot, \cdot \rangle$. For more precise statements of P1-P4, the above proposition, and its proof, we refer the reader to \cite[\S 3]{rudzis2022}. In what follows, we will take the billiard dynamics of the system as given and not concern ourselves further with P1-P4.

\subsubsection{Characterization of rough collision laws.} \label{sssec:charrough}
We refer to configurations $x = (x_0,x_1,x_2)$ of the disk which lie in $\partial \mathcal{M}(\Sigma,\epsilon)$ as \textit{collision configurations}. These may be shown to coincide with the configurations $x$ such that $D(x) \cap \text{Int}W(\Sigma,\epsilon) = \emptyset$, and at least one of the satellites $S_k(\epsilon; x)$ coincides with a point in the boundary of $W(\Sigma,\epsilon)$. From the conditions 1-4 imposed on the cell $\Sigma$,  
$$
\{(x_1,x_2) : x_1 \leq -\epsilon\} \subset W(\Sigma,\epsilon) \subset \{(x_1,x_2) : x_1 \leq 0\}.
$$
Because the satellites lie at unit distance from the center of the disk, if $x_2 \geq 1$, then $(x_0,x_1,x_2) \in \mathcal{M}(\Sigma,\epsilon)$, and there exists a constant $C$ such that if $x_2 \leq 1 - C\epsilon$, then $(x_0,x_1,x_2) \in \mathbb{R}^3 \smallsetminus \mathcal{M}(\Sigma,\epsilon)$. This implies that
$$
H_1^+ := \{(x_0,x_1,x_2) : x_2 \geq 1\} \subset \mathcal{M}(\Sigma,\epsilon),
$$
and 
$$
\sup_{x \in H_1^+} \inf_{y \in \mathbb{R}^3 \smallsetminus \mathcal{M}(\Sigma,\epsilon))} \|x - y\| \leq C\epsilon.
$$

Fix a shape-scale pair $(\Sigma, \epsilon)$. If $\mathcal{M}_1 = H_1^+$ and $\mathcal{M}_0 = \mathcal{M}(\Sigma,\epsilon)$, it follows from \cite[Proposition 4.8]{rudzis2022} that property \ref{property-welldef} is satisfied. Consequently, there exists a full-measure subset $\mathcal{D}_0 \subset \mathbb{R}^2 \times \mathbb{S}^2$ such that the mapping $K^{\Sigma,\epsilon} : \mathcal{D}_0 \to \mathcal{D}_0$, given by
$$
K^{\Sigma,\epsilon}(x,v) = P^{\mathcal{M}(\Sigma,\epsilon),H_1^+}(x,v),
$$
is defined. Here the right-hand side is the reflection law defined by \eqref{eq:reflaw}, where we identify $\partial H_1^+$ with $\mathbb{R}^2$ in the obvious way: $(x_0,x_1) \leftrightarrow (x_0,x_1,1)$. Intuitively, the mapping $K^{\Sigma,\epsilon}$ takes the state of the disk $(x,v)$ the instant before it interacts with the wall to the state of the disk $(x',v')$ the instant after it ceases to interact. By Proposition \ref{prop:rrefP}, $K^{\Sigma,\epsilon}$ is an involution of $\mathcal{D}_0$, and it preserves the measure $\widehat{\Lambda}^2$ on $\mathbb{R}^2 \times \mathbb{S}^2_+$ defined by
\begin{equation} \label{eq:defhatLambda}
\widehat{\Lambda}^2(dx dv) = |dx|\Lambda^2(dv) = \langle v, n_2 \rangle |dx| \sigma^2(dv),
\end{equation}
where $|dx|$ is Lebesgue measure on $\mathbb{R}^2$, $\Lambda^2(dv)$ is the measure defined by \eqref{eq:defLambda}, $\sigma^2$ is the spherical surface measure on $\mathbb{S}^2_+$ induced by the inner product $\langle \cdot, \cdot \rangle$, and the vector $n_2$ is defined as in \eqref{eq:orthonormal}.

Consider a sequence $\{(\Sigma_i,\epsilon_i)\}_{i \in \mathbb{N}}$ of shape-scale pairs, where $\epsilon_i \to 0$ as $i \to \infty$. Our goal is to describe the class of Markov kernels which may be obtained as the limit of deterministic Markov kernels associated with the sequence of mappings $K^{\Sigma_i,\epsilon_i}(x,v)$, in the sense of \eqref{eq:markovlimit}. Such a limiting Markov kernel $K$ will be referred to as a \textit{rough collision law.} The description we present in Theorem \ref{thm:colchar} relates the limit of the collision laws associated with the sequence $(\Sigma_i,\epsilon_i)$ to the limit of reflection laws associated with a related sequence $(\widetilde{\Sigma}_i,\widetilde{\epsilon}_i)$, obtained by foreshortening the walls in one direction. 

We define the half-plane $H_0^+ = \{(x_1,x_2) : x_2 \geq 0\}$, and given any shape-scale pair $(\widetilde{\Sigma}, \widetilde{\epsilon})$, we define 
$$
\mathcal{T}(\widetilde{\Sigma},\widetilde{\epsilon}) = \mathbb{R}^2 \smallsetminus \text{Int}W(\widetilde{\Sigma},\widetilde{\epsilon}).
$$ 
Note that 
$$
\sup_{x \in H_0^+} \inf_{y \in \mathbb{R}^2 \smallsetminus \mathcal{T}(\widetilde{\Sigma},\widetilde{\epsilon}))} \| x - y\| \leq \widetilde{\epsilon},
$$ 
in view of properties \ref{cell1}-\ref{cell4}. By the discussion in \cite[\S1.2.2]{rudzis2022}, if $\mathcal{M}_1 = H_0^+$ and $\mathcal{M}_0 = \mathcal{T}(\widetilde{\Sigma},\widetilde{\epsilon})$, then property \ref{property-welldef} is satisfied. Consequently, there is a full measure subset $\widetilde{\mathcal{D}}_0 \subset \mathbb{R} \times (0,\pi)$ such that the mapping $P^{\widetilde{\Sigma},\widetilde{\epsilon}} : \widetilde{\mathcal{D}}_0 \to \widetilde{\mathcal{D}}_0$, given by 
$$
P^{\widetilde{\Sigma},\widetilde{\epsilon}}(u,\theta) = P^{\mathcal{T}(\widetilde{\Sigma},\widetilde{\epsilon}), H_0^+}(u,\theta),
$$
is defined. Here the right-hand side is defined by \eqref{eq:reflaw}, where $\mathbb{R}^2$ is equipped with the usual Euclidean metric. Note that we are identifying $(u,\theta)$ with the pair $(x,v) \in \partial H_0^+ \times \mathbb{S}^1_+$, where $x = (u,0,0)$ and $\theta$ is the Euclidean angle which $v$ makes in $\mathbb{R}^2$ with the vector $e_1 = (1,0)$. By Proposition \ref{prop:rrefP}, $P^{\widetilde{\Sigma},\widetilde{\epsilon}}$ is an involution on $\widetilde{\mathcal{D}}_0$, and it preserves the measure $\widehat{\Lambda}^1$ on $\mathbb{R} \times (0,\pi)$, defined by 
\begin{equation} \label{eq:defhatLambda1}
\widehat{\Lambda}^1(dx d\theta) = |dx|\Lambda^1(d\theta) = \frac{1}{2}|dx|\sin\theta d\theta,
\end{equation}
where $|dx|$ is Lebesgue measure on $\mathbb{R}$, and $\Lambda^1$ is the measure defined by \eqref{eq:defLambda1}.

For the sequence of walls $\{W(\Sigma_i, \epsilon_i)\}_{i \in \mathbb{N}}$, the shapes of the rough features are captured by the $\Sigma_i$, while their scales are captured by the $\epsilon_i$. The next theorem says that the limiting rough collision law depends on the sequence $\{\epsilon_i\}_{i \in \mathbb{N}}$ only in a very weak sense, namely that the limiting collision law (if it exists) is completely determined by the sequence $\{\Sigma_i\}_{i \in \mathbb{N}}$ provided that the sequence $\epsilon_i$ goes to zero sufficiently fast. Moreover, the limit always exists if the sequence of cells is constant.

\begin{theorem} \label{thm:dichotomy}
\cite[Theorem 1.28]{rudzis2022} For every sequence of cells $\{\Sigma_i\}_{i \in \mathbb{N}}$, there exists a sequence of positive numbers $\{b_i\}_{i \in \mathbb{N}}$ such that exactly one of the following statements is true:
\begin{enumerate}
\item[(A)] There exists a unique Markov kernel $\widetilde{K}(v,dv')$ on $\mathbb{S}^2_+$ such that for any sequence $\epsilon_i \to 0$ with $\epsilon_i \leq b_i$ for all $i$,
\begin{equation} \label{eq:roughlim}
\delta_{K^{\Sigma_i,\epsilon_i}(x,v)}(dx' dv') \to \delta_x(dx') \widetilde{K}(v, dv') \text{ in } \MK((\mathbb{R}^2 \times \mathbb{S}^2_+)^2, \widehat{\Lambda}^2).
\end{equation}
\item[(B)] For any sequence of $\epsilon_i$ with $\epsilon_i \leq b_i$, the limit of $\delta_{K^{\Sigma,\epsilon}(x,v)}(dx' dv')$ in $\MK((\mathbb{R}^2 \times \mathbb{S}^2_+)^2,\widehat{\Lambda}^2)$ does not exist.
\end{enumerate}
Moreover, if the sequence of cells $\Sigma_i = \Sigma$ is constant, then the first option (A) always holds with $b_i = \infty$ for all $i$.
\end{theorem}

We let $\widetilde{\mathcal{A}}$ denote the set of all Markov kernels $\widetilde{K}$ on $\mathbb{S}^2_+$ such that a limit of form  \eqref{eq:roughlim} holds, with $\epsilon_i < b_i$ for all $i \in \mathbb{N}$. Recall the spherical coordinates for $\mathbb{S}^2_+$ defined by \eqref{eq:spherecoord}. The next result characterizes the class of Markov kernels $\widetilde{\mathcal{A}}$, showing that $\widetilde{\mathcal{A}}$ is equal to $\mathcal{A}$, the class of Markov kernels introduced in \S\ref{ssec:ergodic}.

\begin{theorem} \label{thm:colchar}
\cite[Theorem 1.31]{rudzis2022} Let $\widetilde{K}$ be a Markov kernel in $\widetilde{\mathcal{A}}$. In the coordinates $(\theta,\psi)$, $\widetilde{K}$ takes the form
\begin{equation} \label{eq:roughcolform}
\widetilde{K}(\theta,\psi; d\theta' d\psi') = \widetilde{P}(\theta, d\theta') \delta_{\pi-\psi}(d\psi'), 
\end{equation}
where $\widetilde{P}(\theta, d\theta')$ is a Markov kernel on $(0,\pi)$ which has the following properties:
\begin{enumerate}
\item[(i)] The measure $\Lambda^1(d\theta) := \sin\theta d\theta$ is a reversible measure for $\widetilde{P}$.
\item[(ii)] Suppose $\{\Sigma_i\}_{i \in \mathbb{N}}$ is a sequence of cells such that $$\delta_{K^{\Sigma_i,\epsilon_i}(x,v)}(dx' dv') \to \delta_x(dx') \widetilde{K}(v, dv'),$$ for any sequence $\epsilon_i \leq b_i$, in the sense of Theorem \ref{thm:dichotomy}(A). For $i \in \mathbb{N}$, let  
\begin{equation} \label{eq:foreshortening}
\begin{split}
& \widetilde{\Sigma}_i := \{(x_1,x_2) \in \mathbb{S}^1 \times \mathbb{R} : (x_1, (1 + mJ^{-1})^{-1/2}x_2) \in \Sigma_i\} \\
& \widetilde{\epsilon_i} = (1 + mJ^{-1})^{-1/2}\epsilon_i.
\end{split}
\end{equation}
Then 
\begin{equation} \label{eq:roughlim2}
\delta_{P^{\widetilde{\Sigma}_i,\widetilde{\epsilon}_i}(x_1,\theta)}(dx_1' d\theta') \to \delta_{x_1}(dx_1')\widetilde{P}(\theta, d\theta') \text{ in } \MK((\mathbb{R} \times (0,\pi))^2, \widehat{\Lambda}^1)
\end{equation}
weakly in the space of measures on $(\mathbb{R} \times (0,\pi))^2$.
\end{enumerate}
Consequently, $\Lambda^2$ is a reversible measure for $\widetilde{K}$.

Conversely, if $\widetilde{K}$ is a Markov kernel on $\mathbb{R}^2 \times \mathbb{S}^2_+$ of form \eqref{eq:roughcolform}, where $\widetilde{P}(\theta,d\theta')$ is any Markov kernel on $(0,\pi)$ for which $\Lambda^1(d\theta)$ is a reversible measure, then $\widetilde{K}$ lies in the class $\widetilde{\mathcal{A}}$.
\end{theorem}

We refer to the transformation $\mathcal{F}$ taking a shape-scale pair $(\Sigma, \epsilon)$ to the shape-scale pair 
\begin{equation} \label{eq:Ftransformation}
\mathcal{F}(\Sigma,\epsilon) := (\widetilde{\Sigma}, \widetilde{\epsilon}),
\end{equation}
where $\widetilde{\Sigma}$ and $\widetilde{\epsilon}$ are given by \eqref{eq:foreshortening}, as the \textit{foreshortening transformation}. The transformation is clearly a bijective mapping from shape-scale pairs to shape-scale pairs. We will sometimes abuse notation, also writing $\mathcal{F}(\Sigma) = \widetilde{\Sigma}$ and $\mathcal{F}(\epsilon) = \widetilde{\epsilon}$. 

Somewhat informally, the theorem above says that the following diagram commutes.
\begin{equation} \label{eq:diag}
\begin{tikzcd}
\{(\Sigma_i,\epsilon_i)\}_{i \in \mathbb{N}} \arrow[r, "\mathcal{F}"] \arrow[d] & \{(\widetilde{\Sigma}_i, \widetilde{\epsilon_i})\}_{i \in \mathbb{N}} \arrow[l] \arrow[d] \\
\widetilde{K}(\theta,\psi; d\theta' d\psi') \arrow[r] & \widetilde{P}(\theta, d\theta') \arrow[l]
\end{tikzcd}.
\end{equation}
Here $\{(\Sigma_i,\epsilon_i)\}_{i \in \mathbb{N}}$ is a sequence of shape-scale pairs such that the statement (A) in Theorem \ref{thm:dichotomy} holds. The double arrow on the top represents the foreshortening transformation given by \eqref{eq:Ftransformation}, while the bottom double arrow represents the relation \eqref{eq:roughcolform}. The downward pointing arrows on the left and right, respectively, are realized by taking the limits \eqref{eq:roughlim} and \eqref{eq:roughlim2} as $i \to \infty$.

\subsection{Examples} \label{ssec:examples}

The following section illustrates how particular microstructures on the surface of the strip can give rise to ergodic dynamics. We begin by reviewing a general purpose lemma for constructing rough collision laws, proved in \cite{rudzis2022}. The examples presented in \S\ref{sssec:smandns}-\S\ref{sssec:focusing} were also discussed previously in \cite{rudzis2022} (although not in the context of ergodic theory). Consequently, we omit careful derivations of some of the formulas. Of these examples, the first few are not ergodic, while the remaining ones are. The last example, discussed in \S\ref{sssec:deltanubs} is new and merits a more detailed analysis. 

\subsubsection{General lemma for constructing rough collision laws}

Suppose the Markov kernel $K(v,dv')$ for the velocity process in $\hat{\mathbb{S}^2}$ is determined by some sequence of shape-scale pairs $\{(\Sigma_i,\epsilon_i)\}_{i \in \mathbb{N}}$. In view of Theorem \ref{thm:ergcol}, Theorem \ref{thm:colchar}, and \eqref{eq:diag}, to prove that $\Lambda^2$ is (resp. is not) ergodic with respect to $K$, it is sufficient to carry out the following strategy:
\begin{enumerate}
\item[1.] Find a description of the limiting reflection law $\widetilde{P}(\theta, d\theta')$ associated with the sequence $\{(\widetilde{\Sigma}_i,\widetilde{\epsilon}_i)\}_{i \in \mathbb{N}}$ obtained through the foreshortening transformation \eqref{eq:Ftransformation}; and 
\item[2.] Prove that $\Lambda^1$ is (resp. is not) an ergodic measure for $\widetilde{P}(\theta, d\theta')$.
\end{enumerate}

Step 1 may be achieved using the next lemma, proved previously in \cite{rudzis2022}. Let a shape-scale pair $(\widetilde{\Sigma},\widetilde{\epsilon})$ be fixed. Given $(x,\theta) \in \mathbb{R} \times (0,\pi)$, suppose $(x',\theta')$ is the random variable in $\mathbb{R} \times (0,\pi)$ with law $P^{\widetilde{\Sigma},\widetilde{\epsilon}}(x,\theta, d x' d \theta')$. We define $\widetilde{P}^{\widetilde{\Sigma},\widetilde{\epsilon}}(\theta,d\theta')$ to be the law of $\theta'$ as $x$ varies uniformly in the period $[0,\widetilde{\epsilon}]$. In other words, for any $f \in C_c((0,\pi))$,
\begin{equation} \label{eq:spaceaverage}
    \int_{(0,\pi)} f(\theta') \widetilde{P}^{\widetilde{\Sigma},\widetilde{\epsilon}}(\theta,\dd\theta') := \frac{1}{\widetilde{\epsilon}}\int_{[0,\widetilde{\epsilon}]} \int_{\mathbb{R} \times (0,\pi)} f(\theta')\mathbb{P}^{\widetilde{\Sigma},\widetilde{\epsilon}}(x,\theta, \dd x' \dd\theta') \dd x.
\end{equation}

\begin{lemma}\cite[Lemma 2.1]{rudzis2022} \label{lem:construct}
\begin{enumerate}
\item[(i)] If the sequence $\{P^{\widetilde{\Sigma}_i,\widetilde{\epsilon}_i}\}_{i \in \mathbb{N}}$ converges in $\MK(\mathbb{R} \times (0,\pi), \widehat{\Lambda}^1)$, then the sequence $\{\widetilde{P}^{\widetilde{\Sigma}_i,\widetilde{\epsilon}_i}\}_{i \in \mathbb{N}}$ converges in $\MK((0,\pi),\Lambda^1)$. 

\item[(ii)] If the sequence $\{\widetilde{P}^{\widetilde{\Sigma}_i,\widetilde{\epsilon}_i}\}_{i \in \mathbb{N}}$ converges in $\MK((0,\pi),\Lambda^1)$ to some Markov kernel $\widetilde{P}$, then the sequence $\{P^{\widetilde{\Sigma}_i,\widetilde{\epsilon}_i}\}_{i \in \mathbb{N}}$ converges in $\MK(\mathbb{R} \times (0,\pi), \widehat{\Lambda}^1)$ to $\id \times \widetilde{P}$.

\item[(iii)] Suppose that $\widetilde{\Sigma}_i = \widetilde{\Sigma}$ is constant. Then for any sequence $\widetilde{\epsilon}_i \to 0$, the sequence $\{P^{\widetilde{\Sigma}, \widetilde{\epsilon}_i}\}_{i \in \mathbb{N}}$ converges in $\MK(\mathbb{R} \times (0,\pi), \widehat{\Lambda}^1)$ to $\id \times \widetilde{P}$, where $\widetilde{P}$ satisfies
\begin{equation} \label{eq:unifstart}
    \widetilde{P}(\theta, \dd\theta') = \widetilde{P}^{\widetilde{\Sigma}, 1}(\theta, \dd\theta').
\end{equation}
\end{enumerate}
\end{lemma}

\subsubsection{Smooth and no-slip collisions} \label{sssec:smandns} The Markov kernel $K$ is determined by the rough collision law $K_+$, as in \eqref{eq:twosides}. The two simplest examples of rough collision laws are the \textit{smooth} (or classical) collision law 
$$
K_+^\text{sm}(\theta,\psi; d\theta' d\psi') := \delta_{(\pi - \theta, \pi - \psi)}(d\theta' d\psi')
$$
and the so-called \textit{no-slip} collision law 
$$
K_+^{\text{ns}}(\theta, \psi; d\theta' d\psi') := \delta_{(\theta, \pi - \psi)}(d\theta' d\psi').
$$
Both describe collisions which are deterministic. The first of these describes collisions in which the wall is completely smooth, i.e. the impact direction is normal to the surface of the wall. Trivially, the sequence of shape-scale pairs $(\Sigma_i,\epsilon_i)$ where $\Sigma_i = \Sigma^{\text{sm}} := \mathbb{S}^1 \times (-\infty,0]$, and $\epsilon_i$ is a sequence of positive numbers converging to zero, gives rise to such collisions.
The reflection law factor in the decomposition \eqref{eq:roughcolform} is the classical \textit{specular} (angle of incidence equals angle of reflection) reflection law $\widetilde{P}^{\text{sp}}(\theta, d\theta') := \delta_{\pi - \theta}(d\theta')$, describing the reflection of a point-particle from a smooth surface.

The no-slip collision may be viewed as a kind of idealized frictional collision. Early studies of no-slip collisions appear in the work of F. B. Pidduck \cite{pidduck1922}, R. L. Garwin \cite{garwin1969}, and D. S. Broomhead and E. Gutkin \cite{broomheadgutkin1993}. This type of collision dynamics has received more recent attention as well, from both applied \cite{hefner2004superball, cross2005, taveres2007, harrowellwang2021} and theoretical \cite{coxferesward2016diffgeo, coxfereszhang2018stability, CCCF2020cylinder, coxfereszhao2021rolling, ahmedcoxwang2022vmass} perspectives. In the decomposition \eqref{eq:roughcolform} of $K^{\text{ns}}_+$, the corresponding reflection law factor is simply the identity Markov kernel $\widetilde{P}^{\text{rr}}(\theta, d\theta') = \delta_\theta(d\theta')$, which we also refer to as the \textit{retroreflection} law. Retroreflections do occur in nature, at least to a high degree of approximation, with examples including cat's eyes and street signs with reflective paint \cite{retrorefWikipedia}. Unlike smooth collisions, it is far from trivial to describe a sequence of microstructures which gives rise to no-slip collisions. Nonetheless, explicit descriptions of periodic microstructures giving rise to retroreflection are obtained in \cite[Chapter 9]{plakhov2012}. By applying the inverse of the foreshortening transformation \eqref{eq:Ftransformation} to such a microstructure, one may in principle obtain a microstructure which gives rise to no-slip collisions.

\begin{figure}
    \centering
    \includegraphics[width = 0.48\linewidth]{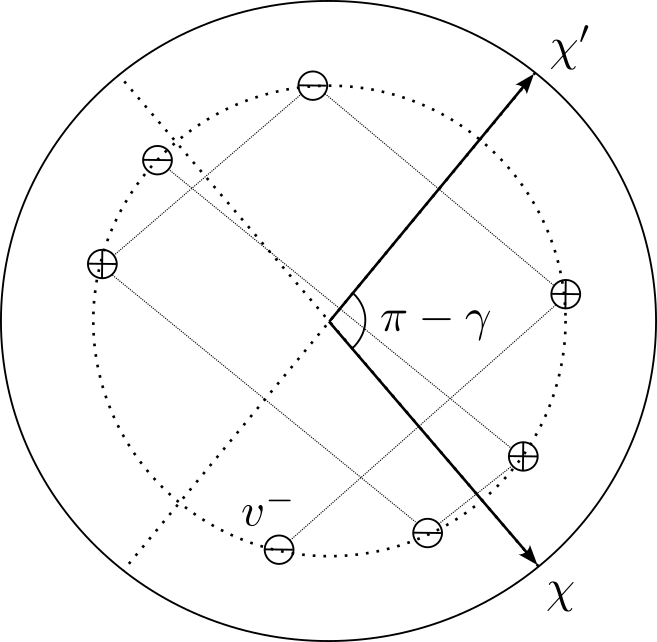}
    \vspace{0.3in}
    \caption{Above is shown the kinetic energy unit sphere $\mathbb{S}^2$, oriented so that the unit vector $n_2 = (0,0,1)$, points ``out of the board.'' The vectors $\chi = \sqrt{2(J+m)^{-1}}(1,-1,0)$ and $\chi' = R(\chi) = \sqrt{2(J+m)^{-1}}(1,1,0)$ determine the locally conserved quantities \eqref{eq:conserved1} and \eqref{eq:conserved2}, respectively. The angle between $\chi$ and $\chi'$, with respect to the inner product $\langle \cdot, \cdot \rangle$, is $\pi - \gamma$, where $\gamma$ is defined by \eqref{eq:gammadef}. The velocity process alternates between points in the hemispheres $\mathbb{S}^2_-$ and $\mathbb{S}^2_+$, represented above by the points labeled $\ominus$ and $\oplus$, respectively. Under the no-slip collision law, the projected images of the velocities alternate between reflections through the vector $\chi$ and through the vector $\chi'$. Consequently, if $v^-$ is the initial velocity, the subsequent velocities will be concentrated in the pair of circles $\mathcal{C}_{v^-} := \{v \in \mathbb{S}^2 : \langle v, n_2\rangle = \pm \langle v^-,n_2\rangle\}$ for all time. If $\gamma/\pi$ is irrational, then uniform measure on $\mathcal{C}_{v^-}$ is ergodic for any $v^- \in \hat{\mathbb{S}}^2$, whereas if $\gamma/\pi$ is rational, then any ergodic measure will be uniform over a finite subset of $\mathcal{C}_{v^-}$.} 
    \label{fig:noslip}
\end{figure}

Let $K^{\text{sm}}$ and $K^{\text{ns}}$ denote the Markov kernels for the velocity processes determined by the collision laws $K_+^{\text{sm}}$ and $K_+^{\text{ns}}$, respectively. The measure $\Lambda^2$ is ergodic neither for the dynamics determined by $K^{\text{sm}}$ nor $K^{\text{ns}}$. Both these collision laws do not satisfy the hypothesis of Theorem \ref{thm:ergcol}, insofar as the reflection factor $\widetilde{P}$ is not ergodic. Passing from the coordinate representation back to expressions in terms of velocities and applying \eqref{eq:twosides}, it is not hard to show that Markov kernels are given by 
$$
K^{\text{sm}}(v, dv') = v - 2\langle v, n_2 \rangle n_2, \quad \quad \quad K^{\text{ns}}(v, dv') = \begin{cases} -v + 2\langle v, \chi \rangle \chi & \text{ if } v_2 < 0, \\
-v + 2\langle v, \chi' \rangle \chi' & \text{ if } v_2 > 0,
\end{cases}
$$
where $\chi$ and $n_2$ are given by \eqref{eq:orthonormal}, $\chi' = R(\chi)$, and $R$ is the 180 degree rotation defined by \eqref{eq:180degree}. If the initial velocity of the system is $v^-$, then the subsequent velocities under the dynamics determined by the corresponding Markov kernel $K^{\text{sm}}$ will be concentrated in two-point set $\{v^-, v^- - 2\langle v^-, n_2 \rangle n_2\}$ for all time. Under $K^{\text{ns}}$, the dynamical evolution is equivalent to the sequence of reflections, described in Figure \ref{fig:noslip}.  For all time the velocities will be concentrated in the pair of circles 
$$
\mathcal{C}_{v^-} := \{v \in \mathbb{S}^2 : \langle v,n_2\rangle = \pm \langle v^-,n_2\rangle \}.
$$ 
If $\gamma/\pi$ is irrational, then uniform measure on $\mathcal{C}_{v^-}$ is ergodic for any $v^- \in \hat{\mathbb{S}}^2$, whereas if $\gamma/\pi$ is rational, then any ergodic measure will be uniform over a finite subset of $\mathcal{C}_{v^-}$.

\subsubsection{Rectangular teeth} \label{sssec:rec} Consider the sequence of walls obtained by putting $\epsilon$-periodic rectangular notches along the boundary of the half-space $x_2 \leq 0$. More precisely, consider periodic walls of form 
$$
\widetilde{W}_i = W(\widetilde{\Sigma}, \widetilde{\epsilon}_i) = \{(x_1,x_2) \in \mathbb{R}^2 : x_2 \leq t_i(x_1)\},
$$
Here for each $i \in \mathbb{N}$ the ``tooth function'' is given by
\begin{equation}
    t_i(x) = \begin{cases}
    0 & \text{ if } 2k \widetilde{\epsilon}_i \leq x \leq (2k + 1) \widetilde{\epsilon}_i, \\
    -r\widetilde{\epsilon}_i & \text{ if } (2k + 1)\widetilde{\epsilon}_i < x < (2k+2)\widetilde{\epsilon}_i 
    \end{cases} \quad \text{ for } k \in \mathbb{Z},
\end{equation}
where $r > 0$ is the ratio of a height of a tooth to its width and the $\widetilde{\epsilon}_i$ are positive and converge to zero.

The limiting collision law will be some random selection between specular and no-slip collisions. By Lemma \ref{lem:construct}, to determine the limiting reflection law $P_r(\theta,d\theta')$ it is sufficient to compute $\widetilde{P}^{\widetilde{\Sigma},1}(\theta, d\theta')$, which is a matter of trigonometry. Figure \ref{fig:microstructures}(B) suggests how to compute this reflection law. We obtain the following formula: 
$$
P_r(\theta, d\theta') = p_r(\theta) \delta_{\pi - \theta}(d\theta') + (1 - p_r(\theta))\delta_{\theta}(d\theta'),
$$
where 
$$
p_r(\theta) = \begin{cases}
\frac{1}{2} + \frac{1}{2}\{2r|\cot\theta|\} & \text{ if } \lfloor 2 r |\cot\theta| \rfloor \text{ is even}, \\
1 - \frac{1}{2}\{2r|\cot\theta|\} & \text{ if } \lfloor 2 r |\cot\theta| \rfloor \text{ is odd}.
\end{cases}
$$
We remark that the foreshortening transformation \eqref{eq:Ftransformation} leaves the class of walls consisting of rectangular teeth with parameter $r$ invariant, only resulting in a change of the parameter $r$. It follows that the corresponding collision law produced by the rough teeth microstructure is 
$$
K^r_+(\theta,\psi; d\theta' d\psi') := \delta_{\pi - \psi}(d\psi')P_{\hat{r}}(\theta, d\theta'), 
$$
where $\hat{r} = (1 + m/J)^{1/2}r$. Let $K^r$ denote the Markov kernel on $\hat{\mathbb{S}}^2$ determined by \eqref{eq:twosides} with $K_+ = K_+^r$. By observing that $K_r^+(\theta,\psi; \cdot) = p_{\hat{r}}(\theta)K^{\text{sm}}_+(\theta,\psi; \cdot) + (1 - p_{\hat{r}}(\theta))K^{\text{ns}}_+(\theta,\psi; \cdot)$, a simple argument shows that if $\gamma/\pi$ is irrational, then uniform measure on the sets $\mathcal{C}_v$, for $v \in \mathbb{S}^2_+$ will be an ergodic measure for the resulting dynamics.

\begin{figure}
    \centering
    \includegraphics[width = 0.6\linewidth]{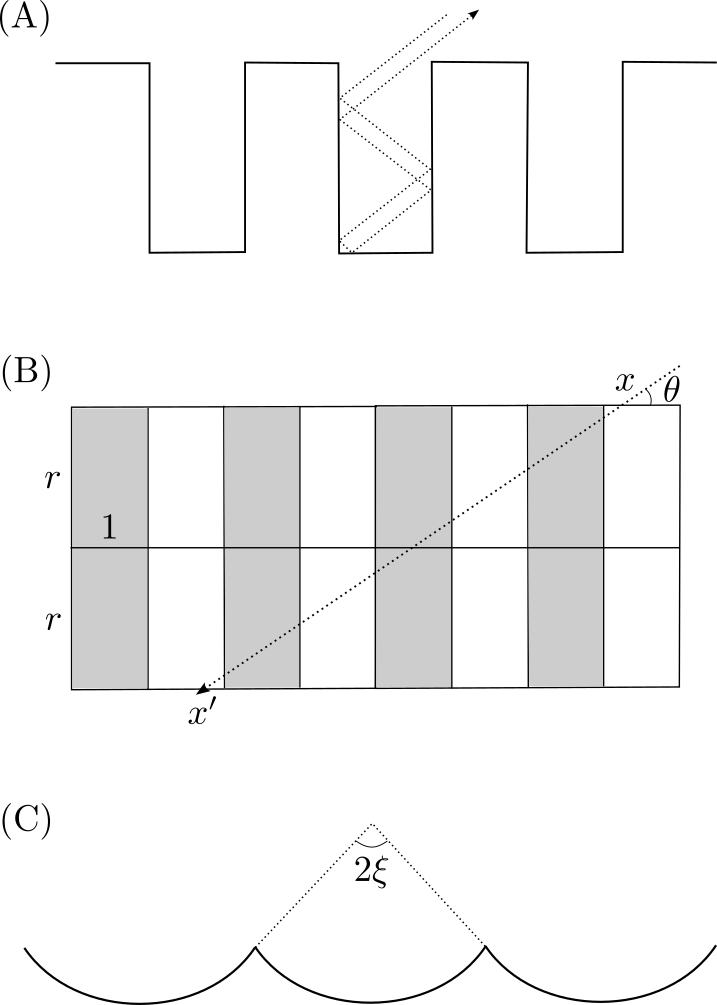}

    \vspace{0.4in}
    
    \caption{(A) Rectangular teeth microstructure. (B) Derivation of the reflection law for the rectangular teeth microstructure: Starting from the initial condition $(x,\theta)$, if the unfolded trajectory between two teeth ends at a point $x'$ in a white region, then the angle of exit $\theta'$ will equal $\pi - \theta$. Otherwise, if $x'$ lies in a gray region, then $\theta' = \theta$. (C) Focusing circular arc microstructure.} 
    \label{fig:microstructures}
\end{figure}

\subsubsection{Focusing circular and elliptical arcs.} \label{sssec:focusing} First let us describe point-particle reflections from a microstructure formed from a periodic pattern of focusing circular arcs. Each arc is focusing (i.e. concave up), and subtends an angle of $2\xi$, as in Figure \ref{fig:microstructures}(C). The derivation of the reflection law $P = P^\text{ca}$ associated with this microstructure is elementary, but somewhat involved. We refer the reader to \cite[\S 2.3.3]{rudzis2022} for details, and give the formula for $P^\text{ca}$ below. Let
$$
\text{Arccos}(x) := \begin{cases}
    2\pi - \arccos(x) & \text{ if } -1 \leq x < 0, \\
    -\arccos(x) & \text{ if } 0 \leq x \leq 1,
\end{cases}
$$
where $\arccos(x)$ is the usual inverse cosine function taking values in the interval $[0,\pi]$.
For each $\theta \in (0,\pi)$, we define the random variable 
$$
\Theta'_\theta := \begin{cases}
    \theta + \lceil \frac{\xi - \Xi_\theta}{2|\pi + \Xi_\theta - \theta|} \rceil 2(\pi + \Xi_\theta - \theta) & \text{ if } \theta \geq \frac{\pi}{2} + \Xi_\theta, \\
    \theta + \lceil \frac{\xi + \Xi_\theta}{2|\theta - \Xi_\theta|} \rceil 2(\theta - \Xi_\theta) & \text{ if } \theta < \frac{\pi}{2} + \Xi_\theta,
    \end{cases}
$$
where 
$$
\Xi_\theta := \text{Arccos}\left(X\cos(\theta - \xi) + (1 - X)\cos(\theta + \xi) \right),
$$
and $X$ is uniform in $[-1,1]$. For each $\theta \in (0,\pi)$, $P^{\text{ca}}(\theta, d\theta')$ is the law of $\Theta'_\theta$. 

To obtain a corresponding collision law, define a periodic microstructure consisting of elliptical arcs, such that after applying the foreshortening transformation \eqref{eq:Ftransformation}, we obtain the microstructure consisting of circular arcs, defined above. More precisely, consider a sequence of positive numbers $\epsilon_n \to 0$ and an $\epsilon_n$-periodic function $t_n$ defined by  
$$
t_n(x) = \frac{\epsilon_n}{2}\cot\xi - (1+m/J)^{-1/2}\sqrt{\frac{\epsilon_n^2}{4}\csc^2\xi - \left(x - \frac{\epsilon_n}{2}\right)^2} \text{ for } x \in [0, \epsilon_n],
$$
and extended periodically to all $x \in \mathbb{R}$. Then the resulting collision law is given by 
$$
K^{\text{ea}}_+(\theta,\psi; d\theta' d\psi') = P^{\text{ca}}(\theta, d\theta') \delta_{\pi - \psi}(d\psi').
$$
We let $K^{\text{ea}}(\theta, \psi; d\theta' d\psi')$ be defined by \eqref{eq:twosides} with $K_+ = K_+^{\text{ea}}$. Since $P^{\text{ca}}(\theta, d\theta')$ is nonsingular, one may show using Theorem \ref{thm:ergcol} that $\Lambda^2$ is ergodic with respect to the dynamics induced by $K^{\text{ea}}$.

\subsubsection{Cells with $\delta$-nubs} \label{sssec:deltanubs} Given any periodic microstructure, by inserting small dispersive arcs (the ``nubs'') into the boundary of the microstructure, one obtains collision dynamics which are ergodic. Moreover, up to a small error, the collision dynamics agree with those determined by the original microstructure (see Theorem \ref{thm:refnub} and Corollary \ref{cor:refnub}).

To describe this construction precisely, fix $\delta > 0$ and let $h : [-\delta, \delta] \to \mathbb{R}$ be a smooth, strictly concave-down function such that, for some constant $c > 0$, $h(-\delta) = h(\delta) \geq -c\delta$ and the range of $h$ is $[h(\delta),0]$. Here, by ``smooth,'' we mean that $h$ extends to some infinitely differentiable function on some open interval containing $[-\delta,\delta]$. Let $\Gamma_h$ denote the graph of $h$. We refer to $\Gamma_h$ as the \textit{nub shape}.

Also let $n_h = (n_h^1, n_h^2) : [-\delta,\delta] \to \mathbb{S}^1 \subset \mathbb{R}^2$ denote the unit normal to the graph of $h$. Note that $n_h$ is smooth, and the first component $n_h^1$ is strictly decreasing.

Fix a cell $\Sigma$, and consider the 1-periodic wall $W(\Sigma,1)$. To avoid certain technicalities, we will assume that the boundary of $W(\Sigma,1)$ consists of a single piecewise $C^2$-curve. (Thus, the wall itself must be connected.) We will also assume without loss of generality that the origin $(0,0)$ is a point of maximal height on the wall. We construct a new 1-periodic wall from the old as follows. For $k \in \mathbb{Z}$, let the portion of the boundary of the new wall below the interval $[-\delta + k, \delta + k]$ be given by the arc $\Gamma_h^{(k)} := \Gamma_h + k e_1 = \{(x_1,x_2) : (x_1 - k, x_2) \in \Gamma_h\}$. Then join the endpoints of each pair of adjacent arcs $\Gamma_h^{(k)}$ and $\Gamma_h^{(k+1)}$ by a similar copy of the portion of the boundary of $W$ which runs between the points $(0,0)$ and $(1,0)$, suitably translated and rescaled, as in Figure \ref{fig:nub1}. The resulting curve forms the boundary of a new wall $W^\delta \subset \mathbb{R} \times (-\infty,0]$. Since the wall is 1-periodic and piecewise smooth, there is a unique cell $\Sigma^\delta$ such that $W^\delta = W(\Sigma^\delta,1)$.

\begin{remark} \label{rem:shortnub} \normalfont
For the purpose of applying Theorem \ref{thm:colchar}, an important property of the above construction is that the foreshortening transformation $\mathcal{F}$, defined by \eqref{eq:Ftransformation}, preserves the class of cells with $\delta$-nubs. In fact, it is easy to check that the following diagram commutes:
\begin{equation} \label{eq:nubdiag}
\begin{tikzcd}
(\Sigma,\epsilon) \arrow[r, "\mathcal{F}"] \arrow[d, "h"] & (\widetilde{\Sigma}, \widetilde{\epsilon}) \arrow[d, "\widetilde{h}"] \\
(\Sigma^\delta,\epsilon) \arrow[r, "\mathcal{F}"] & (\widetilde{\Sigma}^\delta, \widetilde{\epsilon}) 
\end{tikzcd},
\end{equation}
where the downward pointing arrow on the left represents attaching a $\delta$-nub with nub shape given by $h$, and the downward pointing arrow on the right represents attaching a $\delta$-nub with nub shape given by $\widetilde{h} := (1 + mJ^{-1})^{1/2}h$. That is, if we start with some shape-scale pair $(\Sigma,\epsilon)$ and attach a $\delta$-nub with shape given by $h$, as in the construction above, and then apply the foreshortening transformation $\mathcal{F}$, the resulting shape-scale pair is the same as if we first apply $\mathcal{F}$ and then attach a $\delta$-nub with shape given by $\widetilde{h}$.

\end{remark}

\begin{figure}
    \centering
    \includegraphics[width = \linewidth]{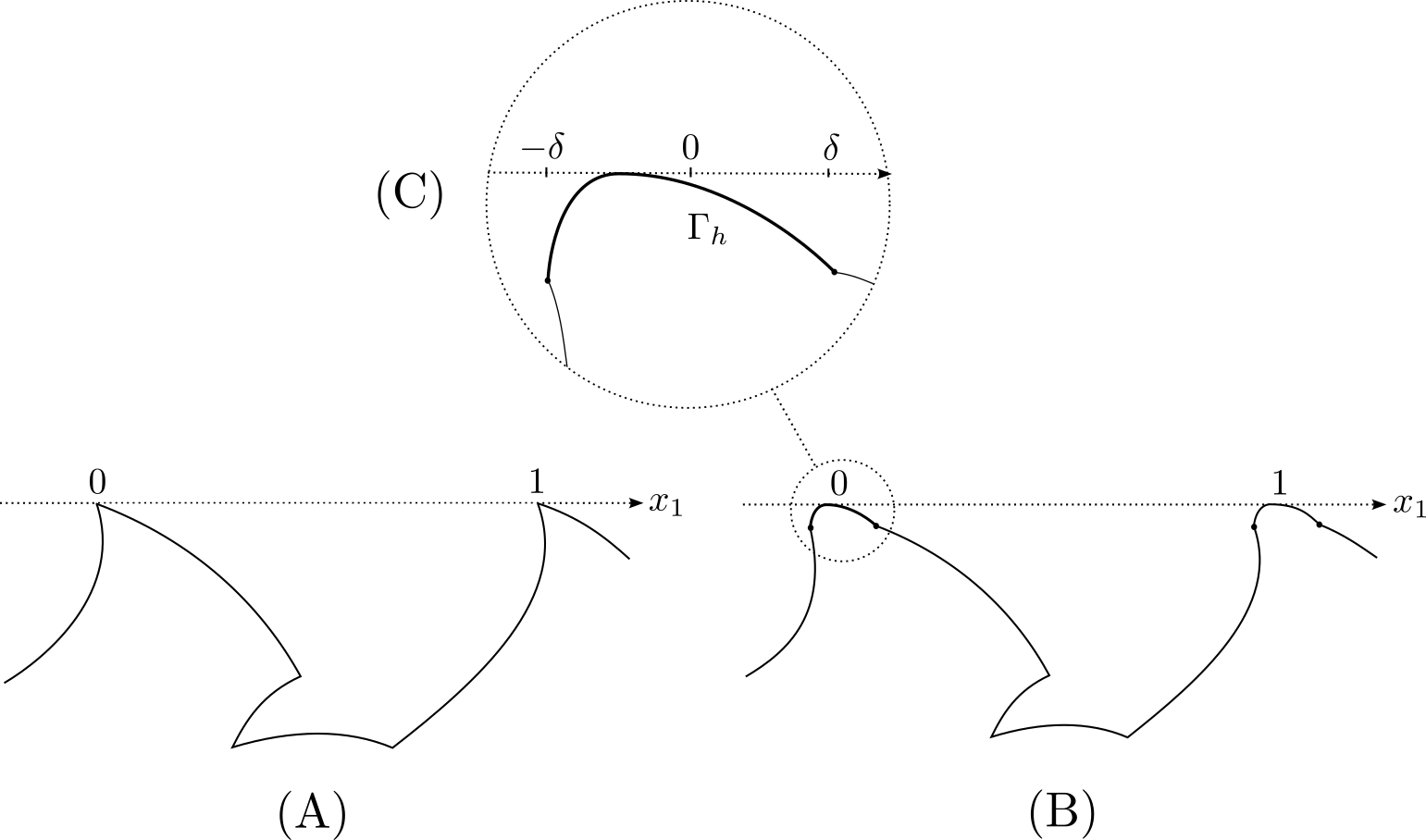}
    \caption{(A) Original periodic microstructure $W$. (B) New microstructure $W^\delta$. (C) Detail of a $\delta$-nub.}
    \label{fig:nub1}
\end{figure}

Recall from Lemma \ref{lem:construct} that if $(\Sigma_i,\epsilon_i)$ is a sequence of shape-scale pairs where $\Sigma_i = \Sigma$ is constant, then the limiting reflection law always exists and is given by $\widetilde{P}^{\Sigma, 1}(\theta, d\theta')$, defined by \eqref{eq:unifstart}. Thus the following theorem establishes ergodicity for reflection laws obtained from the constant sequence of shapes $\Sigma_i = \Sigma^\delta$.

\begin{theorem} \label{thm:refnub}
Put $P(\theta, d\theta') = \widetilde{P}^{\Sigma, 1}(\theta, d\theta')$ and $P^\delta(\theta, d\theta') = \widetilde{P}^{\Sigma^\delta, 1}(\theta, d\theta')$.

(i) $\Lambda^1(d\theta)$ is an ergodic measure for $P^\delta(\theta, d\theta')$.

(ii) There exist constants $c, \delta_0 > 0$ and a subset $\widetilde{A} \subset (0,\pi)$ with $|(0,\pi) \smallsetminus \widetilde{A}| \leq c\delta^{2/9}$ such that, for all $\delta \in (0,\delta_0)$, $\theta \in \widetilde{A}$, and $B \subset (0,\pi)$,
$$
|P(\theta,B) - P^\delta(\theta,B)| \leq c\delta^{2/9}.
$$
\end{theorem}

\begin{remark} \normalfont
Part (i) of Theorem \ref{thm:refnub} is quite similar to \cite[Proposition 5.6]{feres2007RW}, which says that in a random billiard system similar to ours, if the microstructure contains a point of maximal height at which the curvature is positive and nonzero, then the Markov chain is ergodic. Part (ii) of Theorem \ref{thm:refnub} is, as far as we know, new.
\end{remark}

Let $K_+$ and $K_+^\delta$ be the collision laws determined by the sequences of shape-scale pairs $\{(\Sigma,\epsilon_i)\}_{i \in \mathbb{N}}$ and $\{(\Sigma^\delta,\epsilon_i)\}_{i \in \mathbb{N}}$, respectively. In view of Remark \ref{rem:shortnub}, the foreshortening transformation takes walls with $\delta$-nubs to walls with $\delta$-nubs. Consequently, the reflection factor in the decomposition \eqref{eq:roughcolform} for $K_+^\delta$ will be of form $\widetilde{P}^{\widetilde{\Sigma}^{\delta},1}$. From this, we get the following corollary.

\begin{corollary} \label{cor:refnub}
Assume that $\gamma/\pi$ is irrational. Given a cell $\Sigma$, consider the sequences of shape-scale pairs $\{(\Sigma_i, \epsilon_i)\}_{i \in \mathbb{N}}$ and $\{(\Sigma'_i, \epsilon'_i)\}_{i \in \mathbb{N}}$, where $\Sigma_i = \Sigma$ and $\Sigma_i' = \Sigma^\delta$ for all $i \in \mathbb{N}$. Let $K_+$ and $K_+^\delta$ be the rough collision laws given rise to by the first and second sequence respectively, and let $K$ and $K^\delta$ denote the corresponding Markov kernels on $\hat{\mathbb{S}}^2$ given by \eqref{eq:twosides}. Then

(i) $\Lambda^2$ is the unique ergodic measure for $K^\delta$, and 

(ii) There exists constants $c, \delta_0 > 0$ and a subset $\widetilde{R} \subset \mathbb{S}^2_+$ with $\sigma^2(\mathbb{S}^2_+ \smallsetminus \widetilde{R}) \leq c\delta^{2/9}$ such that, for all $\delta \in (0,\delta_0)$, $\theta \in \widetilde{R}$, and $B \subset \mathbb{S}^2_+$, 
$$
|K_+(v,B) - K_+^\delta(v,B)| \leq c\delta^{2/9}.
$$
\end{corollary}

We prove Theorem \ref{thm:refnub} and Corollary \ref{cor:refnub} in \S\ref{ssec:refnub}.

\subsection{Proof of Lemma \ref{lem:iso}} \label{ssec:isoproof}

Recall the definitions of the maps $R$, $\Phi$, and $H$, from \S\ref{ssec:ergodic}.

\begin{proof}[Proof of Lemma \ref{lem:iso}]
\textit{Step 1. We show that $R_*K = K$.} To see this, let $f$ be any test function on $\hat{\mathbb{S}}^2$, and consider the following calculation 
\[
\begin{split}
& \int_{\hat{\mathbb{S}}^2} f(v')R_*K(v, dv') = \int_{\hat{\mathbb{S}}^2} f(R(v'))K(R^{-1}(v), dv') \\
& = \int_{\mathbb{S}^2_+} f(R(v'))K(R^{-1}(v), dv') + \int_{\mathbb{S}^2_-} f(R(v'))K(R^{-1}(v), dv') \\
& = \int_{\mathbb{S}^2_+} f(R(v'))K_+(-R^{-1}(v), dv') + \int_{\mathbb{S}^2_-} f(v')K_+(-v, dv') \\
& = \int_{\hat{\mathbb{S}}^2_+} f(v')K(v, dv'),
\end{split} 
\]
where the last two lines use \eqref{eq:twosides} and the fact that $R$ is a linear involution.

Next, define $\widetilde{R} : E^* \to E^*$ by $\widetilde{R} = \Phi^{-1} \circ R \circ \Phi$. Note that since $H = \Phi^{-1}$, it is immediate that 
\begin{equation} \label{eq:commute}
H \circ R = \widetilde{R} \circ H. 
\end{equation}

\textit{Step 2. We will show that $\widetilde{R}_*Q = Q$.} To see this, first observe that for any $(u,v,s) \in E$, 
$$
\widetilde{R}(u,v,s) = (v,u,-s).
$$ 
This follows from the definition of $\Phi$ and involutivity of $R$:
$$
R \circ \Phi(u,v,s) = u n_1 + v R(n_1) - s(1 - u^2 - v^2)^{1/2} n_2 = \Phi(v, u, -s),
$$
noting also that $R$ maps $n_2$ to $-n_2$. Consequently, for any test function $f$, 
\[
\begin{split}
& \int_{E^*} f(u',v',s')\widetilde{R}_*Q(u,v,s; du' dv' ds') \\
& = \int_{E^*} f(\widetilde{R}(u',v',s')) Q(\widetilde{R}^{-1}(u,v,s); du' dv' s') \\
& = \int_{E^*} f(v',u',-s') Q(v,u,-s; du' dv' s') \\
& = \begin{cases}
\int_{(-1,1)} f(u,\ell_u(x'),-1)\hat{q}(\ell_u^{-1}(x),dx') & \text{ if } s = 1, \\
\int_{(-1,1)} f(\ell_v(x'),v,1)\hat{q}(\ell_v^{-1}(x),dx') & \text{ if } s = -1,
\end{cases} \\
& = \int_{E^*} f(u',v',s')Q(u,v,s; du' dv' ds'),
\end{split}
\]
where the last two equalities follow by \eqref{eq:Qdef}, which completes Step 2.

\textit{Completion of proof of Lemma.} We claim that to prove the lemma, it is sufficient to show that, for all $(u,v,s) \in E^-$,
\begin{equation} \label{eq:pushtoprove}
H_*K(u,v,s; du' dv' ds') = Q(u,v,s; du' dv' ds').
\end{equation}
Indeed, once we have \eqref{eq:pushtoprove}, by Steps 1 and 2, for any $(u,v,s) \in E^+$ and measurable set $A \subset E^*$, 
\[
\begin{split}
H_*K(u,v,s; A) & = H_* R_* K(u,v,s; A) \\
& = K((H \circ R)^{-1}(u,v,s); (H \circ R)^{-1}(A)) \\
& = K((\widetilde{R} \circ H)^{-1}(u,v,s); (\widetilde{R} \circ H)^{-1}(A)) \\
& = H_*K(\widetilde{R}^{-1}(u,v,s); \widetilde{R}^{-1}(A)) \\
& = Q(\widetilde{R}^{-1}(u,v,s); \widetilde{R}^{-1}(A)) \\
& = \widetilde{R}_*Q(u,v,s; A) = Q(u,v,s; A),
\end{split}
\]
where we have used \eqref{eq:commute} in the third line.

To prove \eqref{eq:pushtoprove}, we first compute the spherical coordinate representations of $K$ and $H$. Let $\widehat{H}(\theta,\psi) = H \circ \widetilde{G}(\theta,\psi)$, and let $\widehat{K}(\theta,\psi; d\theta' d\psi') = \widetilde{G}_*K(\theta,\psi; d\theta' d\psi')$. 


To obtain a formula for $\widehat{H}$, suppose that $(u,v,s) = \widehat{H}(\theta,\psi) \in E^*$, where $(\theta,\psi) \in (0,2\pi) \times (0,\pi)$. Then 
\begin{equation} \label{eq:tosolve}
\Phi(u,v,s) = \widetilde{G}(\theta,\psi).
\end{equation}
By definition of $\Phi$ and orthonormality of the basis $(\chi,n_1,n_2)$,
\[
\begin{split}
\Phi(u,v,s) & = \langle \chi, \Phi(u,v,s) \rangle \chi + \langle n_1, \Phi(u,v,s) \rangle n_1 + \langle n_2, \Phi(u,v,s) \rangle n_2 \\
& = u \langle R(n_1), \chi \rangle \chi + (u \langle R(n_1), n_1 \rangle + v)n_1 + s(1 - u^2 - v^2)^{1/2} n_2 \\
& = (u\sin\gamma) \chi + (u \cos\gamma + v)n_1 + s(1 - u^2 - v^2)^{1/2} n_2.
\end{split}
\]
But in view of the definition of $\widetilde{G}$, \eqref{eq:tosolve} holds if and only if $$
\cos\psi = u\sin\gamma, \quad \cos\theta\sin\psi = u \cos\gamma + v, \quad \text{ and } s = \sgn(\sin\theta);
$$
or equivalently, 
$$
u = \csc\gamma\cos\psi, \quad v = \cos\theta\sin\psi - \cos\psi\cot\gamma = \ell_u(\cos\theta), \quad s = \sgn(\sin\theta),
$$
where $\ell_u$ is given by the formula \eqref{eq:Sform}. We have proved:
\begin{equation} \label{eq:Hrep}
\widehat{H}(\theta,\psi) = \left( \csc\gamma\cos\psi, \ell_{\csc\gamma\cos\psi}(\cos\theta), \sgn(\sin\theta)\right).
\end{equation}

Note that $H$ maps $E^-$ onto $\mathbb{S}^2_-$, and $\mathbb{S}^2_-$ coincides with $\{(\theta,\psi) : \pi < \theta < 2\pi, 0 < \psi < \pi\}$ in the spherical coordinates. Thus, it is possible to invert $\widehat{H}$ in \eqref{eq:Hrep} to obtain the formula,
\begin{equation} \label{eq:Hinvrep}
\widehat{H}^{-1}(u,v,-1) = \left(2\pi - \arccos(\ell_u^{-1}(v)), \arccos(u\sin\gamma) \right), \quad \text{ for all } (u,v) \in E.
\end{equation}

We may also obtain a formula for $\widehat{K}$ restricted to $(\pi,2\pi) \times (0,\pi)$. Observe that the negation map $v \mapsto -v$ from $\mathbb{S}^2_-$ to $\mathbb{S}^2_+$ is given in coordinates by $(\theta,\psi) \mapsto (\theta - \pi, \pi - \psi)$. Using this, as well as \eqref{eq:twosides}
and \eqref{eq:roughcolform0}, we have
\begin{equation} \label{eq:Kpmrep}
\begin{split}
\widehat{K}(\theta,\psi; d\theta' d\psi') & = K_+(\theta - \pi, \pi - \psi; d\theta' d\psi') \\
& = P(\theta - \pi, d\theta')\delta_{\psi}(d\psi'), \quad \text{ for all } (\theta,\psi) \in (\pi,2\pi) \times (0,\pi).
\end{split}
\end{equation}

To finish the argument, we make the following calculation. Let $(u,v,-1) \in E^-$, and let $f$ be any test function on $E^*$. We compute 
\[
\begin{split}
& \int_{E^*} f(u',v',s')\widehat{H}_*\widehat{K}(u,v,-1; du' dv' ds') \\
& = \int_{(0, 2\pi) \times (0,\pi)} f(\widehat{H}(\theta',\psi'))\widehat{K}(\widehat{H}^{-1}(u,v,-1); d\theta' d\psi') \\
& = \int_{(0, 2\pi) \times (0,\pi)} f\left( \csc\gamma\cos\psi', \ell_{\csc\gamma\cos\psi'} (\cos\theta'), 1 \right) \\
& \hspace{2in} \cdot \widehat{K}(2\pi - \arccos(\ell_u^{-1}(v)), \arccos(u\sin\gamma); d\theta' d\psi'),
\end{split}
\]
where we use \eqref{eq:Hrep} and \eqref{eq:Hinvrep} in the last line. Applying \eqref{eq:Kpmrep}, this is equal to 
\[
\begin{split}
&\int_{(0, \pi)} f\left(u, \ell_{u}(\cos\theta'), 1 \right) P(\pi - \arccos(\ell_u^{-1}(v)),d\theta') \\
& = \int_{(-1,1)} f(u,\ell_u(x), 1)h_*P(-\ell_u^{-1}(v), dx') \\
& = \int_{(-1,1)} f(u, \ell_u(x), 1)\hat{q}(\ell_u^{-1}(v),dx') \\
& = \int_{(-\csc\gamma,\csc\gamma)} \int_{E_{u'}} f(u',v',1)(\ell_u)_*\hat{q}(v, dv') \delta_u(du') \\
& = \int_{E^*} f(u',v',s')Q(u,v,-1; du' dv' ds'),
\end{split}
\]
and this completes the proof.
\end{proof}

\subsection{Proofs of Theorem \ref{thm:refnub} and Corollary \ref{cor:refnub}} \label{ssec:refnub}

\subsubsection{Proof of Theorem \ref{thm:refnub}(i)}

Define a measure on $(0,\pi) \times (0,\pi)$ by $\mu^\delta(d\theta d\theta') = P^\delta(\theta, d\theta') \Lambda^1(d\theta)$. Let 
$$
M^\delta := \{(\theta, \theta') \in (0,\pi) \times (0,\pi): \mu^\delta \gg |\cdot|_U \text{ for some open } U \ni (\theta, \theta')\},
$$
where $|\cdot|_U$ denotes Lebesgue measure on $U$. For $i = 1,2$, we let $\pi_i : (0,\pi) \times (0,\pi) \to (0,\pi)$ denote projection onto the first and second coordinate, respectively. To show that $\Lambda^1$ is ergodic, the following fact is key.

\begin{lemma} \label{lem:fullprojection}
$\pi_1 M^\delta = (0,\pi)$.
\end{lemma}

\begin{proof}
Define $F^\delta : \mathbb{R} \times (0,\pi) \to (0,\pi) \times (0,\pi)$ by 
$$
F^\delta(x,\theta) = (\theta, p_2 \circ P^{\Sigma^\delta,1}(x,\theta)),
$$
where $p_2 : \mathbb{R} \times (0,\pi) \to (0,\pi)$ denotes projection onto the second (angular) coordinate. We note that $F^\delta(x,\theta)$ is $1$-periodic in $x$ since $P^{\Sigma^\delta,1}$ is. Moreover, if $X \sim \text{Unif}[0,1]$ and $\Theta \sim \Lambda^1$ are independent, then $F^\delta(X, \Theta) \sim \mu^\delta$. Consequently, to prove the claim, it suffices to show that for every $\theta \in (0,\pi)$, there exists $(u^*, \theta') \in \mathbb{R} \times (0,\pi)$ and open sets $U_0 \ni (u^*,\theta)$ and $U \ni (\theta,\theta')$ such that $F^\delta$ restricts to a diffeomorphism from $U_0$ onto $U$. (Note that by 1-periodicity, it does not matter if $u^*$ lies in $[0,1]$ or not.)

To this end, for $u \in (-\delta, \delta)$, let $\alpha(u) \in (0,\pi)$ denote the counterclockwise angle which the unit normal vector $n(u)$ makes with the positive $x_1$-axis. Since $h$ is smooth and strictly concave, $\alpha$ is smooth and $\alpha'(u) < 0$ for all $u \in (-\delta, \delta)$. Consider functions $\Psi_0 : (-\delta, \delta) \times (0,\pi) \to \mathbb{R} \times (0,\pi)$ and $\Psi_1 : (-\delta, \delta) \times (0,\pi) \to (0,\pi) \times \mathbb{R}$, defined as follows:
$$
\Psi_0(u,\theta) = (u - h(u)\cot\theta, \theta), \quad\quad \Psi_1(u,\theta) = (\theta, 2\alpha(u) - \theta).
$$
Suppose a billiard particle starts from the state $(x,\theta)$, reflects specularly from the nub at the point $(u,h(u))$, and returns to the horizontal axis without reflecting additional times from the boundary. If $(\theta,\theta') = F^\delta(x,\theta)$, it is easy to check that $\Psi_0(u,\theta) = (x,\theta)$ and $\Psi_1(u,\theta) = (\theta,\theta')$. See Figure \ref{fig:nubreflection}.

\begin{figure}
    \centering
    \includegraphics[width = 0.7\linewidth]{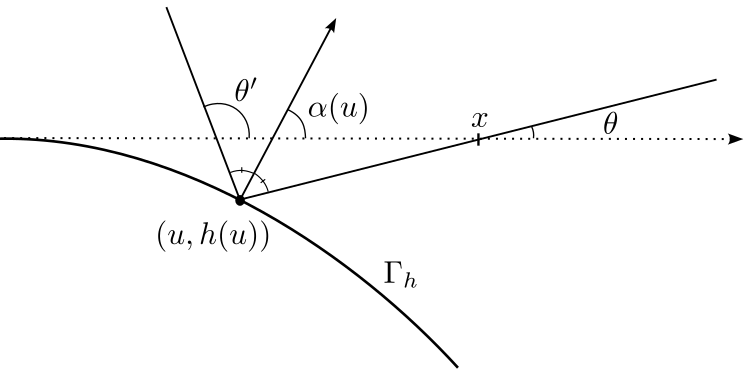}
    \caption{Reflection from a nub.}
    \label{fig:nubreflection}
\end{figure}

Let $u^* \in (-\delta,\delta)$ be the point at which $h$ attains its maximum, and note that $h(u^*) = h'(u^*) = 0$ and $\alpha(u^*) = \pi/2$. Also note if a billiard trajectory reflects from the nub at the point $(u^*,0)$, it cannot reflect additional times from the boundary. It is enough to show that for every $\theta \in (0,\pi)$, the total derivatives $D\Psi_0$ and $D\Psi_1$ are invertible at the point $(u^*,\theta)$. For then $\Psi_0$ and $\Psi_1$ are invertible in a neighborhood of the point $(u^*,\theta)$ by the inverse function theorem, and $F^\delta$ restricts to a diffeomorphism of form $\Psi_1 \circ \Psi_0^{-1} : U_0 \to U$, where $U_0$ is some neighborhood of $\Psi_0(u^*,\theta) = (u^*,\theta)$ and $U$ some neighborhood of $\Psi_1(u^*,\theta) = (\theta,\pi - \theta)$. For $(u,\theta) \in \mathbb{R} \times (0,\pi)$, we compute
$$
D\Psi_0(u,\theta) = \begin{bmatrix}
1 - h'(u)\cot\theta & h(u) \csc^2\theta \\
0 & 1
\end{bmatrix}, \quad D\Psi_1(u,\theta) = \begin{bmatrix}
0 & 1 \\
2\alpha'(u) & -1
\end{bmatrix}.
$$
We see that $D\Psi_0(u^*,\theta)$ and $D\Psi_1(u^*,\theta)$ are invertible, since $D\Psi_0(u^*,\theta)$ is just the identity, while $\det D\Psi_1(u,\theta) = -2\alpha'(u) > 0$ for any $u \in (-\delta,\delta)$. This proves the lemma.
\end{proof}

We can now prove the first part of the theorem.

\begin{proof}[Proof of Theorem \ref{thm:refnub}(i)] Let $A \subset (0,\pi)$ be an invariant set for $P^\delta$. We must show that $\Lambda^1(A) \in \{0,1\}$. Let $W$ be the set of all $\theta \in (0,\pi)$ such that for some open $I \subset (0,\pi)$, $\theta \in I$ and $\Lambda^1(I \smallsetminus A) = 0$. Let $W'$ be the set of all $\theta \in (0,\pi)$ such that for some open $I \subset (0,\pi)$, $\theta \in I$ and $\Lambda^1(I \cap A) = 0$. Clearly $W$ and $W'$ are disjoint open subsets of $(0,\pi)$. We will show that if the claim holds, then 
\begin{equation} \label{eq:unionisall}
W \cup W' = (0,\pi).
\end{equation}
Noting connectedness of $(0,\pi)$, it is immediate from \eqref{eq:unionisall} that either $W = (0,\pi)$ or $W' = (0,\pi)$. But this implies ergodicity, since if $W = (0,\pi)$, then choosing some countable cover for $(0,\pi)$ by open intervals $I$ such that $\Lambda^1(A \smallsetminus I) = 0$, we deduce that $\Lambda^1(A) = 1$; and similarly if $W' = (0,\pi)$, we deduce that $\Lambda^1(A) = 0$. 

To see why the claim implies \eqref{eq:unionisall}, first note by invariance of the measure $\Lambda^1$,
\begin{equation} \label{eq:fullmeas}
\mu^\delta((A \times A) \cup (A^c \times A^c)) = \Lambda^1(A) + \Lambda^1(A^c) = 1.
\end{equation}
In addition, by the claim, if $\theta \in (0,\pi)$, there exists an open set $U \subset (0,\pi) \times (0,\pi)$ such that $\theta \in \pi_1 U$ and $\mu^\delta \gg |\cdot|_U$. Without loss of generality, we may suppose that $U = I \times J$ for some open $I$ and $J$. In view of \eqref{eq:fullmeas}, $\mu^\delta(A \times A^c) = \mu^\delta(A^c \times A) = 0$. Consequently, 
$$
0 = |A \times A^c|_U = |A|_I |A^c|_J,
$$
and 
$$
0 = |A^c \times A|_U = |A^c|_I |A|_J.
$$
Since $|A|_J$ and $|A^c|_J$ cannot  both be zero, either $|A|_I = 0$ or $|A^c|_I = 0$. Equivalently, either $\Lambda^1(I \cap A) = 0$ or $\Lambda^1(I \smallsetminus A) = 0$. In the first case, $\theta \in W'$ and in the second case, $\theta \in W$, as desired. 
\end{proof}

\subsubsection{Proof of Theorem \ref{thm:refnub}(ii)}

\begin{proof}[Proof of Theorem \ref{thm:refnub}(ii)] Consider a billiard particle which enters the lower half-plane $x_2 \leq 0$ and reflects from the wall $W^\delta$ some number of times before leaving the lower half-plane and ceasing to reflect from the wall. We make some elementary observations:
\begin{enumerate}
\item All points of reflection must lie on the portion of the boundary of $W^\delta$ bounded between points of maximal height on adjacent nubs.

\item Since the normal vector on a nub has a positive $x_2$-component, the billiard particle can hit a nub at most twice, possibly once when the particle first hits the wall and possibly again when it last hits the wall.
\end{enumerate}

Let $V_0 \subset [0,1] \times (0,\pi)$ be the set of all initial configurations $(x,\theta)$ such that the billiard particle starting from $(x,\theta)$ does not reflect from the nub initially. Let $V_1 \subset [0,1] \times (0,\pi)$ be the set of all initial configurations $(x,\theta)$ such that the billiard particle does not reflect from a nub on its final reflection from the wall. Let $V = V_0 \cap V_1$.

Let $H : V_0 \to [0,1] \times (0,\pi)$ be defined by 
$$
H(x,\theta) = \left( \frac{x - |h(\delta)|\cot\theta}{1 - 2\delta}, \theta \right).
$$
Geometrically, $H$ has the following interpretation. Let $P_1 = (\delta,h(\delta))$ and $P_2 = (1 - \delta, h(\delta))$ be, respectively, the right and left endpoints of the nubs bounding the period of $\partial W^\delta$ below the interval $[0,1]$. The trajectory of a billiard particle starting from $(x,\theta) \in V_0$ will intersect the line segment $P_1 P_2$ at a point $R = u P_1 + (1 - u) P_2$, where $(u,\theta) = H(x,\theta)$. We define
$$
V_0' = H(V_0), V_1' = H(V_1), \text{ and } V' = H(V).
$$

Let $A := p_2V = p_2V'$, where $p_2 : [0,1] \times (0,\pi) \to (0,\pi)$ is projection onto the second (angular) coordinate, and let $X \sim \text{Unif}[0,1]$.

\textbf{Observation:} \textit{For any $\theta \in A$, $p_2 \circ P^{\Sigma^\delta,1}(X,\theta)$, conditional on $(X,\theta) \in V$, is equal in distribution to $p_2 \circ P^{\Sigma,1}(X,\theta)$, conditional on $(X,\theta) \in V'$.}

To see this, first note that the portion of the boundary of $W^\delta$ which runs between the points $P_1$ and $P_2$ is geometrically similar to the portion of the boundary of $W$ between the points $(0,0)$ and $(1,0)$. Moreover, conditional on $(X,\theta) \in V$, the trajectory starting from $(X,\theta)$ never reflects from a nub. Consequently, the angle of exit for the trajectory starting from $(X,\theta)$ and reflecting from $W^\delta$ is equal to the angle of exit for the trajectory starting from $H(X,\theta)$ and reflecting from $W$, i.e. 
\begin{equation} \label{eq:similarity}
p_2 \circ P^{\Sigma^\delta,1}(X,\theta) = p_2 \circ P^{\Sigma, 1}(H(X,\theta)).
\end{equation}
Since for each $\theta \in A$, $H(\cdot,\theta)$ is an affine isomorphism mapping $\{x \in [0,1] : (x,\theta) \in V\}$ onto $\{u \in [0,1] : (u,\theta) \in V'\}$, we see that $(X,\theta)$, conditional on $(X,\theta) \in V'$, is equal in distribution to $H(X,\theta)$, conditional on $(X,\theta) \in V$. The observation follows from this fact and \eqref{eq:similarity}.

Let $\widetilde{A} \subset A$. For any $\theta \in \widetilde{A}$, and $B \subset (0,\pi)$,
\begin{align*}
P(\theta,B) - P^\delta(\theta,B) & = \mathbb{P}(P^{\Sigma,1}(X,\theta) \in B) - \mathbb{P}(P^{\Sigma^\delta,1}(X,\theta) \in B) \\
& = \mathbb{P}(P^{\Sigma,1}(X,\theta) \in B \ | \ (X,\theta) \in V')\mathbb{P}((X,\theta) \in V') \\
& \quad + \mathbb{P}(P^{\Sigma,1}(X,\theta) \in B \ | \ (X,\theta) \notin V')\mathbb{P}((X,\theta) \notin V') \\
& \quad - \mathbb{P}(P^{\Sigma^\delta,1}(X,\theta) \in B \ | \ (X,\theta) \in V)\mathbb{P}((X,\theta) \in V) \\
& \quad - \mathbb{P}(P^{\Sigma^\delta,1}(X,\theta) \in B \ | \ (X,\theta) \notin V)\mathbb{P}((X,\theta) \notin V).
\end{align*}
By the observation, the conditional probabilities appearing in the second and fourth line above are equal. Using this and the fact that probabilities are bounded by 1, we obtain 
\begin{equation} \label{eq:probstobd}
\begin{split}
& |P(\theta,B) - P^\delta(\theta,B)| \\
& \leq c_1\max\left\{\left| \mathbb{P}((X,\theta) \in V') - \mathbb{P}((X,\theta) \in V) \right|, \mathbb{P}((X,\theta) \notin V'), \mathbb{P}((X,\theta) \notin V)\right\} \\
& \leq c_1\max\{\mathbb{P}((X,\theta) \notin V'), \mathbb{P}((X,\theta) \notin V)\}.
\end{split}
\end{equation}
We will be done if we can find a constant $c > 0$ and some $\widetilde{A} \subset A$ such that, for $\delta$ sufficiently small, $|(0,\pi) \smallsetminus \widetilde{A}| \leq c\delta^{2/9}$ and, for all $\theta \in \widetilde{A}$, the right-hand side of \eqref{eq:probstobd} is bounded above by $c\delta^{2/9}$.

\begin{figure}
    \centering
    \includegraphics[width = 0.8\linewidth]{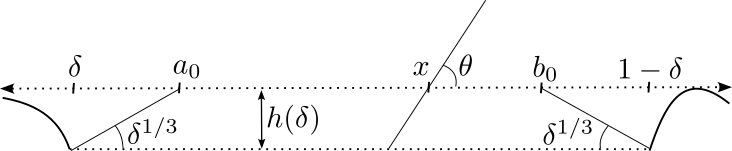}
    \caption{If $\widetilde{A}_0$ and $I_0$ are defined as in \eqref{eq:intervaldefs}, then $I_0 = (a_0,b_0)$ and the billiard trajectory starting from $(x,\theta) \in I_0 \times \widetilde{A}_0$ will not hit the nubs initially.}
    \label{fig:controlin}
\end{figure}

To begin the construction of $\widetilde{A}$, we define intervals
\begin{equation} \label{eq:intervaldefs}
\widetilde{A}_0 = (\delta^{1/3}, \pi - \delta^{1/3}) \quad \text{ and } \quad I_0 = (\delta + |h(\delta)|\cot\delta^{1/3}, 1 - \delta - |h(\delta)|\cot\delta^{1/3}).
\end{equation}
It is a matter of trigonometry to check that $I_0 \times \widetilde{A}_0 \subset V_0$ (see Figure \ref{fig:controlin}). Hence,
\begin{equation} \label{eq:R0inV}
R_0 := I_0 \times \widetilde{A}_0 \cap (P^{\Sigma^\delta,1})^{-1}(I_0 \times \widetilde{A}_0) \subset V.
\end{equation}
Since $P^{\Sigma^\delta,1}$ preserves $\widehat{\Lambda}^1$,  
\begin{align*}
& \widehat{\Lambda}^1((P^{\Sigma^\delta,1})^{-1}(I_0 \times \widetilde{A}_0)) = \widehat{\Lambda}^1(I_0 \times \widetilde{A}_0) \\
& \quad = (1 - 2\delta - 2|h(\delta)|\cot \delta^{1/3}) \cos \delta^{1/3} \\
& \quad \geq 1 - c_2 \delta^{2/3},
\end{align*}
where the last line uses the fact that $|h(\delta)| \leq c_3 \delta$, $\cos \delta^{1/3} \leq 1 - c_4 \delta^{2/3}$, and $\sin \delta^{1/3} \geq c_5 \delta^{1/3}$, for $\delta$ sufficiently small. Thus, 
\begin{equation} \label{eq:r0lb}
\widehat{\Lambda}^1(R_0) \geq 1 - c_6 \delta^{2/3}.
\end{equation}
Let 
$$
T_\delta = \{\theta \in (0,\pi) : \mathbb{P}((X,\theta) \notin R_0) > \delta^{2/9}\}.
$$
We have 
\begin{equation} \label{eq:r0clb}
\widehat{\Lambda}^1([0,1] \times (0,\pi) \smallsetminus R_0) = \int_{T_\delta} \mathbb{P}((X,\theta) \notin R_0) \Lambda^1(d\theta) \geq \delta^{2/9}\Lambda^1(T_\delta).
\end{equation}
Moreover, 
\begin{equation} \label{eq:tdeltalb}
\Lambda^1(T_\delta) \geq \int_{\delta^{2/9}}^{\pi - \delta^{2/9}} \mathbf{1}_{T_\delta}(\theta)\frac{1}{2}\sin\theta d\theta \geq \frac{1}{2}\delta^{2/9}|T_\delta \cap (\delta^{2/9}, \pi - \delta^{2/9})| \geq \frac{1}{2}\delta^{2/9}|T_\delta| - \delta^{4/9}. 
\end{equation}
Set 
$$
\widetilde{A}_1 = (0,\pi) \smallsetminus T_\delta = \{\theta \in (0,\pi) : \mathbb{P}((X,\theta) \notin R_0) \leq \delta^{2/9}\}.
$$
Combining \eqref{eq:r0lb}, \eqref{eq:r0clb}, and \eqref{eq:tdeltalb}, we obtain 
\begin{equation} \label{eq:A1compbd}
|(0,\pi) \smallsetminus \widetilde{A}_1| = |T_\delta| \leq c_7 \delta^{2/9},
\end{equation}
and by \eqref{eq:R0inV}, for any $\theta \in \widetilde{A}_1$, 
\begin{equation} \label{eq:notinv_bd}
\mathbb{P}((X,\theta) \notin V) \leq \mathbb{P}((X,\theta) \notin R_0) \leq \delta^{2/9}.
\end{equation}

\begin{figure}
    \centering
    \includegraphics[width = 0.8\linewidth]{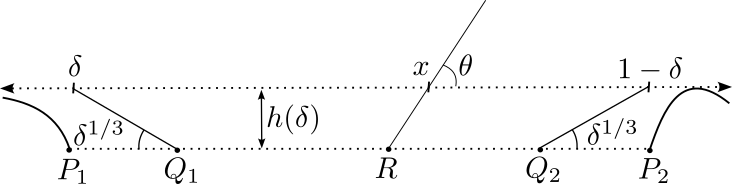}
    \caption{In the figure above, $P_1 = (\delta,h(\delta))$, $P_2 = (1 - \delta, h(\delta))$, $Q_1 = u_1 P_1 + (1 - u_1)P_2$ and $Q_2 = u_2 P_1 + (1 - u_2)P_2$, where $u_1$ and $u_2$ are the left and right endpoints, respectively, of the interval $I_1$, defined by \eqref{eq:I1def}. If $(u,\theta) \in I_1 \times \widetilde{A}_0$ and $R := u P_1 + (1 - u) P_1$, then $u = H(x,\theta)$, for some $\delta < x < 1 - \delta$, as shown. By noting that $(x,\theta) \in V_0$, this implies $(u,\theta) \in V_0'$, proving the containment \eqref{eq:productinV0}.}
    \label{fig:controlout}
\end{figure}

Next, consider the interval 
\begin{equation} \label{eq:I1def}
I_1 := \left( \frac{|h(\delta)|\cot\delta^{1/3}}{1 - 2\delta}, 1 - \frac{|h(\delta)|\cot\delta^{1/3}}{1 - 2\delta} \right).
\end{equation}
A simple geometric argument (see Figure \ref{fig:controlout}) shows that 
\begin{equation} \label{eq:productinV0}
I_1 \times \widetilde{A}_0 \subset V_0',
\end{equation}
and consequently, $(P^{\Sigma,1})^{-1}(I_1 \times \widetilde{A}_0) \subset V_1'$. Thus, 
\begin{equation} \label{eq:R1inVprime}
R_1 := (I_1 \times \widetilde{A}_0) \cap (P^{\Sigma,1})^{-1}(I_1 \times \widetilde{A}_0) \subset V'.
\end{equation}
Since $P^{\Sigma,1}$ preserves $\widehat{\Lambda}^1$, we have
\begin{align*}
&\widehat{\Lambda}^1((P^{\Sigma,1})^{-1}(I_1 \times \widetilde{A}_0)) = \widehat{\Lambda}^1(I_1 \times \widetilde{A}_0) \\
& \quad = \left(1 - \frac{2|h(\delta)|\cot\delta^{1/3}}{1 - 2\delta} \right) (1 - 2\delta^{1/3}) \\
& \quad \geq 1 - c_5 \delta^{2/3}.
\end{align*}
By essentially the same argument by which we obtained \eqref{eq:A1compbd}, if we define 
$$
\widetilde{A}_2 = \{\theta \in (0,\pi) : \mathbb{P}((X,\theta) \notin R_1) \leq \delta^{2/9}\},
$$
then
\begin{equation} \label{eq:A2compbd}
|(0,\pi) \smallsetminus \widetilde{A}_2| \leq c_8 \delta^{2/9}, 
\end{equation}
and by \eqref{eq:R1inVprime}, for any $\theta \in \widetilde{A}_2$,
\begin{equation} \label{eq:notinvprime_bd}
\mathbb{P}((X,\theta) \notin V') \leq \mathbb{P}((X,\theta) \notin R_1) \leq \delta^{2/9}.
\end{equation}
To conclude the argument, let $\widetilde{A} = \widetilde{A}_1 \cap \widetilde{A}_2$. Then by \eqref{eq:A1compbd} and \eqref{eq:A2compbd}, $|(0,\pi) \smallsetminus \widetilde{A}| \leq c_9 \delta^{2/9}$, and in view of \eqref{eq:notinv_bd} and \eqref{eq:notinvprime_bd}, the right-hand side of \eqref{eq:probstobd} is bounded above by $c_{10}\delta^{2/9}$, as desired.
\end{proof}

\begin{remark} \label{rem:Lambbd}
\normalfont In the argument above, by \eqref{eq:r0lb} and \eqref{eq:r0clb} we have $\Lambda^1((0,\pi) \smallsetminus \widetilde{A}_1) = \Lambda^1(T_\delta) \leq c_{11}\delta^{4/9}$. Similarly, $\Lambda^1((0,\pi) \smallsetminus \widetilde{A}_2) \leq c_{12}\delta^{4/9}$, and therefore $$
\Lambda^1((0,\pi) \smallsetminus \widetilde{A}) \leq c_{13}\delta^{4/9}.$$ 
We will need this in the proof of the corollary.
\end{remark}

\subsubsection{Proof of Corollary \ref{cor:refnub}}

\begin{proof}[Proof of Corollary \ref{cor:refnub}]
(i) By Theorem \ref{thm:refnub}(i), $\Lambda^1$ is an ergodic measure for $P$. Moreover, it follows from the proof of Lemma \ref{lem:fullprojection} that $P(\theta, d\theta')$ is nonsingular with respect to $\Lambda^1$ for all $\theta \in (0,\pi)$. Therefore, (i) follows by Theorem \ref{thm:ergcol} and Remark \ref{rem:nonsing}.

(ii) Denote by $\widetilde{P}^\delta$ the reflection factor in the decomposition \eqref{eq:roughcolform} for $K_+^\delta$. By Theorem \ref{thm:colchar} and Remark \ref{rem:shortnub}, $\widetilde{P}^\delta = \widetilde{P}^{\widetilde{\Sigma}^{\delta},1}$, where $\widetilde{\Sigma} = \mathcal{F}(\Sigma)$. By Theorem \ref{thm:refnub} and Remark \ref{rem:Lambbd}, there exist constants $c, \delta_0 > 0$ and $\widetilde{A} \subset (0,\pi)$ such that for all $\delta \in (0,\delta_0)$, 
$$
\Lambda^1((0,\pi) \smallsetminus \widetilde{A}) \leq c\delta^{4/9},
$$ 
and for all $\theta \in \widetilde{A}$ and $B \subset (0,\pi)$,
\begin{equation} \label{eq:refkernelbd}
|\widetilde{P}(\theta,B) - \widetilde{P}^\delta(\theta,B)| \leq c \delta^{2/9}.
\end{equation}
Let $\widetilde{R}$ be the subset of $\mathbb{S}^2_+$ which in the spherical coordinates $(\theta,\psi)$, defined by \eqref{eq:spherecoord}, is given by 
$$
\widetilde{R} = \{(\theta,\psi) \in (0,\pi) \times (0,\pi): \theta \in \widetilde{A}\}.
$$
By the spherical coordinate representation \eqref{eq:lambda2coordrep} for $\Lambda^2$,
\begin{equation} \label{eq:LambRc}
\Lambda^2(\mathbb{S}^2_+ \smallsetminus \widetilde{R}) = \int_{(0,\pi) \smallsetminus \widetilde{A}} \int_{(0,\pi)} \frac{1}{\pi} \sin^2 \psi d\psi \Lambda^1(d\theta) = \frac{1}{2}\Lambda^1((0,\pi) \smallsetminus \widetilde{A}) \leq c_1 \delta^{4/9}.
\end{equation}
On the other hand, by the usual coordinate representation of spherical surface measure $\sigma^2$, 
\begin{equation}
\begin{split}
\Lambda^2(\mathbb{S}^2_+ \smallsetminus \widetilde{R}) & = \int_{(0,\pi) \smallsetminus \widetilde{A}} \int_{(0,\pi)} \frac{1}{\pi} \sin\theta \sin\psi \sigma^2(d\theta d\psi) \\
& \geq \int_{\delta^{1/9}}^{\pi-\delta^{1/9}} \int_{\delta^{1/9}}^{\pi-\delta^{1/9}} \mathbf{1}_{\widetilde{A}^c}(\theta) \frac{1}{\pi} \sin\theta \sin\psi \sigma^2(d\theta d\psi) \\
& \geq c_3 \delta^{2/9}\left( \sigma^2(\mathbb{S}^2_+ \smallsetminus \widetilde{R}) - \delta^{1/9} \right) \\
& \geq c_4 \delta^{2/9}\sigma^2(\mathbb{S}^2_+ \smallsetminus \widetilde{R}), \label{eq:Rcompbd}
\end{split}
\end{equation}
where the third line uses the fact that if $\delta^{1/9} \leq \theta, \psi \leq \pi - \delta^{1/9}$, then $\sin\theta\sin\psi \geq \delta^{2/9}$. Combining \eqref{eq:LambRc} and \eqref{eq:Rcompbd}, we obtain
$$
\sigma^2(\mathbb{S}^2_+ \smallsetminus \widetilde{R}) \leq c_5 \delta^{2/9},
$$
Moreover, for $B \subset \mathbb{S}^2_+$ and $(\theta,\psi) \in \mathbb{S}^2_+$, if we let
$$
B_\psi = \{\theta' \in (0,\pi) : (\theta',\psi) \in B\},
$$
then by \eqref{eq:roughcolform} and \eqref{eq:refkernelbd}, 
$$
|K_+(\theta,\psi; B) - K_+^\delta(\theta,\psi; B)| = |\widetilde{P}(\theta,B_{-\psi}) - \widetilde{P}^\delta(\theta,B_{-\psi})| \leq c_1 \delta^{2/9},
$$
as desired.
\end{proof}

\section{Proofs of main results} \label{sec:ergproofs}

\subsection{Proofs of Propositions} \label{ssec:propsproofs}

\begin{lemma} \label{lem:m1m2qQ}
If $f$ and $g$ are in $L^1(E^*, m^2)$, then the following identities hold:
\begin{enumerate}
\item[(i)]
\begin{equation} \label{eq:m1m2qQ1}
\begin{split}
&\int_{E^*} f(u,v,s) m^2(du dv ds) \\
& \quad = \frac{1}{2|E|}\int_{-\csc\gamma}^{\csc\gamma} \int_{(-1,1)} [f(\ell_v(x),v,1) + f(\ell_v(x),v,-1)] m^1(dx) |E_v| dv \\
& \quad = \frac{1}{2|E|}\int_{-\csc\gamma}^{\csc\gamma} \int_{(-1,1)} [f(u,\ell_u(x),1) + f(u, \ell_u(x),-1)] m^1(dx) |E_u| du.
\end{split}
\end{equation}
\item[(ii)]
\begin{equation} \label{eq:m1m2qQ2}
\begin{split}
&\int_{E^* \times E^*} f(u',v',s')g(u,v,s)Q(u,v,s; du' dv' ds')m^2(du dv ds) \\
& = \frac{1}{|E|}\int_{-\csc\gamma}^{\csc\gamma} \int_{(-1,1)} \left( \int_{(-1,1)} \left[ f(u,\ell_u(x'),-1) + f(\ell_u(x'),u,1)\right]q(-x,dx') \right) \\
& \hspace{1.2in} \cdot \left[g(u,\ell_u(x), 1) + g(\ell_u(x), u, -1)\right]m^1(dx) |E_u| du.
\end{split}
\end{equation}
\end{enumerate}
\end{lemma}

\noindent Note that $|E_u| = 2\sqrt{1 - u^2 \sin^2\gamma}$; thus the formulas above are completely explicit.

\begin{proof}[Proof of Lemma \ref{lem:m1m2qQ}]
(i) To obtain the first equality, we compute
\[
\begin{split}
& \int_{E^*} f(u,v,s)m^2(du dv ds) \\
& = \frac{1}{2|E|} \int_E [f(u,v,1) + f(u,v,-1)] du dv \\
& = \frac{1}{2|E|} \int_{-\csc\gamma}^{\csc\gamma} \int_{E_v} [f(u,v,1) + f(u,v,-1)] du dv \\
& = \frac{1}{2|E|} \int_{-\csc\gamma}^{\csc\gamma} \int_{(-1,1)} [f(\ell_v(x),v,1) + f(\ell_v(x),v,-1)] |E_v| du dv, 
\end{split}
\]
noting that $\ell_v'(x) = |E_v| = 2\sqrt{1 - v^2 \sin\gamma}$. The second equality in (i) follows by a similar calculation, except that in the third line we do the integration in reverse order.

(ii) By part (i) and \eqref{eq:Qdef},
\[
\begin{split}
&\int_{E^* \times E^*} g(u,v,s)f(u',v',s')Q(u,v,s; du' dv' ds')m^2(du dv ds) \\
& = \frac{1}{2|E|}\int_E \left( \int_{E_u} f(u,v',-1) (\ell_u)_*\hat{q}(v, dv')\right) g(u,v,1) dv du \\
& \quad + \frac{1}{2|E|}\int_E \left( \int_{E_v} f(u',v,1) (\ell_v)_*\hat{q}(u, du')\right) g(u,v,-1) du dv \\
& = \frac{1}{2|E|}\int_{-\csc\gamma}^{\csc\gamma} \int_{E_u} \left( \int_{(-1,1)} f(u, \ell_u(x'), -1)q(-\ell_u^{-1}(v), dx') \right) g(u,v,1) dv du \\
& \quad + \frac{1}{2|E|}\int_{-\csc\gamma}^{\csc\gamma} \int_{E_v} \left( \int_{(-1,1)} f(\ell_v(x'), v, 1)q(-\ell_v^{-1}(u), dx') \right) g(u,v,-1) du dv \\
& = \frac{1}{|E|}\int_{-\csc\gamma}^{\csc\gamma} \int_{(-1,1)} \left( \int_{(-1,1)} f(u, \ell_u(x'), -1)q(-x, dx') \right) \\ & \hspace{2.5in} \cdot g(u,\ell_u(x),1) |E_u|m^1(dx) du \\
& \quad + \frac{1}{|E|}\int_{-\csc\gamma}^{\csc\gamma} \int_{(-1,1)} \left( \int_{(-1,1)} f(\ell_v(x'), v, 1)q(-x, dx') \right) \\ & \hspace{2.5in} \cdot g(\ell_v(x),v,-1) |E_v|m^1(dx) dv,
\end{split}
\]
and the result follows.
\end{proof}

\begin{proof}[Proof of Proposition \ref{prop:ir}]
($i \Leftrightarrow ii$) Let $j, k \in \text{BC}(-1,1)$, and assume that $m^1$ is a reversible measure for $q$. Then 
\[
\begin{split}
& \int_{(-1,1) \times (-1,1)} j(x)k(x')\hat{q}(x,dx') m^1(dx) = \int_{(-1,1) \times (-1,1)} j(-\hat{x})k(x')q(\hat{x}, dx') m^1(d\hat{x}) \\
& = \int_{(-1,1) \times (-1,1)} k(\hat{x})j(-x')q(\hat{x},dx')m^1(d\hat{x}) = \int_{(-1,1) \times (-1,1)} k(-x)j(-x')\hat{q}(x,dx')m^1(dx),
\end{split}
\]
where $\hat{x} = -x$, which shows that $m^1$ is a $\dag$-reversible measure for $\hat{q}$. 

If we assume merely that $m^1$ is an invariant measure for $q$, then we get the corresponding implication by taking $j = 1$, and noting that $m^1$ is invariant under negation.

The reverse implication is obtained by an almost identical argument.

($ii \Rightarrow iii$) Let $f,g \in \text{BC}(E^*)$, and assume that $m^1$ is $\dag$-reversible for $\hat{q}$. By Lemma \ref{lem:m1m2qQ}, we have 
\begin{equation}
\begin{split}
&\int_{E^* \times E^*} g(u,v,s)f(u',v',s')Q(u,v,s; du' dv' ds')m^2(du dv ds) \\
& = \frac{1}{|E|}\int_{-\csc\gamma}^{\csc\gamma} \int_{(-1,1)} \left( \int_{(-1,1)} \left[ f(u,\ell_u(x'),-1) + f(\ell_u(x'),u,1) \right]\hat{q}(x,dx') \right) \\
& \hspace{1.5in} \cdot \left[ g(u,\ell_u(x),1) + g(\ell_u(x), u, -1) \right] m^1(dx) |E_u| du \\
& = \frac{1}{|E|}\int_{-\csc\gamma}^{\csc\gamma} \int_{(-1,1)} \left( \int_{(-1,1)} \left[ g(u,\ell_u(x'),1) + g(\ell_u(x'), u, -1) \right]\hat{q}^\dag(x,dx') \right) \\
& \hspace{1.5in} \cdot \left[ f(u,\ell_u(x),-1) + f(\ell_u(x),u,1) \right] m^1(dx) |E_u| du \\
& = \int_{E^* \times E^*} f(u,v,s)g(u',v',s')Q^\dag(u,v,s; du' dv' ds')m^2(du dv ds), \label{eq:m2istrev}
\end{split}
\end{equation}
thus proving that $m^2$ is a $\dag$-reversible measure for $Q$.

If we assume instead that $m^1$ is invariant with respect to $q$, then we get the corresponding implication by taking $g = 1$ in the calculation above.

($iii \Rightarrow ii$) Assume that $m^1$ is not a $\dag$-reversible measure for $\hat{q}$. We will show that $m^2$ is not a $\dag$-reversible measure for $Q$. Indeed, by assumption, there exist $\varepsilon > 0$ and $j_0, k_0 \in \text{BC}(-1,1)$ such that 
\begin{equation}
\begin{split}
& \int_{(-1,1) \times (-1,1)} k_0(x)j_0(x')p(x,dx')m^1(dx) \\
& \geq \varepsilon + \int_{(-1,1) \times (-1,1)} j_0(x)k_0(x')\hat{q}^\dag(x,dx')m^1(dx). \label{eq:nottrev}
\end{split}
\end{equation}
We define $f_0, g_0 \in \text{BC}(E^*)$ as follows:
\[
f_0(u,v,s) = \begin{cases}
\frac{1}{2}j_0(\ell_v^{-1}(u)) & \text{ if } s = 1, \\
\frac{1}{2}j_0(\ell_u^{-1}(v)) & \text{ if } s = -1.
\end{cases}
\]
and 
\[
g_0(u,v,s) = \begin{cases}
\frac{1}{2}k_0(\ell_u^{-1}(v)) & \text{ if } s = 1, \\
\frac{1}{2}k_0(\ell_v^{-1}(u)) & \text{ if } s = -1.
\end{cases}
\]
Observe that 
\[
f_0(u, \ell_u(x'),-1) + f_0(\ell_u(x'),u,1) = j_0(x'),
\]
and 
\[
g_0(u,\ell_u(x),1) + g_0(\ell_u(x),u,-1) = k_0(x).
\]
Thus, by a similar calculation to \eqref{eq:m2istrev},
\begin{equation}
\begin{split}
& \int_{E^* \times E^*} g(u,v,s)f(u',v',s')Q(u,v,s; du' dv' ds')m^2(du dv ds) \\
& = \frac{1}{2|E|} \int_{-\csc\gamma}^{\csc\gamma} \int_{(-1,1) \times (-1,1)} k_0(x)j_0(x')\hat{q}(x,dx') m^1(dx) |E_u| du \\
& \geq c_1 \varepsilon + \frac{1}{2|E|} \int_{-\csc\gamma}^{\csc\gamma} \int_{(-1,1) \times (-1,1)} j_0(x)k_0(x')\hat{q}^\dag(x,dx') m^1(dx) |E_u| du \\
& = c_1 \varepsilon + \int_{E^* \times E^*} f(u,v,s)g(u',v',s')Q^\dag(u,v,s; du' dv' ds')m^2(du dv ds),
\end{split}
\end{equation}
where the inequality follows by \eqref{eq:nottrev}, and $c_1$ is some positive constant. This shows $m^2$ is not $\dag$-reversible.

If we assume instead that $m^1$ is not invariant with respect to $\hat{q}$, then one may show that $m^2$ is not invariant with respect to $Q$ by a very similar argument. (One replaces $k_0$ with $1$ and $g_0$ with $1$, and performs essentially the same calculations as above.)
\end{proof}

\subsection{Proof of Theorem \ref{thm:mainresult}} \label{ssec:main1proof}

\subsubsection{The Abel transform} \label{ssec:abel}

A key ingredient in the proof of Theorem \ref{thm:mainresult} is the Abel transform, a classic integral transform first studied by Niels Henrik Abel in the context of geometric optics. This transform is defined as follows:
\begin{equation}\label{eq:abeldef}
\mathcal{A}[\varphi](x) = \int_0^x \frac{\varphi(t)dt}{\sqrt{x - t}}, \quad\quad 0 < x < 1.
\end{equation}
This transform is a bounded operator on $L^1(0,1)$. This can be seen by integrating the right-hand side of \eqref{eq:abeldef} and using Fubini's theorem to interchange the order of integration. More generally, one can apply Minkowski's and H\"{o}lder's inequalities to \eqref{eq:abeldef} show that $\mathcal{A}[\varphi]$ is a bounded operator on $L^p(0,1)$ for all $p \in [1,2)$.

While there exist explicit inversion formulas for this transform, describing a sufficiently broad class of functions $\varphi$ for which these inversion formulas are valid is a nontrivial task, complicated by the singularity at $t = x$. The main fact we will need is that Abel's integral equation
\begin{equation} \label{eq:Abel-equation}
\mathcal{A}[\varphi](x) = f(x), \quad\quad 0 < x < 1,
\end{equation}
has at most one solution in $L^1(0,1)$. That is, 

\begin{theorem} \label{thm_Abel}
For every $f \in L^1(0,1)$, there is at most one $\varphi \in L^1(0,1)$ (up to a Lebesgue-null set) such that equation \eqref{eq:Abel-equation} is satisfied for all $x \in (0,1)$.
\end{theorem}

\noindent We refer the reader to \cite{GVAbelIntEq} for more information on the Abel transform. Theorem \ref{thm_Abel} was originally proved by Tonelli in \cite{Tonelli}.

\subsubsection{Notation and auxiliary lemmas}

In this section, we will frequently need to express the fact that two sets are equal, up to a null set difference. If $A, B \subset (-1,1)$, we will write $A \triangleq B$ if $A$ and $B$ are equal, up to an $m^1$-null set, i.e. $m^1((A \smallsetminus B) \cup (B \smallsetminus A)) = 0$. Similarly, if $C, D \subset E^*$, we will write $C \triangleq D$ if $m^2((C \smallsetminus D) \cup (D \smallsetminus C)) = 0$. It will usually be clear from context which measure this notation refers to.

We let $E^+ = E \times \{1\}$ and $E^- = E \times \{-1\}$. (Hence $E^* = E^+ \cup E^-$.) If $F \subset E^*$, we let $F^\pm = F \cap E^\pm$. For any $G \subset E$ and $(u,v) \in E$, we define the following subsets of $(-1,1)$:
\[
[G]^u = \ell_u^{-1}(\{v' \in E_u : (u,v') \in G\}), \quad \text{ and } \quad [G]_v = \ell_v^{-1}(\{u' \in E_v : (u',v) \in G\}),
\]
where we recall that $\ell_u$ is defined by \eqref{eq:Sform}.
The following identities are trivial to verify but will be needed a number of times. For any $G_1, G_2 \subset E$, and $(u,v) \in E$,
\begin{equation}
[G_1 \cup G_2]^u = [G_1]^u \cup [G_2]^u, \quad\quad\quad [G_1 \cup G_2]_v = [G_1]_v \cup [G_2]_v,
\end{equation}
\begin{equation}
[G_1 \cap G_2]^u = [G_1]^u \cap [G_2]^u, \quad\quad\quad [G_1 \cap G_2]_v = [G_1]_v \cap [G_2]_v.
\end{equation}

Define maps $T_0, T_1 : E^* \to E^*$ by
\begin{equation} \label{eq:T0def}
\begin{split}
T_0(u,v,s) & := (\ell_v(-\ell_v^{-1}(u)), v, -s) = (-u - 2v\cos\gamma,v,-s)
\end{split}
\end{equation}
and
\begin{equation} \label{eq:T1def}
\begin{split}
T_1(u,v,s) & := (u, \ell_u(-\ell_u^{-1}(v)), -s) = (u, -v - 2u\cos\gamma, -s).
\end{split}
\end{equation}

\begin{lemma} 
For any $(u,v) \in E$ and any $F \subset E^*$, 
\begin{equation} \label{eq:bracs}
[T_0(F)^\pm]_v = -[F^\mp]_v, \quad \text{ and } \quad [T_1(F)^\pm]^u = -[F^\mp]^u. 
\end{equation}
\end{lemma}

\begin{proof}
The first equality in \eqref{eq:bracs} may be obtained through the following chain of equivalences. For any $x \in (-1,1)$, 
\[
\begin{split}
x \in [T_0(F)^\pm]_v & \quad \Leftrightarrow \quad (\ell_v(x),v,\pm 1) \in T_0(F) \\
& \quad \Leftrightarrow \quad \exists w \in (-\csc\gamma, \csc\gamma) \text{ such that } \ell_v(x) = \ell_v(-\ell_v^{-1}(w)) \\
& \hspace{2.7in} \text{ and } (w,v,\mp 1) \in F \\
& \quad \Leftrightarrow \quad (\ell_v(-x), v, \mp 1) \in F \\
& \quad \Leftrightarrow \quad x \in -[F^\mp]_v.
\end{split}
\]
The proof of the second equality in \eqref{eq:bracs} is similar.
\end{proof}

\begin{lemma} \label{lem:slices}
Suppose $F \subset E^*$ is $(m^2,Q)$-invariant. For almost every $(u,v) \in E$,
\begin{enumerate}[label = (\roman*)]
\item $q(x,[F^+]_v) = 1$ for almost every $x \in -[F^-]_v$;
\item $q(x,-[F^-]_v) = 1$ for almost every $x \in [F^+]_v$;
\item $q(x,[F^-]^u) = 1$ for almost every $x \in -[F^+]^u$; and
\item $q(x,-[F^+]^u) = 1$ for almost every $x \in [F^-]^u$.
\end{enumerate}
\end{lemma}

\begin{proof}
We will only prove (i) and (ii) since the proofs of (iii) and (iv) are similar. Consider the following calculation.
\begin{equation} \label{eq:slices1}
\begin{split}
& \int_{-\csc\gamma}^{\csc\gamma} m^1(-[F^-]_v)|E_v| dv \\
& = \int_{-\csc\gamma}^{\csc\gamma} m^1([T_0(F)^+]_v)|E_v| dv \\
& = 2|E|m^2(T_0(F)^+) \\
& = 2|E|m^2(F^-) \\
& = 2|E| \int_{F^-} Q(u,v,s; F^+) m^2(du dv ds) \\
& = \int_{-\csc\gamma}^{\csc\gamma} \int_{(-1,1)} \left( \int_{(-1,1)} \mathbf{1}_{F^+}(\ell_v(x'),v,1)q(-x, dx') \right)  \\
& \hspace{2in} \cdot \mathbf{1}_{F^-}(\ell_v(x),v,-1)m^1(dx)|E_v| dv \\
& = \int_{-\csc\gamma}^{\csc\gamma} \int_{(-1,1)} \left( \int_{(-1,1)} \mathbf{1}_{F^+}(\ell_v(x'),v,1)q(x, dx') \right)  \\
& \hspace{2in} \cdot \mathbf{1}_{F^-}(\ell_v(-x),v,-1)m^1(dx)|E_v| dv \\ 
& =\int_{-\csc\gamma}^{\csc\gamma} \int_{-[F^-]_v} q(x, [F^+]_v)m^1(dx) |E_v| dv.
\end{split}
\end{equation}
Here the first equality uses equation \eqref{eq:bracs}; the second equality follows from Lemma \ref{lem:m1m2qQ}(i), the third equality uses the fact that $T_0$ preserves the measure $m^2$, as well as the easily checked identity $T_0(F^-) = T_0(F)^+$; the fourth equality follows from $(m^2,Q)$-invariance of $F$, as well as the fact that $Q(u,v,s; \cdot)$ is supported in $E^+$ for every $(u,v,s) \in E^-$; the fifth equality follows from Lemma \ref{lem:m1m2qQ}(ii); in the sixth equality we make the change of variables $x \mapsto -x$; and the last equality follows from definitions of $[F^+]_v$ and $[F^-]_v$.

Since $0 \leq q \leq 1$, \eqref{eq:slices1} implies that $q(x,[F^+]_v) = 1$ for almost every $(x,v) \in -[F^-]_v \times (-\csc\gamma, \csc\gamma)$, and this proves (i). By this result, and reversibility of $q$ with respect to $m^1$, we also have for a.e. $v \in (-\csc\gamma, \csc\gamma)$.
\begin{equation} \label{eq:slices2}
   m^1(-[F^-]_v) = \int_{-[F^-]_v} q(x,[F^+]_v)m^1(dx) = \int_{[F^+]_v} q(x, -[F^-]_v)m^1(dx).
\end{equation}
Again since $0 \leq q \leq 1$, \eqref{eq:slices2} implies that 
\begin{equation} \label{eq:slices3}
m^1([F^+]_v) \geq m^1(-[F^-]_v).
\end{equation}
By Lemma \ref{lem:m1m2qQ} and \eqref{eq:slices3},
\begin{equation} \label{eq:slices4}
\begin{split}
m^2(F^+) & = \frac{1}{2|E|} \int_{-\csc\gamma}^{\csc\gamma} m^1([F^+]_v)dv \\
& \geq \frac{1}{2|E|} \int_{-\csc\gamma}^{\csc\gamma} m^1(-[F^-]_v)|E_v| dv \\
& = \frac{1}{2|E|}\int_{-\csc\gamma}^{\csc\gamma} m^1([F^-]_v)|E_v| dv \\
& = m^2(F^-),
\end{split}
\end{equation}
with equality if and only if \eqref{eq:slices3} is an equality for a.e. $v \in (-\csc\gamma, \csc\gamma)$. But since $F$ is  $(m^2,Q)$-invariant, $m^2$ is a $\dag$-reversible measure for $Q$, and probabilities are bounded by $1$, 
\[
\begin{split}
m^2(F^+) & = \int_{F^+} Q(u,v,s; F^-) m^2(du dv ds) \\
& = \int_{F^-} Q^\dag(u,v,s; F^+) m^2(du dv ds) \leq m^2(F^-).
\end{split}
\]
We conclude that \eqref{eq:slices4} is an equality, and \eqref{eq:slices3} is an equality for almost every $v$. Hence, in view of \eqref{eq:slices2}, (ii) must be satisfied.
\end{proof}

Define projection maps $\pi_0, \pi_1 : E^* \to \mathbb{R}$ by $\pi_0(u,v,s) = v$, and $\pi_1(u,v,s) = u$.

\begin{lemma} \label{lem:invarstruc}
Assume that $m^1$ is reversible with respect to $q$. If $F \subset E^*$ is $(m^2,Q)$-invariant, then $T_0(F)$ and $T_1(F)$ are also $(m^2,Q)$-invariant. If $m^1$ is, moreover, ergodic with respect to $q$, then there exist subsets $P_0, P_1 \subset (-\csc\gamma,\csc\gamma)$ such that 
\begin{equation} \label{eq:fibers}
F \cup T_0(F) \triangleq \pi_0^{-1}(P_0) \quad \text{ and } \quad F \cup T_1(F) \triangleq \pi_1^{-1}(P_1).
\end{equation}

\end{lemma}

\begin{proof}
The result is trivial if $m^2(F) = 0$, so assume that this is not the case. To show that $T_0(F)$ is $(m^2,Q)$-invariant, we compute
\begin{equation} \label{eq:iv1}
\begin{split}
& \int_{T_0(F)^-} Q(u,v,s; T_0(F)^+) m^2(du dv ds) \\
& = \int_{F^+} Q(T_0(u,v,s); T_0(F)^+) m^2(du dv ds) \\
& = \frac{1}{2|E|} \int_{-\csc\gamma}^{\csc\gamma} \int_{[F^+]_v} q(x, [T_0(F)^+]_v)m^1(dx) |E_v| dv \\
& = \frac{1}{2|E|} \int_{-\csc\gamma}^{\csc\gamma} \int_{[F^+]_v} q(x, -[F^-]_v)m^1(dx) |E_v| dv \\
& = \frac{1}{2|E|} \int_{-\csc\gamma}^{\csc\gamma} \int_{-[F^-]_v} q(x, [F^+]_v)m^1(dx)|E_v| dv \\
& = \frac{1}{2|E|} \int_{-\csc\gamma}^{\csc\gamma} m^1(-[F^-]_v) |E_v| dv \\
& = m^2(T_0(F)^-).
\end{split}
\end{equation}
Here the first equality follows by invariance of $m^2$ under the mapping $T_0$; the second equality follows by Lemma \ref{lem:m1m2qQ}(i); the third equality uses \eqref{eq:bracs}; the fourth equality follows by reversibility of $q$ with respect to $m^1$; the fifth equality follows by Lemma \ref{lem:slices}(i); and the last equality uses \eqref{eq:bracs} and Lemma \ref{lem:m1m2qQ}(i) again. A symmetric calculation shows that 
\begin{equation} \label{eq:iv2}
\int_{T_0(F)^+} Q(u,v,s; T_0(F)^-)m^2(du dv ds) = m^2(T_0(F)^+).
\end{equation}
Equations \eqref{eq:iv1} and \eqref{eq:iv2} imply that $Q(u,v,s; T_0(F)^+) = 1$ for a.e. $(u,v,s) \in T_0(F)^-$, and $Q(u,v,s; T_0(F)^-) = 1$ for a.e. $(u,v,s) \in T_0(F)^+$. Hence, $Q(u,v,s; T_0(F)) = 1$ for a.e. $(u,v,s) \in T_0(F)$, which shows $T_0(F)$ is $(Q,m^2)$-invariant.

The proof that $T_1(F)$ is $(m^2,Q)$-invariant follows by a symmetric argument.

To prove that \eqref{eq:fibers}, we proceed as follows. Define
\[
P_0 = \{v \in (-\csc\gamma, \csc\gamma) : m^1([(F \cup T_0(F))^+]_v) > 0\}.
\]
Notice that $[(F \cup T_0(F))^+]_v = -[(T_0(F) \cup F)^-]_v$ by \eqref{eq:bracs}. Hence, by Lemma \ref{lem:slices}(i), for almost every $v$,
\[
q(u, [(F \cup T_0(F))^+]_v) = 1 \text{ for a.e. } u \in [(F \cup T_0(F))^+]_v.
\]
This shows $[(F \cup T_0(F))^+]_v$ is $(m^1,q)$-invariant. By ergodicity of $m^1$ with respect to $q$, $m^1([(F \cup T_0(F))^+]_v) = 1$ for a.e. $v \in P_0$. Similarly, $m^1([(F \cup T_0(F))^-]_v) = 1$ for a.e. $v \in P_0$. By Lemma \ref{lem:m1m2qQ}(i), 
\[
\begin{split}
& m^2(F \cup T_0(F)) \\
& = \frac{1}{2|E|}\int_{P_0} [m^1([(F \cup T_0(F))^+]_v) + m^1([(F \cup T_0(F))^-]_v)]|E_v| dv \\
& = \frac{1}{|E|} \int_{P_0} |E_v| dv = m^2(\pi_0^{-1}(P_0)).
\end{split}
\]
Since it is also evident from the definition of $\pi_0$ that, up to null sets, $F \cup T_0(F) \subset \pi_0^{-1}(P_0)$, the equality above proves that $F \cup T_0(F) \triangleq \pi_0^{-1}(P_0)$. The second equality in \eqref{eq:fibers} follows by a symmetric argument.
\end{proof}

\begin{corollary} \label{cor:erg}
Assume $m^1$ is ergodic with respect to $q$. If $F$ is $(m^2,Q)$-ergodic, then $T_0(F)$ and $T_1(F)$ are also $(m^2,Q)$-ergodic.
\end{corollary}

\begin{proof}
We will show just that $T_0(F)$ is $(m^2,Q)$-ergodic, as the argument for $T_1(F)$ is similar. Suppose $G \subset T_0(F)$ is $(m^2, Q)$-invariant. We must show that $m^2(G) \in \{0,m^2(T_0(F))\}$. By Lemma \ref{lem:invarstruc}, $T_0(G) \subset F$ is also $(m^2,Q)$-invariant. By ergodicity of $F$, it follows $m^2(T_0(G)) \in \{0, m^2(F)\}$, and the result follows since $T_0$ preserves the measure $m^2$. 
\end{proof}

Let $U_\gamma : \mathbb{R}^2 \to \mathbb{R}^2$ be the linear transformation given by 
\begin{equation} \label{eq:Udef}
U_\gamma(u,v) = (u + v\cos\gamma, v\sin\gamma).
\end{equation}
This transformation has the property that $U_\gamma(E) = D^2$, the unit disk in $\mathbb{R}^2$. For each $\alpha \in \mathbb{R}$, let $R_\alpha : \mathbb{R}^2 \to \mathbb{R}^2$ denote the counterclockwise rotation about the origin by the angle $\alpha$, i.e.
\[
R_\alpha(x,y) = (x\cos\alpha + y\sin\alpha, -x\sin\alpha + y\cos\alpha).
\]
Define $\widehat{R}_\alpha : E^* \to E^*$ by 
\[
\widehat{R}_\alpha(u,v,s) = (U_\gamma^{-1} \circ R_\alpha \circ U_\gamma(u,v),s).
\]
Also, let $V : E^* \to [-1,1]$ be defined by 
\[
V(u,v,s) = s|U_\gamma(u,v)|^2,
\]
where the right-hand side is $s$ times the square of the Euclidean norm of $U_\gamma(u,v)$.

\begin{lemma} \label{lem:tri}
Assume that $\gamma/\pi$ is irrational, and assume that $m^1$ is ergodic with respect to $q$. Suppose $F \subset E^*$ is $(m^2,Q)$-ergodic and $0 < m^2(F) < 1$. Then $F$ has the following properties:
\begin{enumerate}[label = (\roman*)]
\item There exists a measurable set $K \subset [-1,1]$ such that $F \triangleq V^{-1}(K)$.
\item $F \cup T_0(F) \triangleq E^*$, and $F \cap T_0(F) \triangleq \emptyset$.
\item $m^2(F) = m^2(T_0(F)) = \frac{1}{2}$.
\end{enumerate}
\end{lemma}

\begin{proof}
(i) Let $F$ be such a subset of $E^*$. We will show that for all $\alpha \in \mathbb{R}$, 
\begin{equation} \label{eq:tri1}
F \triangleq \widehat{R}_\alpha(F). 
\end{equation}
To see why this is sufficient to prove (i), note that \eqref{eq:tri1} holds if any only if for all $\alpha$, $U_\gamma(F^+) \triangleq R_\alpha \circ U_\gamma(F^+)$ and $U_\gamma(F^-) \triangleq R_\alpha \circ U_\gamma(F^-)$, i.e. $U_\gamma(F^+)$ and $U_\gamma(F^-)$ are both rotationally invariant as subsets of $\mathbb{R}^2$, up to null sets. Equivalently, if $\zeta(x) = |x|^2$ is the square of the Euclidean norm on $\mathbb{R}^2$, then there exists subsets $K_1, K_2 \subset [0,1)$ such that $U_\gamma(F^+) \triangleq \zeta^{-1}(K_1)$ and $U_\gamma(F^-) \triangleq \zeta^{-1}(K_2)$. Letting $K = K_1 \cup -K_2$ gives us the result.

To prove \eqref{eq:tri1}, first note that by Corollary \ref{cor:erg}, $T_0(F)$ and $T_1(T_0(F))$ are $(m^2,Q)$-ergodic. We will show 

\vspace{0.1in}

\textit{Claim.} $T_1 \circ T_0 = \widehat{R}_{2\gamma}$.

\vspace{0.1in}
To see this, first note that by \eqref{eq:T0def},
\[
T_0(u,v,s) = (\widetilde{T}_0(u,v),-s) \quad \text{ where } \quad \widetilde{T}_0(u,v) = (-u - 2v\cos\gamma, v).
\]
Therefore, by \eqref{eq:Udef},
\begin{equation} \label{eq:tri2}
\begin{split}
U_\gamma \circ \widetilde{T}_0 \circ U_\gamma^{-1}(x,y) & = U_\gamma \circ \widetilde{T}_0( x - y\cot\gamma, y\csc\gamma) \\
& = U_\gamma(-x - y\cot\gamma, y\csc\gamma) \\
& = (-x,y).
\end{split}
\end{equation}
Similarly, by \eqref{eq:T1def},
\[
T_1(u,v,s) = (\widetilde{T}_1(u,v),s), \quad \text{ where } \quad \widetilde{T}_1(u,v) = (u, -v - 2u\cos\gamma).
\]
Hence,
\begin{equation} \label{eq:tri3}
\begin{split}
& U_\gamma \circ \widetilde{T}_1 \circ U_\gamma^{-1}(x,y) = U_\gamma \circ \widetilde{T}_1( x - y\cot\gamma, y\csc\gamma) \\
& \quad\quad = U_\gamma(x - y\cot\gamma, -2x\cos\gamma + y(-\csc\gamma + 2\cot\gamma\cos\gamma)) \\
& \quad\quad = \left(x(\sin^2\gamma - \cos^2\gamma) - 2y\cos\gamma\sin\gamma, -2x\cos\gamma\sin\gamma - y(\sin^2\gamma - \cos^2\gamma)\right).
\end{split}
\end{equation}
From \eqref{eq:tri2}, we see that $U_\gamma \circ \widetilde{T}_0 \circ U_\gamma^{-1}$ is a reflection in $\mathbb{R}^2$ through the line spanned by $e_2 = (0,1)$. Likewise, we deduce from \eqref{eq:tri3} that $U_\gamma \circ \widetilde{T}_1 \circ U_\gamma^{-1}$ is a reflection through the line spanned by the vector $(-\sin\gamma,\cos\gamma)$. The composition of these two reflections is a counterclockwise rotation about the origin through the angle $2\gamma$, i.e.
\[
U_\gamma \circ \widetilde{T}_1 \circ \widetilde{T}_0 \circ U_\gamma^{-1} = R_{2\gamma}.
\]
It follows that 
\[
T_1 \circ T_0(u,v,s) = (\widetilde{T}_1 \circ \widetilde{T}_0(u,v), s) = (U_\gamma^{-1} \circ R_{2\gamma} \circ U_\gamma(u,v), s) = \widehat{R}_{2\gamma}(u,v),
\]
which proves the claim.

By the claim, $\widehat{R}_{2\gamma}(F)$ is $(m^2,Q)$-ergodic. Iterating this result, $\widehat{R}_{2k\gamma}(F)$ is $(m^2,Q)$-ergodic for every $k \in \mathbb{Z}_+$. Suppose that for every pair of distinct integers $k_0, k_1$, $m^2(\widehat{R}_{2k_0\gamma}(F) \cap \widehat{R}_{2k_0\gamma}(F)) = 0$. Since $\widehat{R}_{2\gamma} = T_1 \circ T_0$ is a composition of maps which preserve $m_2$, it also preserves $m^2$; hence for any $N \geq 0$, 
\[
m^2\left( \bigcup_{k = 0}^N \widehat{R}_{2k\gamma}(F) \right) = \sum_{k = 0}^N m^2(\widehat{R}_{2k\gamma}(F)) = (N+1)m^2(F) \to \infty \text{ as } N \to \infty.
\]
(here using $m^2(F) > 0$). This contradicts the finiteness of $m^2$. 

Therefore, there exists $k^* \in \mathbb{Z}_+$ such that $m^2(F \cap R_{2k^*\gamma}(F)) > 0$. As $F$ and $R_{2k^*\gamma}(F)$ are $(m^2,Q)$-ergodic, this implies $F \triangleq R_{2k^*\gamma}(F)$. Iterating the application of the map $R_{2k^*\gamma}$ shows that $F \triangleq R_{2mk^*\gamma}(F)$ for all $m \in \mathbb{Z}_+$.

By hypothesis $\gamma/\pi$ is an irrational number. Let $\alpha \in \mathbb{R}$. By well-known properties of irrational rotation, there exists a sequence of integers $p_j \to \infty$ such that $\widehat{R}_{2p_j k^*\gamma} \to \widehat{R}_\alpha$ pointwise. Thus  
\[
\begin{split}
m^2(F \triangle \widehat{R}_{\alpha}(F)) & = \lim_{j \to \infty }m^2(\widehat{R}_{2p_jk^* \gamma}(F) \triangle \widehat{R}_{\alpha}(F)) = 0, 
\end{split}
\]
where $\triangle$ denotes the set difference and the limit follows by bounded convergence. This proves \eqref{eq:tri1}.

(ii) To prove the first statement in (ii) it is sufficient to show that $U_\gamma((F \cup T_0(F))^+)$ and $U_\gamma((F \cup T_0(F))^-)$ are both full-measure subsets of $D^2$. Let $P_0 \subset (-\csc\gamma, \csc\gamma)$ be chosen as in Lemma \ref{lem:invarstruc}, so that 
\begin{equation} \label{eq:tri4}
F \cup T_0(F) \triangleq \pi_0^{-1}(P_0),
\end{equation}
and let 
\[
\widetilde{P}_0 = \{y \in (-1,1) : y \csc\gamma \in P_0\}.
\]
Let $\pi_{0\pm} : E^\pm \to (-\csc\gamma, \csc\gamma)$ denote the restriction of the projection map $\pi_0$ 
to $E^\pm$. A simple calculation shows that 
\[
\pi_{0+} \circ U_\gamma^{-1}(x,y) = y\csc\gamma.
\]
From this and \eqref{eq:tri4}, we obtain
\begin{equation} \label{eq:tri5}
U_\gamma((F \cup T_0(F))^+) \triangleq U_\gamma(\pi_{0+}^{-1}(P_0)) = \{(x,y) \in D^2 : y \in \widetilde{P}_0\}.
\end{equation}
On the other hand, by part (i) of this lemma, $U_\gamma((F \cup T_0(F))^+)$ is rotationally symmetric; in particular,
\begin{equation} \label{eq:tri6}
U_\gamma((F \cup T_0(F))^+) \triangleq \{(x,y) \in D^2 : x^2 + y^2 \in K \cap [0,1)\}.
\end{equation}
The only way that the quantities \eqref{eq:tri5} and \eqref{eq:tri6} can be equal (up to null sets) is if $|U_\gamma((F \cup T_0(F))^+)| \in \{0,|D^2|\}$. A similar argument shows that $|U_\gamma((F \cup T_0(F))^-)| \in \{0,|D^2|\}$. Since $m^2(F) > 0$ by assumption, either $|U_\gamma((F \cup T_0(F))^+)| = |D^2|$ or $|U_\gamma((F \cup T_0(F))^-)| = |D^2|$. It follows from \eqref{eq:tri4} that
\[
(F \cup T_0(F))^+ = (F \cup T_0(F))^-,
\]
from which we conclude that 
\[
|U_\gamma((F \cup T_0(F))^+)| = |U_\gamma((F \cup T_0(F))^-)| = |D^2|.
\]
It follows that $m^2(F \cup T_0(F)) = m^2(E^*)$, as desired.

To prove the second statement in (ii), note that if $m^2(F \cap T_0(F)) > 0$, then by ergodicity of $F$ and $T_0(F)$, $F \triangleq T_0(F) \triangleq E^*$, contradicting the assumption that $m^2(F) < 1$.

(iii) The map $T_0$ preserves the measure $m^2$. Therefore (iii) is an immediate consequence of (ii).
\end{proof}

\subsubsection{Proof of Theorem \ref{thm:mainresult}, forward implication}

\begin{proof}[Proof of Theorem \ref{thm:mainresult}, forward implication] Suppose $F \subset E^*$ is $(m^2,Q)$-ergodic and $m^2(F) > 0$. Assume for a contradiction that $m^2(F) < 1$. By Lemma \ref{lem:slices} and reversibility of $q$ with respect to $m^1$, for a.e. $v \in (-\csc\gamma, \csc\gamma)$, 
\begin{equation} \label{eq:main1}
\begin{split}
m^1(-[F^-]_v) & = \int_{-[F^-]_v} q(x,[F^+]_v) m^1(dx) \\
& = \int_{[F^+]_v} q(x,-[F^-]_v) m^1(dx) = m^1([F^+]_v).
\end{split}
\end{equation}
Furthermore, by Lemma \ref{lem:tri}(ii), Lemma \ref{lem:m1m2qQ}(i), and \eqref{eq:bracs}, 
\[
\begin{split}
0 & = m^2((F \cap T_0(F))^+) \\
& = \frac{1}{2|E|}\int_{-\csc\gamma}^{\csc\gamma} m^1([(F \cap T_0(F))^+]_v)|E_v| \ dv \\
& = \frac{1}{2|E|}\int_{-\csc\gamma}^{\csc\gamma} m^1([F^+]_v \cap -[F^-]_v)|E_v| \ dv.
\end{split}
\]
It follows that for a.e. $v \in (-\csc\gamma, \csc\gamma)$,
\begin{equation} \label{eq:main2}
m^1([F^+]_v \cap -[F^-]_v) = 0.
\end{equation}
Using the same lemma, 
\[
1 = m^2((F \cup T_0(F))^+) = \frac{1}{2|E|}\int_{-\csc\gamma}^{\csc\gamma} m^1([F^+]_v \cup -[F^-]_v)|E_v| \ dv,
\]
from which we conclude that for a.e. $v \in (-\csc\gamma, \csc\gamma)$,
\begin{equation} \label{eq:main3}
m^1([F^+]_v \cup -[F^-]_v) = 1.
\end{equation}
Equations \eqref{eq:main1}, \eqref{eq:main2}, and \eqref{eq:main3} imply that for a.e. $v \in (-\csc\gamma, \csc\gamma)$,
\begin{equation} \label{eq:main4}
m^1([F^+]_v) = m^1([F^-]_v) = \frac{1}{2}.
\end{equation}
Let $K \subset (-1,1)$ be chosen as in Lemma \ref{lem:tri}(i), so that $F = V^{-1}(K)$. Let $K^+ = K \cap [0,1)$. For $v \in (-\csc\gamma,\csc\gamma)$, let $V_v : (-1,1) \to [0,1)$ be defined by 
\[
V_v(x) = V(\ell_v(x),v,1).
\]
Explicitly,
\begin{equation} \label{eq:main5}
V_v(x) = |U_\gamma(\ell_v(x),v)|^2 = x^2(1 - v^2\sin^2\gamma) + v^2\sin^2\gamma,
\end{equation}
here using the formulas \eqref{eq:Udef} and \eqref{eq:Sform} for $U_\gamma$ and $\ell_v$, respectively. It is easy to check that 
\begin{equation} \label{eq:main6}
[F^+]_v = V_v^{-1}(K^+).
\end{equation}
From \eqref{eq:main5} we see that $V_v$ restricts to bijections from $(-1,0]$ onto $[v^2\sin^2\gamma,1)$ and from $[0,1)$ onto $[v^2\sin^2\gamma,1)$. By \eqref{eq:main4},\eqref{eq:main6}, and symmetry, for almost every $v \in (-\csc\gamma,\csc\gamma)$,
\begin{equation} \label{eq:main7}
\begin{split}
\frac{1}{2} & = m^1([F^+]_v) = m^1(V_v^{-1}(K^+)) = \frac{1}{2}m^1(V_v^{-1}(K^+) \cap [0,1)) \\
& = \frac{1}{2}\int_{V_v^{-1}(K^+) \cap [0,1)} dx = \int_{\widetilde{K}^+ \cap [0,\sqrt{1 - v^2\sin^2\gamma})} \frac{z dz}{(1 - v^2\sin^2\gamma)^{1/2}(1 - v^2 \sin^2\gamma - z^2)^{1/2}},
\end{split}
\end{equation}
where $z = \sqrt{1 - V_v(x)}$ and $\widetilde{K}^+ = \{\sqrt{1 - x^2} : x \in K^+\}$. Setting $\rho = \sqrt{1 - v^2\sin^2\gamma}$ and $t = z^2$, we obtain that the quantity \eqref{eq:main7} is equal to 
\begin{equation}
\int_0^{\rho} \mathbf{1}_{\widetilde{K}_+}(z) \frac{zdz}{\rho \sqrt{\rho^2 - z^2}} = \int_0^{\rho^2} \mathbf{1}_{\widetilde{K}_+}(t^{1/2}) \frac{dt}{2\rho\sqrt{\rho^2 - t}} = \frac{1}{2\rho}\mathcal{A}[\varphi_K](\rho^2),
\end{equation}
where $\mathcal{A}$ is the Abel transform, defined in \S\ref{ssec:abel}, and $\varphi_K(t) := \mathbf{1}_{\widetilde{K}_+}(t^{1/2})$. Since we have shown the quantity above is equal to $1/2$, it follows by setting $\xi = \rho^2$ that 
\begin{equation}
\mathcal{A}[\varphi_K](\xi) = \sqrt{\xi}.
\end{equation}
The equality above holds for almost every $\xi$ in $(0,1]$, since $v$ is allowed to range over almost every value in $(-\csc\gamma,\csc\gamma)$. But one may also verify that $\mathcal{A}[1/2](\xi) = \sqrt{\xi}$. By uniqueness of solutions to the Abel integral equation \eqref{eq:Abel-equation} in $L^1(0,1)$, this implies that $\varphi_K = 1/2$ a.e., giving us the desired contradiction. 
\end{proof}

\subsubsection{Proof of Theorem \ref{thm:mainresult}, backward implication}

\begin{proof}[Proof of Theorem \ref{thm:mainresult}, backward implication] Assume that (ii) holds, $m^2$ is an ergodic measure for $Q$, and $\hat{q}^\dag = \hat{q}$. Let $J \subset (-1,1)$ be some $(m^1,q)$-invariant set with $m^1(J) > 0$. It is enough to show that $m^1(J) = 1$. Define $F \subset E^*$ by 
$$
F = F^+ \cup F^- \text{ where } F^+ = \{(u,v,1) : \ell_u^{-1}(v) \in -J\}, \text{ } F^- = \{(u,v,-1) : \ell_u^{-1}(v) \in J\}.
$$
We will show that $F$ is $(m^2, Q)$-invariant. By Lemma \ref{lem:m1m2qQ}(ii), we have 
\begin{align*}
&\int_{F^+} Q(u,v,s; F^-)m^2(du' dv' ds')  \\
& = \frac{1}{|E|} \int_{-\csc\gamma}^{\csc\gamma} \int_{(-1,1)} \left( \int_{(-1,1)} \mathbf{1}_{F^-}(u,\ell_u(x'),-1)\hat{q}(x,dx') \right) \\
& \hspace{2.5in}\times \mathbf{1}_{F^+}(u,\ell_u(x),1)m^1(dx)|E_u| du. 
\end{align*}
Noting that $\mathbf{1}_{F^-}(u,\ell_u(x'),-1) = \mathbf{1}_{J}(x')$ and $\mathbf{1}_{F^+}(u,\ell_u(x),1) = \mathbf{1}_{-J}(x)$, the quantity above is equal to 
\begin{align*}
&\frac{1}{|E|} \int_{-\csc\gamma}^{\csc\gamma} \int_{-J} \hat{q}(x,J) m^1(dx) |E_u| du \\
& = \frac{1}{|E|} \int_{-\csc\gamma}^{\csc\gamma} \int_{J} q(x,J) m^1(dx) |E_u| du \\
& = \frac{1}{|E|} \int_{-\csc\gamma}^{\csc\gamma} m^1(J) |E_u| du,
\end{align*}
where the last line follows since $J$ is an $(m^1, q)$-invariant set. On the other hand, noting that $m^1(J) = m^1(-J)$, the last line is equal to 
\begin{align*}
&\frac{1}{|E|} \int_{-\csc\gamma}^{\csc\gamma} \int_{(-1,1)} \mathbf{1}_{-J}(x)m^1(dx) |E_u| du \\
& = \frac{1}{|E|} \int_{-\csc\gamma}^{\csc\gamma} \int_{(-1,1)} \mathbf{1}_{F^+}(u,\ell_u(x),1)(x)m^1(dx) |E_u| du \\
& = m^2(F^+),
\end{align*}
by Lemma \eqref{lem:m1m2qQ}(i). We have shown that
\begin{equation} \label{eq:maxone}
\int_{F^+} Q(u,v,s; F^-)m^2(du' dv' ds') = m^2(F^+), 
\end{equation}
and this implies that $Q(u,v,s; F^-) = 1$ for $m^2$-a.e. $(u,v,s) \in F^+$. Note that $m^2(F^-) = m^2(F^+)$ (this can be easily checked, using Lemma \ref{lem:m1m2qQ}(i) again). This, together with Proposition \ref{prop:ir} and \eqref{eq:maxone}, gives us
$$
\int_{F^-} Q^\dag(u,v,s; F^+)m^2(du' dv' ds') = m^2(F^-).
$$
Hence $Q^\dag(u,v,s; F^+) = 1$ for $m^2$-a.e. $(u,v,s) \in F^-$. Because by hypothesis $\hat{q}^\dag = \hat{q}$, we have $Q^\dag = Q$. We conclude that $F$ is $(m^2,Q)$-invariant. By ergodicity of $m^2$, this implies that $1 = m^2(F)$, and thus $m^2(F^+) = m^2(F^-) = 1/2$. Therefore, applying Lemma \ref{lem:m1m2qQ}(i) in a similar way as before, 
$$
\frac{1}{2} = m^2(F^+) = \frac{1}{|E|} \int_{-\csc\gamma}^{\csc\gamma} \int_{(-1,1)} m^1(J) |E_u| du = \frac{m^1(J)}{2};
$$
hence $m^1(J) = 1$, as desired.
\end{proof}

\subsection{Proof of Theorem \ref{thm:qNS}} \label{ssec:main2proof}

We begin with a general, measure theoretic lemma.   

\begin{lemma} \label{lem:tunnel}
Let $p$ be a Markov kernel on the measurable space $(S,\mathcal{S})$, and suppose $\mu$ is an invariant probability measure for $p$ on $(S,\mathcal{S})$, with the property that, for any $(\mu,p)$-ergodic subset $G \subset S$, $\mu(G) = 0$. Then for every $k \geq 1$, there exists a partition of $S$ into $(\mu,p)$-invariant subsets $\{F_1, F_2, \dots, F_{2^{2k}}\}$ such that $0 < \mu(F_i) \leq 2^{-k}$ for $1 \leq i \leq 2^{2k}$.
\end{lemma}

\begin{proof}
The proof proceeds in two steps.

\textit{Step 1. We show that for any $(\mu,p)$-invariant subset $F \subset S$, there exists a $(\mu,p)$-invariant subset $F' \subset F$ such that $4^{-1}\mu(F) \leq \mu(F') \leq 2^{-1}\mu(F)$.}

Indeed, this result is obvious if $\mu(F) = 0$, so assume this is not the case. Define 
$\mathcal{T}(F) = \{F' \subset F : F' \text{ is $(\mu,p)$-invariant, and } 0 < \mu(F') \leq 2^{-1}\mu(F)\}$. Since $\mu(F) > 0$, $F$ cannot be $(\mu,p)$-ergodic, and therefore there exists an invariant $F' \subset F$ with $0 < \mu(F') < \mu(F)$. By replacing $F'$ with $F \smallsetminus F'$ if necessary, we will have $\mu(F') \leq 2^{-1}\mu(F)$. Thus, $\mathcal{T}(F)$ is nonempty, the following number 
\begin{equation}
t^* := \sup\{\mu(F') : F' \in \mathcal{T}(F)\}
\end{equation}
exists, and $0 < t^* \leq 2^{-1}\mu(F)$. Suppose for a contradiction that $t^* < 4^{-1}\mu(F)$. For each $j \geq 1$, choose $F_j \in \mathcal{T}(F)$ such that 
\begin{equation}
t^* - j^{-1} < \mu(F_j) \leq t^*.
\end{equation}
Let $\widehat{F} = \bigcup_{j = 1}^\infty F_j$, and consider the following two cases:

\textit{Case 1: $\mu(\widehat{F}) > t^*$.} Then 
\begin{equation}
k^* := \min\left\{k \geq 1 : \mu\left(\bigcup_{j = 1}^k F_j\right) > t^* \right\}
\end{equation}
exists and is finite. Let $F' = \bigcup_{j = 1}^{k^*} F_j$. Then $F' \subset F$, and $F'$ is $(\mu,p)$-invariant, and we have 
\begin{equation}
0 < \mu(F') \leq \mu\left(\bigcup_{j = 1}^{k^*-1} F_j\right) + \mu(F_{k^*}) \leq t^* + t^* \leq 2^{-1}\mu(F).
\end{equation}
Hence $F' \in \mathcal{T}(F)$. But $\mu(F') > t^*$, a contradiction.

\textit{Case 2: $\mu(\widehat{F}) \leq t^*$.} This is in fact an equality, because 
\begin{equation}
\mu(\widehat{F}) \geq \mu(F_j) \to t^* \text{ as } j \to \infty.
\end{equation}
Since $\widehat{F}$ is $(\mu,p)$-invariant and $\mu(\widehat{F}) = t^* \leq 2^{-1}\mu(F)$, the set $F \smallsetminus \widehat{F}$ is invariant, and $\mu(F \smallsetminus \widehat{F}) \geq 2^{-1}\mu(F) > 0$. Since $F \smallsetminus \widehat{F}$ cannot be $(\mu,p)$-ergodic, there exists an invariant $F'' \subset F \smallsetminus \widehat{F}$ with $0 < \mu(F'') \leq 2^{-1} \mu(F \smallsetminus \widehat{F}) \leq 2^{-1}\mu(F)$. Hence, $F'' \in \mathcal{T}(F)$. Define $F' = \widehat{F} \cup F''$. Then $F'$ is invariant, and since $\widehat{F}, F'' \in \mathcal{T}(F)$, 
\[
0 < \mu(F') = \mu(\widehat{F}) + \mu(F'') \leq 2t^* \leq 2^{-1}\mu(F).
\]
Therefore $F' \in \mathcal{T}(F)$. But $\mu(F') = t^* + \mu(F'') > t^*$, a contradiction.

This shows that $t^* \geq 4^{-1}\mu(F)$, and we are finished with Step 1.

\textit{Step 2.} \textit{We will show that if $F \subset S$ is $(\mu,p)$-invariant and $\mu(F) > 0$, then there exists a partition $\{F_1, F_2, F_3, F_4\}$ of $F$ such that each $F_j$ is $(\mu,p)$-invariant and $0 < \mu(F_j) < 2^{-1}\mu(F)$.} The Lemma will follow by iterating Step 2.

Indeed, define inductively pairs of sets $\{(F_k^1, F_k^2)\}_{k \geq 1}$ as follows: $F_1^1$ is an invariant subset of $F$ such that $4^{-1}\mu(F) \leq \mu(F_1^1) \leq 2^{-1}\mu(F)$, and $F_1^2 = F \smallsetminus F_1^1$. For $k > 1$, define $F_k^1$ to be an invariant subset of $F_{k-1}^2$ such that $4^{-1}\mu(F_{k-1}^2) < \mu(F_k^1) \leq 2^{-1}\mu(F_{k-1}^2)$, and $F_k^2 = F_{k-1}^2 \smallsetminus F_k^1$. Clearly, for each $k \geq 1$, the collection 
\[
\{F_1^1, F_2^1, \dots, F_k^1\} \cup \{F_k^2\}
\]
is a partition of $F$ into invariant subsets. We have $\mu(F_1^1) \leq 2^{-1}\mu(F)$, and for each $j \geq 2$, $\mu(F_j^1) \leq 2^{-1}\mu(F_{j-1}^2) \leq 2^{-1}\mu(F)$. Moreover, $\mu(F_k^2) \leq (3/4)\mu(F_{k-1}^2)$,
and by induction $\mu(F_k^2) \leq (3/4)^k \mu(F)$. By taking $k = 3$, we obtain that $\mu(F_k^2) \leq (27/64) \mu(F)$, and so $\{F_1^1,F_2^1,F_3^1, F_3^2\}$ is a partition of $F$ satisfying the desired properties.
\end{proof}

\begin{proof}[Proof of Theorem \ref{thm:qNS}] (i) We may assume without loss of generality that $q(x,dx')$ is nonsingular with respect to $m(dx')$ for \textit{every} $x \in (-1,1)$. To see this, observe that if $\Theta$ is the null set of $x \in (-1,1)$ such that $q(x,\cdot) \perp m^1$, then redefining $q(x,dx') = m(dx')$ at points $x$ in $\Theta$ does not change the hypothesis of part (i) and preserves the set of ergodic measures for $Q$ which are absolutely continuous with respect to $m^2$. This follows by observing that, by \eqref{eq:Qdef}, the modification only affects $Q$ on $m^2$-null set; hence, if $\mu$ is any probability measure on $E^*$ with $\mu \ll m^2$, the change will preserve the collection of $(\mu,Q)$-invariant subsets of $E^*$.

Define the Markov kernel $Q^2(u,v,s; du' dv' ds')$ on $E^*$ by 
$$
Q^2(u,v,s; B) = \int_E Q(u',v',s'; A)Q(u,v,s; du' dv' ds'), \quad\quad B \in \mathcal{B}(E^*).
$$
The rest of the proof of (i) proceeds in two steps.

\textit{Step 1. We will show that for all $(u,v,s) \in E^*$, $Q^2(u,v,s; du' dv' ds')$ is nonsingular with respect to $m^2$.} 

Indeed, let $\Xi$ be any $m^2$-full measure subset of $E^*$. It is sufficient to show that $Q(u,v,s; \Xi) > 0$ for all $(u,v,s) \in E^*$.  By Lemma \ref{lem:m1m2qQ},
\[
\begin{split}
1 = m^2(\Xi) & = \frac{1}{2|E|}\int_{-\csc\gamma}^{\csc\gamma} \int_{(-1,1)} \left(m^1([\Xi^+]_v) + m^1([\Xi^-]_v)\right)|E_v|dv \\
& = \frac{1}{2|E|}\int_{-\csc\gamma}^{\csc\gamma} \int_{(-1,1)} \left(m^1([\Xi^+]^u) + m^1([\Xi^-]^u)\right)|E_u|du,
\end{split}
\]
from which it follows that, for almost every $(u,v) \in E$, $[\Xi^+]_v$, $[\Xi^-]_v$, $[\Xi^+]^u$, and $[\Xi^-]^u$ are full-measure subsets of $(-1,1)$. On the other hand, unwinding definitions, for any $(u',v',s') \in E^*$, 
\begin{equation} \label{eq:Qasq}
\begin{split}
Q(u',v',s'; \Xi) & = \begin{cases}
\int_{(-1,1)} \mathbf{1}_{\Xi}(u',v',1)(S_{v'})_*\hat{q}(u',du'') & \text{ if } s = -1, \\
\int_{(-1,1)} \mathbf{1}_{\Xi}(u',v',-1)(S_{u'})_*\hat{q}(v',dv'') & \text{ if } s = 1
\end{cases} \\
& = \begin{cases}
\hat{q}(S_{v'}^{-1}(u'); [\Xi^+]_{v'}) & \text{ if } s = -1, \\
\hat{q}(S_{u'}^{-1}(v'); [\Xi^-]^{u'}) & \text{ if } s = 1.
\end{cases}
\end{split}
\end{equation}
Hence, for all $(u,v,s) \in E^*$,
\[
\begin{split}
Q^2(u,v,s; \Xi) & = \begin{cases}
\int_{(-1,1)} Q(u',v,1; \Xi)(\ell_v)_*\hat{q}(u,du') & \text{ if } s = -1, \\
\int_{(-1,1)} Q(u,v',1; \Xi)(\ell_u)_*\hat{q}(v,dv') & \text{ if } s = 1
\end{cases} \\
& = \begin{cases}
\int_{(-1,1)} Q(\ell_v(x'),v,1;\Xi)\hat{q}(\ell_v^{-1}(u), dx') & \text{ if } s = -1, \\
\int_{(-1,1)} Q(u,\ell_u(x'),-1;\Xi)\hat{q}(\ell_u^{-1}(v), dx') & \text{ if } s = 1
\end{cases} \\
& = \begin{cases}
\int_{(-1,1)} \hat{q}(S_{\ell_v(x')}^{-1}(v), [\Xi^-]^{\ell_v(x')})\hat{q}(\ell_v^{-1}(u), dx') & \text{ if } s = -1, \\
\int_{(-1,1)} \hat{q}(S_{\ell_u(x')}^{-1}(u), [\Xi^+]_{\ell_u(x')})\hat{q}(\ell_u^{-1}(v), dx') & \text{ if } s = 1.
\end{cases}
\end{split}
\]
where we substitute in \eqref{eq:Qasq} to obtain the last equality. Since $[\Xi^-]^{\ell_v(x')}$ and $[\Xi^+]_{\ell_u(x')}$ are full-measure subsets of $(-1,1)$ and $\hat{q}(x,dx')$ is nonsingular for all $x \in (-1,1)$, the last quantity is strictly positive.

\textit{Step 2.} To complete the proof, it is enough to show that there exists some $(m^2,Q)$-ergodic set $F \subset E^*$ such that $m^2(F) > 0$. For then the probability measure $\tilde{m}^2(dx) := \mathbf{1}_{F}(x)m^2(dx)/m^2(F)$ is an ergodic measure for $Q$.

To this end, suppose for a contradiction that every $(m^2,Q)$-ergodic subset of $E^*$ is $m^2$-null. By the Radon-Nikodym theorem, for each $w = (u,v,s) \in E^*$, we may write 
$$Q(w, dw') = f_w(w') m^2(dw') + \nu_w(d w'),$$
where $f_w \in L^1(m^2)$ and $\nu_w$ is singular with respect to $m^2$. Let $P_w = \{w' \in E^* : f_w(w') > 0\}$. By Step 1, $m^2(P_w) > 0$. Furthermore, if $F$ is a $(m^2,Q)$-invariant subset of $E^*$, then for $m^2$-a.e. $w \in F$, $Q(w,P_w\smallsetminus F) = 0$. Hence, for $m^2$-a.e. $w \in F$, 
$$
\int_{E^*} \mathbf{1}_{P_w \smallsetminus F}(w')f_w(w')m^2(dw') = 0,
$$
from which we conclude $m^2(P_w \smallsetminus F) = 0$.

By Lemma \ref{lem:tunnel}, we can partition $E^*$ into $(m^2,Q)$-invariant subsets $\{F_1, F_2, \dots, F_{2^{2k}}\}$ such that $0 < m^2(F_j) \leq 2^{-k}$ for $1 \leq j \leq 2^{2k}$. Then for $m^2$-a.e. $w \in F_j$, $Q(w, P_w \smallsetminus F_j) = 0$. This implies that for almost every $w \in F_j$, $m^2(P_w) \leq m^2(F_j) \leq 2^{-k}$. Since every point in $E^*$ lies in some $F_j$, $m^2(P_w) \leq 2^{-k}$ for a.e. $w \in E^*$. Since $k$ is an arbitrary positive integer, we conclude that $m^2(P_w) = 0$ for $m^2$-almost every $w \in E^*$, a contradiction, proving (i).

\textit{Proof of (ii).} Note that the ``WLOG'' of the proof of part (i) is satisfied in the hypothesis of part (ii). Thus, Step 1 in the proof of (i) still goes through, and $Q^2(w,\cdot)$ is nonsingular with respect to $m^2$ for every $w \in E^*$. Suppose that $\mu$ is an invariant measure for $Q$, and suppose $A \subset E^*$ and $\mu(A) = 0$. It is enough to show that $m^2(E^* \smallsetminus A) > 0$. By invariance,
$
\int_{E^* \times E^*} Q^2(w,A) \mu(dw) = \mu(A) = 0.
$
Thus, for $\mu$-a.e. $w \in E^*$, $Q^2(w, A) = 0$. In particular, since $\mu$ is nonzero, at least one $w$ exists such that $Q^2(w,A) = 0$, and since $Q^2(w,\cdot)$ is nonsingular with respect to $m^2$, this implies $m^2(E^* \smallsetminus A) > 0$, as desired.
\end{proof}

\bibliographystyle{plain}
\bibliography{ergodic_ref}

\end{document}